\documentclass[11pt, leqno]{article}

\usepackage[dvips]{graphicx}
\usepackage{graphicx}
\usepackage{amssymb}
\usepackage{amsmath}
\usepackage{color}
\usepackage{times}

\usepackage[unicode,bookmarks,colorlinks]{hyperref}
\hypersetup{
    linkcolor=brickred,
}

\definecolor{mahogany}{cmyk}{0, 0.77, 0.87, 0}
\definecolor{salmon}{cmyk}{0, 0.53, 0.38, 0}
\definecolor{melon}{cmyk}{0, 0.46, 0.50, 0}
\definecolor{yellowgreen}{cmyk}{0.44, 0, 0.74, 0}
\definecolor{brickred}{cmyk}{0, 0.89, 0.94, 0.28}
\definecolor{OliveGreen}{cmyk}{0.64, 0, 0.95, 0.40}
\definecolor{RawSienna}{cmyk}{0, 0.72, 1.0, 0.45}
\definecolor{ZurichRed}{rgb}{1, 0, 0} 

\usepackage{amsmath,amstext,amssymb,amsopn,amsthm}
\usepackage{amsmath,amssymb,amsthm}
\usepackage[mathscr]{eucal}

\newtheorem{conjecture}{Conjecture}
\newtheorem{definition}{Definition}
\newtheorem{problem}{Problem}
\newtheorem{question}{Question}

\begin{document}

\newtheorem{lemma}[thm]{Lemma}
\newtheorem{remark}{Remark}
\newtheorem{proposition}{Proposition}
\newtheorem{theorem}{Theorem}[section]
\newtheorem{deff}[thm]{Definition}
\newtheorem{case}[thm]{Case}
\newtheorem{prop}[thm]{Proposition}
\newtheorem{example}{Example}

\newtheorem{corollary}{Corollary}

\numberwithin{equation}{subsection}
\numberwithin{definition}{subsection}
\numberwithin{corollary}{subsection}

\numberwithin{theorem}{subsection}

\numberwithin{remark}{subsection}
\numberwithin{example}{subsection}
\numberwithin{proposition}{subsection}

\newcommand{\gap}{\lambda_{2,D}^V-\lambda_{1,D}^V}
\newcommand{\gapR}{\lambda_{2,R}-\lambda_{1,R}}
\newcommand{\bD}{\mathrm{I\! D\!}}
\newcommand{\calD}{\mathcal{D}}
\newcommand{\calA}{\mathcal{A}}

\newcommand{\conjugate}[1]{\overline{#1}}
\newcommand{\abs}[1]{\left| #1 \right|}
\newcommand{\cl}[1]{\overline{#1}}
\newcommand{\expr}[1]{\left( #1 \right)}
\newcommand{\set}[1]{\left\{ #1 \right\}}

\newcommand{\calC}{\mathcal{C}}
\newcommand{\calE}{\mathcal{E}}
\newcommand{\calF}{\mathcal{F}}
\newcommand{\Rd}{\mathbb{R}^d}
\newcommand{\BR}{\mathcal{B}(\Rd)}
\newcommand{\R}{\mathbb{R}}
\newcommand{\al}{\alpha}
\newcommand{\RR}[1]{\mathbb{#1}}
\newcommand{\bR}{\mathrm{I\! R\!}}
\newcommand{\ga}{\gamma}
\newcommand{\om}{\omega}
\newcommand{\A}{\mathbb{A}}
\newcommand{\bH}{\mathbb{H}}
\newcommand{\B}{\mathbb{B}}

\newcommand{\bb}[1]{\mathbb{#1}}
\newcommand{\bI}{\bb{I}}
\newcommand{\bN}{\bb{N}}

\newcommand{\uS}{\mathbb{S}}
\newcommand{\M}{{\mathcal{M}}}
\newcommand{\calB}{{\mathcal{B}}}

\newcommand{\W}{{\mathcal{W}}}

\newcommand{\m}{{\mathcal{m}}}

\newcommand {\mac}[1] { \mathbb{#1} }

\newcommand{\bC}{\Bbb C}


\begin{titlepage}
\title{The foundational inequalities of D.L. Burkholder and some of their ramifications\thanks{To appear, Illinois Journal of Mathematics, Volume in honor of D.L. Burkholder.}}
\author{Rodrigo Ba\~nuelos\thanks{Supported in part by NSF Grant \#0603701-DMS}
\\Department of  Mathematics
\\Purdue University\\West Lafayette, IN
47906\\banuelos@math.purdue.edu
}
\maketitle
\noindent 
\begin{center}{\it To Don Burkholder, with the greatest respect and admiration\\ for his personal kindness and his mathematical accomplishments.}
\end{center}

\begin{abstract}  
This paper present an overview of some of the applications of the
martingale inequalities of D.L.~Burkholder to $L^p$-bounds
for singular integral operators, concentrating on the
Hilbert transform, first and second order Riesz transforms, the Beurling-Ahlfors operator and other multipliers
obtained by projections (conditional expectations) of transformations
of stochastic integrals.  While martingale inequalities can be used to prove the boundedness of a wider class of Calder\'on-Zygmund singular integrals, the aim of this paper is to show results which give optimal or near optimal bounds in the norms, hence our restriction to the above operators. 

Connections of Burkholder's foundational work on sharp martingale inequalities to other areas of mathematics where either the results themselves or techniques to prove them have become of considerable interest in recent years, are discussed. These include the 1952 conjecture of C.B. Morrey on rank-one convex and quasiconvex functions with connections to problems in the calculus of variations and the 1982 conjecture of T. Iwaniec on the $L^p$-norm of the Beurling-Ahlfors operator with connections to problems in the theory of qasiconformal mappings. Open  questions, problems  and conjectures are listed throughout the paper and copious references are provided.

\end{abstract}
\tableofcontents
\end{titlepage}

\section{Introduction}
In 1966, D.L. Burkholder published a landmark paper titled {\it ``Martingale Transforms"} \cite{Bur11}. Among some of the results  contained in this paper is the now celebrated $L^p$-boundedness of  martingale transforms. In his 1984  paper \cite{Bur39} {\it ``Boundary value problems and sharp inequalities for martingale transforms,"} also a landmark in martingale theory, Burkholder proved sharp versions of the 1966 martingale inequalities.  These results lie at the heart of the applications of martingales   to many areas in probability and analysis where their influence has been deep and lasting.  

The 1966 paper led to the explosion of martingale inequalities that  produced, among many other influential results,  the celebrated ``Burkholder-Davis-Gundy" inequalities which have been indispensable in the development of stochastic  analysis and its applications to so many areas of mathematics. For a historical overview of martingale inequalities beginning with the 1966 paper, see \cite{BanDav1}. The purpose of the present paper is to describe some of the applications of Burkholder's sharp martingale inequalities to singular integrals arising from his 1984 paper and to elaborate on how his method of proof, now commonly referred to it simply  as ``the Burkholder method,"  has led to many other sharp martingale inequalities with  interesting applications. In particular, we describe  applications to a well known conjecture of T. Iwaniec concerning the $L^p$-norm of the Beurling-Ahlfors operator.  We also discuss how Burkholder's work on sharp martingale inequalities has come into play in the investigation of rank-one convex and quasiconvex functions and its relation to a longstanding open problem of Morrey. The connections to the last two problems arise not only from the results proved in the 1984 paper, but from the techniques used in the proofs.

Absent from this paper are the numerous applications of Burkholder's inequalities to singular integrals and other operators for functions taking values in Banach spaces with the {\it unconditional martingale difference sequences} (UMD) property. This has been a  very active area of research with contributions by many mathematicians since the appearance of Burkholder's seminal paper  \cite{Bur35} {\it ``A geometrical characterization of Banach spaces in which martingale difference sequences are unconditional"}. These applications are not discussed here. For some of Burkholder's other contributions to UMD spaces, we refer the reader to \cite{Bur38} and \cite{Bur53}.   We also refer the reader to \cite{Bou2}, \cite{Fig1}, \cite{Hyt00}, \cite{Hyt0}, \cite{Hyt1} and \cite{Hyt2},  which contain many results and further references on singular integrals and other operators with values in Banach spaces with the UMD property and to G. Pisier's overview paper {\it ``Don Burkholder's work on Banach spaces"}  \cite{Pis0}. 

The techniques introduced by Burkholder in \cite{Bur39} were so novel that only in recent years has their  full and  wider impact in areas far removed from their original applications to martingales begun to emerge.  We hope this paper will serve as a starting point for further explorations and applications of sharp martingale inequalities.

\subsection{The foundational inequalities} 

Let  $f=\{f_n, n\geq 0\}$ be a martingale on a probability space $(\Omega, \mathcal{F}, P)$ 
with respect to the  sequence of $\sigma$-fields $\calF_n\subset \calF_{n+1}$, $n\geq 0$, contained in $\calF$. The sequence  
$d=\{d_k, k\geq 0\}$, where $d_k=f_k-f_{k-1}$ for $k\geq 1$ and $d_0=f_0$, is called the martingale difference sequence of $f$. 
Thus $f_n=\sum_{k=0}^n d_k$.  Given a sequence of random variables $\{v_k, k\geq 0\}$ uniformly bounded by $1$ for all $k$ and with $v_k$, $k\geq 1$,  measurable with respect to $\calF_{k-1}$ and $v_0$ constant, (such sequence is said to be predictable), the martingale difference sequence $e=\{v_kd_k, k\geq 0\}$ generates a new martingale called the {\it ``martingale transform"} of $f$ and  denoted here by $v\ast f$.  Thus $(v\ast f)_n=\sum_{k=0}^n v_kd_k$. The maximal function of a martingale is denoted by $f^*=\sup_{n\geq 0}|f_n|$.  We also set $\|f\|_p=\sup_{n\geq 0}\|f_n\|_p$. Burkholder's 1966 result in \cite{Bur11} asserts that the operator $f\to v\ast f$ is bounded on $L^p$, for all $1<p<\infty$, and that it is weak-type $(1,1)$.  More precisely, he proved the following 

\begin{theorem}\label{bur66} Let $f=\{f_n, n\geq 0\}$ be a martingale with difference sequence $d=\{d_k, k\geq 0\}$.  Let $\{v_k, k\geq 0\}$ be a predictable sequence with $|v_k|\leq 1$ a.s. for all $k$.  There is a universal constant $C_1$ and a constant $C_p$  depending only $p$ such that 
\begin{equation}\label{bur1}
\|v\ast f\|_p\leq C_p\|f\|_p, \,\,\, 1<p<\infty,
\end{equation}
and 
\begin{equation}
P\{{(v\ast f)}^*>\lambda\}\leq \frac{C_1}{\lambda}\|f\|_1,\,\,\, \lambda >0.
\end{equation}

\end{theorem}

Let $\{h_k, k\geq 0\}$ be the Haar system in the Lebesgue unit interval $[0, 1)$. That is, $h_0=[0, 1), h_1=[0, \,1/2)-[1/2, \,1),  h_3=[0, \,1/4)-[1/4, \,1/2), h_4=[1/2, \,3/4)-[3/4,\, 1), \dots$, where we use the same notation for an interval as its indicator function. By considering dyadic martingales, inequality (\ref{bur1}) contains the classical inequality of R.E.A.C. Paley \cite{Pal}  which asserts that there is a constant  $C_p$ depending only $p$ such that  for any sequence of real numbers $a_k$,
\begin{equation}\label{Paleyreal}
\Big\|\sum_{k=0}^{n} \varepsilon_k a_k h_k\Big\|_p\leq C_p\Big\|\sum_{k=0}^{n}  a_k h_k\Big\|_p, \, \,\, 1<p<\infty,
\end{equation}
 whenever $\varepsilon_k\in \{1, -1\}$. We should note that Paley's original inequality was given in terms of the Walsh system of functions in the unit interval and that Marcinkiewicz \cite{Mar1} derived the equivalent formulation given here in terms Haar functions.

It is difficult to overstate the importance of Burkholder's 1966 paper \cite{Bur11} in the subsequent developments of martingale theory and its application to so many different areas of mathematics.  For a historical overview of these developments, we refer the reader to \cite{BanDav1} and \cite{Pis0}. For the purpose of this paper we skip directly to another groundbreaking paper of Burkholder \cite{Bur39} where he identified the best constants $C_1$ and $C_p$  in Theorem \ref{bur66}.  First, for $1<p<\infty$ we let $p^*$  denote the maximum of $p$ and $q$, where $\frac{1}{p}+\frac{1}{q}=1$.   Thus $$p^* = \max\{p,
\frac{p}{p-1}\}$$ and 
\begin{equation}\label{p^*}
p^*-1=\begin{cases} p-1 , \hskip3mm  2\leq p <\infty, \\ 
\frac{1}{p-1},  \hskip4mm  1<p\leq 2.
\end{cases}
\end{equation}
This notation will be used throughout the paper.  
\begin{theorem}\label{sharp1} Let $f=\{f_n, n\geq 0\}$ be a martingale with difference sequence $d=\{d_k, k\geq 0\}$.   Let $v\ast f$ be the martingale transform of $f$ by a real predictable sequence $v=\{v_k, k\geq 0\}$ uniformly bounded in absolute value by 1.  Then 
\begin{equation}\label{bur2}
\|v\ast f\|_p\leq (p^*-1)\|f\|_p, \,\,\, 1<p<\infty.
\end{equation}
The constant $(p^*-1)$ is best possible. Furthermore, in the nontrivial case when $0<\|f\|_p<\infty$, equality holds if and only if $p=2$ and $\sum_{k=0}^{\infty} v_k^2 d_k^2=\sum_{k=0}^{\infty} d_k^2$, almost surely.  
\end{theorem}

There are many other sharp martingale transform inequalities proved in \cite{Bur39}. As an illustration, we list the following which even though less relevant to the topic of this paper than the $L^p$-inequalities, still serves as motivation for some of the questions and problems we will raise below. 

\begin{theorem}\label{sharpthm2} Let $1\leq p\leq 2$ and let $f$ and $v$ be as in Theorem \ref{sharp1}.  
Then 
\begin{equation}\label{bur3}
\sup_{\lambda>0}\lambda^p\, P\{(v\ast f)^*>\lambda\}\leq \frac{2}{\Gamma(p+1)}\|f\|_p^p.
\end{equation}
The constant $ \frac{2}{\Gamma(p+1)}$ is best possible.  Furthermore, strict inequality holds if $0<\|f\|_p<\infty$ and $1<p<2$, but equality can hold if $p=1$ or $p=2$. 
\end{theorem}

\begin{remark}\label{realhaar}
An immediate consequence of (\ref{bur2})  is the fact that the constant $C_p$ in Paley's inequality (\ref{Paleyreal}) can be replaced by $(p^*-1)$.  In fact, Burkholder (see \cite[p.~697]{Bur39}) also shows that this is best possible, hence proving that the real unconditional constant of the Haar basis $\{h_k; k\geq 0\}$ of $L_{\bR}^p[0, 1]$,  $1<p<\infty$, is $(p^*-1)$. 
\end{remark}

The proof of Theorem \ref{sharp1} in \cite{Bur39}, which after some preliminary work reduces to the case when the predicable sequence $\{v_k\}\in \{-1, 1\}$, is extremely difficult.  It rests on solving the nonlinear PDE 
\begin{equation}
(p- 1)[yF_y - xF_x]F_{yy} - [(p - 1)F_{y} - xF_{xy}]^2 + x^2F_{xx}F_{yy} = 0
\end{equation}
for $F$ non-constant and 
satisfying some suitable boundary conditions in certain domains of $\bR^2$.  
The solutions to such equation leads to a system of several nonlinear differential inequalities with boundary conditions.  
From this system, a function $u(x, y, t)$ is constructed in the domain 
$$
\Omega = \Big\{(x, y, t)\in \bR^3: \Big|\frac{x-y}{2}\Big|^p <t\Big\}
$$
with certain convexity properties for which, using the techniques of \cite{Bur35}, Burkholder proves that 
\begin{equation}
u(0,0,1)\|g_n\|_p^p\leq \|f_n\|_p^p
\end{equation}
for $1<p\leq 2$.  He then shows that $u(0,0,1)=(p-1)^p$, which gives the bound in Theorem \ref{sharp1} for this range of $p$. The case $2<p<\infty$ follows by duality.  The research announcement  \cite{Bur37} contains a nice summary of the methods used in \cite{Bur39}.
Even today, the proof given in \cite{Bur39} remains quite difficult.   
 
 In a series of papers following \cite{Bur39}, which included many applications to various other sharp inequalities for discrete  martingales and stochastic integrals, Burkholder simplified the proofs in \cite{Bur39} considerably by giving an explicit expression for his ``magical" function $U$.  This simpler proof also led to a more general theorem that has several applications.  In particular, in \cite{Bur45} Burkholder proved the following extension of Theorem \ref{sharp1}.

 \begin{theorem}\label{sharp3} Let $\bH$ be a (real or complex) Hilbert space. For $x\in \bH$, let $|x|$ denote its norm.  Let $f=\{f_n, n\geq 0\}$ and $g=\{g_n, n\geq 0\}$ be two $\bH$-valued martingales on the same filtration with martingale difference sequence $d=\{d_k, k\geq 0\}$ and $e=\{e_k, k\geq 0\}$, respectively, and satisfying the subordination condition  
 \begin{equation}\label{diffsub}
 |e_k|\leq |d_k|,
 \end{equation}
 almost surely for all $k\geq 0$.  Then 
 \begin{equation}\label{bur4}
\|g\|_p\leq (p^*-1)\|f\|_p, \,\,\, 1<p<\infty
\end{equation}
and the constant $(p^*-1)$ is best possible.  Furthermore, in the nontrivial case when $0<\|f\|_p<\infty$, equality holds if and only if $p=2$ and equality holds in (\ref{diffsub}) almost surely for all $k\geq 0$. 
\end{theorem}

\begin{remark}  In \cite{Pel1}, 
A. Pe\l czy\'nski 
conjectured that the complex unconditional constant for the Haar system is the same as the unconditional constant for the real case.  Given Remark \ref{realhaar}, this amounts to proving that 
\begin{equation}\label{Paleycom}
\Big\|\sum_{k=0}^{n} e^{i\theta_k}c_k h_k\Big\|_p\leq (p^*-1)\Big\|\sum_{k=0}^{n}  c_k h_k\Big\|_p, \, \,\, 1<p<\infty,
\end{equation}
for all $c_k\in \bC$ and $\theta_k\in \bR$\,.  But this follows from Theorem \ref{sharp3}. For more on this, see \cite{Bur44}. 
\end{remark}

To prove the inequality (\ref{bur4}), Burkholder considers the function $V:\bH\times\bH\to \bR$\,    defined by 
\begin{equation}\label{v}
V(x, y)=|y|^p-(p^*-1)^p|x|^p.
\end{equation}
The goal is then to show that $EV(f_n, g_n)\leq 0$.  Burkholder then introduces his now famous function 
\begin{equation}\label{u}
U(x, y)=\alpha_p\left(|y|-(p^*-1)|x|\right)\left(|x|+|y|\right)^{p-1}, 
\end{equation}
where 
$$
\alpha_p=p\left(1-\frac{1}{p^*}\right)^{p-1}
$$
and proves that this function satisfies the following properties: 
\begin{eqnarray}\label{sub1}
V(x, y)&\leq &U(x, y)\,\,\,  \text{for all}\,\,\, x, y\in \bH,\label{sub2}\\
EU(f_n, g_n) &\leq& EU(f_{n-1}, g_{n-1}), \,\,\, n\geq 1,\label{sub3}\\
EU(f_0, g_0)&\leq& 0.
\end{eqnarray}

 A nice explanation of Burkholder's PDE and other ideas  in \cite{Bur39} in terms of the theory of Bellman functions was subsequently given by F. Nazarov, S. Treil and A. Volberg.  For this connection and some of the now very extensive literature on this subject, we refer the reader to  \cite{NazTre1, NazTreVol1, NazTreVol2, Vol1}.  
 Quoting from \cite{NazTreVol1}:
  ``It is really amazing that Burkholder was able to solve 
  these PDEÕs: they are really complicated."   
 In the Bellman function language of Nazarov and Volberg, Burkholder gives an explicit expression for the ``true" Bellman function of the above PDE.  Quoting from \cite{NazTre1}, ``the most amazing thing is that the true Bellman function is known! This fantastic achievement belongs to Burkholder." 
 
  Explicit solutions to Bellman problems that arise in many of the applications to harmonic analysis are  very challenging problems.  For more on this, see \cite{Mel1}, \cite{Mel2}, \cite{VasVol1}, and specially the recent papers \cite{VasVol0}, \cite{VasVol2} which contain a treatment, based on the Monge-Amp\`ere equation, on how to solve many Bellman equations, including Burkholder's. While Bellman functions had been  used in the area of control theory for many years, Burkholder pioneering paper \cite{Bur39} was the first to use Bellman functions in problems related to harmonic analysis.

In addition to the results we will discuss in \S2 below, there are many other extensions and refinements of Theorem \ref{sharp1} in the literature now.   Many of these results are due to Burkholder himself; see for example his work in \cite{Bur40, Bur42, Bur43, Bur44, Bur46, Bur47, Bur48,
Bur50, Bur51, Bur52, Bur54}. Some other applications (including many
recent ones) are contained in \cite{Cho3, Cho4, Cho2, Cho1, Ham2, Ham1, Osc6, Osc7, Osc1, Osc2, Osc3, Osc4, Osc5, Wan2, Wan3, Wan4}.   The interested reader is further encouraged to consult many of the other papers of A. Os\c{e}kowski on sharp martingale inequalities not listed here.  Os\c{e}kowski's work further elucidates ``the Burkholder method" and also removes some the mystery of his ``miracle" functions.

\subsection{Outline of the paper} The outline of this paper is as follows.  In \S2, we present versions of Burkholder's inequalities for continuous-time martingales.  These are used for the applications to the Hilbert transform, orthogonal harmonic functions, Riesz transforms in $\bR^n$ and Wiener space, and the Beurling-Ahlfors operator in the complex plane and $\bR^n$. In addition to martingales satisfying the Burkholder subordination condition, we discuss versions of these inequalities for martingales satisfying an orthogonality condition.  The orthogonality condition was introduced in this context in \cite{BanWan3} to prove the sharp $L^p$-bounds for the Riesz transforms.   The applications of the sharp inequalities to singular integrals are given in \S3. In \S4 we describe a more recent connection between Burkholder's inequalities and a class of Fourier multipliers which we called {\it L\'evy multipliers}.   
These multipliers arise from transformations of the L\'evy symbol of the process via the  L\'evy-Khintchine formula. 

A key property of the Burkholder function $U$ (proved in \cite{Bur45}) is that for all $x,\ y,\ h,\ k\in \bH$ with  $|k|\leq |h|$, the function  
$t\to U (x+th,\ y+tk)$  is concave in $\bR$ , or equivalently the function $t\to -U (x+th,\ y+tk)$ is convex in $\bR$ . The concavity property of $t\to U (x+th,\ y+tk)$ is crucial in the proof of the properties in (\ref{sub1})-(\ref{sub3}).  Properly formulated,  this property means that the function $U$ is rank-one convex which then brings connections to a long standing open problem in the calculus of variations known as Morrey's conjecture.  This connections are discussed in \S5. 

Finally, it should be mentioned here that the fact that many singular integrals can be studied by martingale transform techniques applied to martingales 
arising from composition of harmonic functions with Brownian motion has been  well known  for many years and the literature on this topic is very large indeed.  We refer the reader to, for example, \cite{Bur38},  and especially \cite{McC1} where it is shown that under suitable smoothness conditions the H\"ormander $L^p$-multiplier theorem (\cite[p.~96]{Ste2}) follows from discrete-time martingale transforms, and that  it holds even in the setting of UMD Banach spaces.  But in such general settings, and without an ``almost exact" representation of the operators in terms of stochastic integrals, it is not possible to obtain the type of information 
on the $L^p$-constants we want (need) for some of the applications described in this paper.

\section{Sharp inequalities for continuous-time martingales}

We begin by recalling several inequalities for continuous-time martingales based on variants of Burkholder's differential subordination.  Let $\mathbb{H}$ be a separable (real or complex) Hilbert space with norm $|x|$ and inner product
$x\cdot y$ for vectors $x, y\in \mathbb{H}$.
 For this paper we can (and will) assume that the Hilbert space is just $\ell^2$.  In fact, since all the inequalities derived in this paper hold with universal constants independent of the dimension of the space where the martingales take their values, we can just as well work on either $\bR^n$ or $\bC^n$.  Thus from this point on we will either just work on $\ell^2$ and even at time just specify the martingales by coordinates either real or complex, depending on our needs.  The presentation could be simplified and streamlined somewhat by introducing a more uniform notation.  But since we want to make reference back to papers in the literature where either $\ell^2$, $\bR^d$ or $\bC^d$ are regularly used, we prefer to keep it at this somewhat informal level.
 
We consider then two $\ell^2$-valued semi-martingales $X$ and $Y$ which have right-continuous paths with left-limits (r.c.l.l.).  We denote  their common filtration by
$\mathcal{F=}\left\{  \mathcal{F}_{t}\right\}  _{t\geq0}$ which is a family of right-continuous sub-$\sigma$ fields in
the probability space $\left\{  \Omega, \mathcal{A}, P\right\}$ and for which 
 $\mathcal{F}_{0}$ contains all sets of probability zero. We denote the quadratic covariation 
 process between $X$ and $Y$ by 
$\left[  X,Y\right]  =\left\{ \left[  X,Y\right]_{t}; t\geq 0\right\}$. 
For notational 
simplicity, we use $\left[  X\right]  =\left\{  \left[  X\right]
_{t}; t\geq 0\right\}$ to denote $\left[  X,X\right].$ As in the case of discrete time martingales,  we set $\left|  \left|
X\right|  \right|  _{p}=\sup_{t\geq0}\left|  \left|  X_{t}\right|
\right|  _{p}.$

\subsection{Differential subordination} We say that the martingale $Y$ is differentially subordinate to the martingale $X$ if $|Y_0| \leq |X_0|$ and $\left[  X\right]_t  -\left[
Y\right]_t $ is a nondecreasing and non-negative function of $t$. We use the notation $Y<<X$ to indicate this subordination property. This notion of differential subordination is 
inspired by the differential subordination property introduced by Burkholder for discrete martingales as in Theorem \ref{sharp3} and also by the applications to singular integrals.  
For martingales with continuous
paths, it was introduced in \cite{BanWan3}. The inequalities in \cite{BanWan3} were extended under this 
condition to general continuous-time 
parameter martingales by Wang in \cite{Wan1}. Various other sharp
martingale inequalities (including exponential inequalities and inequalities for submartingales) proved by Burkholder in \cite{Bur37}, \cite{Bur45} and \cite{Bur50}, were extended in  \cite{Wan1} from the
discrete time setting (and stochastic integrals) to general continuous-time parameter martingales under the above definition of differential subordination.  More precisely we have the following theorem proved in \cite{BanWan3}. 

\begin{theorem}\label{BW_theorem}
Let $X$ and $Y$ be two $\ell^2$-valued continuous-time parameter martingales with continuous paths. That is, the function $t\to X_t$ is continuous almost surely.  Suppose that 
$Y<<X$. Then 
\begin{equation}
{\|Y\|}_p\leq (p^*-1){\|X\|}_p, \,\,\,\, 1<p<\infty.
\end{equation}
The inequality is sharp and strict if $p\neq2$ and $0<\left|  \left|
X\right|  \right|  _{p}<\infty$.
\end{theorem}

We also have the following extension proved in 
\cite{Wan1}.

\begin{theorem}\label{Wan1} Let $X$ and $Y$ be two $\ell^2$-valued continuous-time
parameter martingales such that $Y<<X$. Then,
\begin{equation}\label{sharpcon1}
\left|  \left|  Y\right|  \right|  _{p}\leq\left(  p^{\ast}-1\right)  \left|
\left|  X\right|  \right|  _{p},\,\,\,\, 1<p<\infty.
\end{equation}
The inequality is sharp and strict if $p\neq2$ and $0<\left|  \left|
X\right|  \right|  _{p}<\infty$. In addition, for all $\lambda \geq 0$, 

\begin{equation}\label{sharpcon2}
P\left(\sup_{t\geq 0}\left(|X_t|+|Y_t|\right) > \lambda\right)\leq \frac{2}{\lambda}\|X\|_p, 
\end{equation}
and this is also sharp. 
\end{theorem}

Except for the results proved in $\S4$ for L\'evy multipliers,  Theorem \ref{BW_theorem} suffices for all the applications presented in this paper. In particular, an important application of Theorem \ref{BW_theorem} arises when the martingales are transformations of stochastic integrals of Brownian motion.  Because of its importance for our application, we state this case as a separate Theorem. 
Let   
$$X_t=(X_t^1, X_t^2, \dots),\,\,\,\,\,Y_t=(Y_t^1, Y_t^2, \dots)$$
be two $\ell^2$--valued martingales on the filtration of $n$--dimensional Brownian 
motion.  We assume they both start at $0$ and 
that they have the stochastic integral representation (see \cite{Dur1}) 
\begin{equation}\label{Brorep}
X_t^{i}=\int_0^t H_s^i\cdot dB_s, \, \, Y_t^{i}=\int_0^t K_s^i\cdot dB_s, 
\end{equation}
where $B_t$ is $n$--dimensional Brownian motion and $H_s$ and $K_s$ are $\bR^n$-valued
processes adapted to its filtration. 
As usual, 
$$\langle X\rangle_t =
\sum_{i=1}^{\infty} \langle X^i\rangle_t=\sum_{i=1}^{\infty} \int_0^t |H^i_s|^2 \,ds
$$ denotes the quadratic variation process of $X_t$ with a similar definition for $\langle
Y \rangle_t$. Also, 
$$
{\langle X^i, Y^j\rangle}_t=\int_0^t H_s^i\cdot K_s^j \,ds$$
 denotes the covariation process.  We set 
 $$\|X\|_p=\sup_{t\geq 0}\Big\| \left(\sum_{i=1}^{\infty} |X_t^i|^2\right)^{1/2}\Big\|_p$$
 with a similar definition for $\|Y\|_p$.  Of course, this is nothing more than $\sup_{t>0}\| |X_t|_{\ell^2}\|_p$ where $|x|_{\ell^2}$ is the norm of the vector $x\in \ell^2$. 
 
 \begin{theorem}  Suppose $X_t=(X_t^1, X_t^2, \dots)$ and $Y_t=(Y_t^1, Y_t^2, \dots)$ are two $\ell^2$-valued martingales on the Brownian filtration with 
 $d\langle Y\rangle_t=\sum_{i=1}^{\infty}
|K_t^i|^2
\leq
\sum_{i=1}^{\infty} |H_t^i|^2=d\langle X\rangle_t$, a.e.  for all 
$t>0$.  Then $Y<<X$ and 
\begin{equation}\label{brownian1}
\|Y\|_p\leq (p^*-1)\|X\|_p,\,\,\,\, 1<p<\infty. 
\end{equation}
This inequality is sharp. 
 \end{theorem}
 
 Given a martingale of the form 
 $$
 X_t=\int_0^t H_s\cdot dB_s
 $$ as above and an $n\times n$ matrix-valued predictable process $A(s)$,  we define
the martingale transform
\begin{equation}\label{transform}
(A*X)_t=\int_0^t (A(s) H_s)\cdot d B_s. 
\end{equation}
We set 
$$\|A\|=\big\|\sup_{s> 0}|A(s)|\big\|_{L^{\infty}},
$$
 where $$\|A(s)\|=\sup\{|A(s)v|\colon
v\in \bR^n, |v| \leq 1\},$$ Our standing assumption throughout is that $\|A\|<\infty$.  For any $t>s$, 
$$
\langle A*X\rangle_t-\langle A*X\rangle_s \leq \|A\|^2 \langle X\rangle_t-
\|A\|^2 \langle X\rangle_s=\|A\|^2 \left(\langle X\rangle_t-
 \langle X\rangle_s\right).
$$

We remark that this definition extends to complex valued martingales and matrices with complex entries.  The quantity $\|A\|$ just has to be modified by defining $$|A(s)|=\sup\{|A(s)v|_{\bC^n}\colon
v\in \bC^n, |v|_{\bC^n} \leq 1\}.$$ Similarly, if $\calA(s)=\{A_i(s)\}_{j=1}^\infty$ is a sequence of $n\times n$
matrix-valued predictable processes,  we set 
$$|\calA(s)|^2=\sup\{\sum\limits_{i=1}^\infty |A_i(s)v|^2\colon v\in \bR^n, 
|v|\leq 1\}$$ and $\|\calA\|=\big\|\sup\limits_{s>0}|\calA(s)|\big\|_{L^{\infty}}$, with a similar definition when the matrices $A_j$ have complex entries. Again, we always assume $\|\calA\|<\infty$. The next corollary is an immediate consequence of Theorem \ref{BW_theorem}

\begin{corollary}\label{transforms1} Let $X_t=(X_t^1, \, X_t^2, \cdots)$ be an $\ell^2$-valued martingale and $\calA=\{A_i(s)\}_{i=1}^{\infty}$ predictable processes with $\|A_i\|\leq M$, for all $i$.  Define   $\calA*X_t=(A_1*X_t^1,\,  A_2*X_t^2, \cdots)$.  Then $\calA*X_t<<M\,X_t$ and 
\begin{equation}
\|\calA*X\|_p \leq (p^*-1)
M \|X|\|_p, \,\,\,\, 1<p<\infty.
\end{equation}
Furthermore,  if $X_t$ is either an $\bR$-valued or $\bC$-valued martingale and this time we define $\calA*X_t=(A_1*X_t, \, A_2*X_t, \cdots)$, then  
$\calA*X_t<<\|\calA\|\, X_t$ and  
\begin{equation}
\|\calA*X\|_p \leq (p^*-1)\|\calA\|\,\|X\|_p,  \,\,\,\, 1<p<\infty.
\end{equation}
These inequalities are sharp.
\end{corollary}

\subsection{Differential subordination and orthogonality}
Applications of the above inequalities to the Hilbert transform and to first order Riesz transforms motivated the notion of orthogonality given here.  While {\it conformal martingales} (see below) had been studied by several authors before in connection with the theory of Hardy $H^p$ spaces and harmonic functions in $\bC$ and $\bC^n$ and other applications (see \cite{GetSha1}, \cite{Fuj1}, \cite{Fuk1}, \cite{Fuk2}, \cite{Ube1}), the notion of orthogonality and subordination used below was introduced in \cite{BanWan3} to study martingale inequalities which arise from the Riesz transforms in $\bR^n$.  We say $X_t=\left(
X_t^{1},X_t^{2},\dots\right)  $ and $Y_t=\left(  Y_t^{1},Y_t^{2},\dots\right)  $ are
orthogonal if for each $i,j,\left[  X^{i},Y^{j}\right]  _{t}=0$ for all $t\geq0.$ 
While this definition is for general $\ell^2$-valued martingales, below we just recall the results 
the  for real valued martingales.  These follow from \cite{BanWan3, BanWan2, BanWan1}. 

\begin{theorem}\label{pic-ess}Let $X$ and $Y$ be two $\bR$-valued
continuous-time orthogonal martingales with $Y<<X$.  Then 
\begin{equation}\label{ortho1}
\|Y\|_p \le \cot\left(\frac{\pi}{2p^{\ast}}\right)\|X\|_p, \hskip.5cm 1 < p < \infty, 
\end{equation}
\begin{equation}\label{ortho2}
 \Big\|\sqrt{|X|^2 + |Y|^2}\Big\|_p \le \csc\left(
\frac{\pi}{2p^{\ast}}\right) \|X\|_p,\hskip.5cm 1 < p < \infty. 
 \end{equation}
These inequalities are sharp and strict if $p\neq 2$ and $0<\|X\|_p<\infty$. In addition,  for any $\lambda\geq 0,$
\begin{equation}\label{weakortho}
\lambda P\left(  \left|  Y\right|  \geq\lambda\right)  \leq D_1\left|  \left|
X\right|  \right|  _{1},
\end{equation}
where 
\begin{equation}\label{catalan}
D_1=\frac{1+\frac{1}{3^2}+\frac{1}{5^2}+\frac{1}{7^2}+\frac{1}{9^2}+
\cdots}{1-\frac{1}{3^2}+\frac{1}{5^2}-\frac{1}{7^2}+\frac{1}{9^2}-\cdots}=\frac{\pi^2}{8\beta(2)}\approx 1.328434313301,
\end{equation}
with $\beta(2)$  the so called ``Catalan" constant whose value is approximately $0.9159655$. 
The inequality (\ref{weakortho}) is sharp. 
\end{theorem}

The following extension of the above theorem was given by Janakiraman in \cite{Jan1}.

\begin{theorem}\label{jan}
Let $X$ and $Y$ be two $\bR$-valued continuous-time parameter 
orthogonal martingales such that $Y<<X$. Then
for any $\lambda>0$, 
\[\lambda^p P(|Y|\geq\lambda)\leq D_p\|X\|_p^p,\,\,\,\, 1\leq p\leq 2,\] where  
\[ D_p ={\left(\frac{1}{\pi}\int_{-\infty}^\infty \frac{{\left|\frac{2}{\pi}
\log{|t|}\right|}^p}{t^2 + 1} dt\right)}^{-1}.\]
This inequality is sharp. 
\end{theorem}

In \cite{Osc8}, Os\c{e}kowski identifies the best constant in the inequalities $\|Y\|_p\leq C_{p, \infty}\|X\|_{\infty}$ and $\|Y\|_1\leq C_{1, p}\|X\|_p$ under the assumption of orthogonality and differential subordination.  His precise result is the following  

\begin{theorem}\label{osc} Let $X$ and $Y$ be two $\bR$-valued continuous-time parameter 
orthogonal martingales such that $Y<<X$. Then for $1<p<\infty$, 
\begin{equation}\label{osc1}
\|Y\|_1\leq C_{1, p}\|X\|_p \hskip.5cm \text{and} \hskip.5cm \|Y\|_p\leq C_{p, \infty}\|X\|_{\infty}, 
\end{equation}
where 
$$
C_{p, \infty}=1, \hskip.5cm  1<p\leq 2, 
$$
$$
C_{p, \infty}=
\left(\frac{2^{p+2}\Gamma(p+1)}{\pi^{p+1}}\sum_{k=0}^{\infty}\frac{(-1)^k}{(2k+1)^{p+1}}\right)^{1/p}, \hskip.5cm  2<p<\infty,
$$
and 
$$
C_{1, p}=C_{\frac{p}{p-1}, \infty}.
$$
These inequalities are sharp.
\end{theorem} 

\begin{problem}\label{problem1} Determine the best constant for $2<p<\infty$ in Theorem \ref{jan}. 
\end{problem} 

\begin{remark}\label{weakremark}   If we drop the assumption of orthogonality and assume only  differential subordination, we have the bound $2/\Gamma(p+1)$,\,  for $1\leq p\leq 2$, due to Burkholder \cite{Bur39}, and for $2\leq p<\infty$,  the bound  ${p^{p-1}}/{2}$  due to Suh \cite{Suh1}; see also, \cite{Wan1}.  We should also point out here that Janakiraman's result was inspired by Choi's result in \cite{Cho4} which proves a version of Theorem \ref{jan} for $p=1$ for differentially subordinate orthogonal harmonic functions as in (\ref{jan2}) below.  There are also more recent versions of the inequalities of Burkholder and Suh by Os\c{e}kowski \cite{Osc1}   
for non-negative martingales $X$. 
\end{remark}

In terms of martingales on the $n$-dimensional Brownian motion the above  results give

 \begin{corollary}\label{orthoweak1}  Let 
 $X_t=\int_0^t H_s\cdot dB_s$ and $Y_t=\int_0^t K_s\cdot dB_s$
be two $\bR$-valued martingales on the filtration of n-dimensional Brownian motion with $K_t\cdot H_t=0$ and $|K_s|\leq |H_s|$ a.e.  for all 
$t>0$.  Then
\begin{equation}\label{brownian2}
\|Y\|_p \le \cot\left(  \frac{\pi}{2p^{\ast}}\right)\|X\|_p, \hskip.5cm 1<p<\infty, 
\end{equation}
\begin{equation}\label{brownian3}
 \Big\|\sqrt{|X|^2 + |Y|^2}\Big\|_p \le \csc\left(
\frac{\pi}{2p^{\ast}}\right) \|X\|_p, \hskip.5cm 1<p<\infty,
\end{equation}
 and for any $\lambda>0$, 
\begin{equation}\label{brownian4}
\lambda^p P\left(  \left|  Y\right|  \geq\lambda\right)  \leq D_p\left|  \left|
X\right|  \right|_p^p,  \hskip.5cm  1\leq p\leq 2.
\end{equation}
These inequalities are all sharp. 
 \end{corollary}

Stated in terms of the martingale transform by predictable matrices, we have

\begin{corollary}\label{transforms2} Let $A(s)$ be a $n\times n$ matrix-valued
predictable process with real entries with the property that $[A(s) v]\cdot v=0$, for all 
$v\in \bR^n$.
Then 
$$
\|A*X\|_p \leq \cot\left(\frac{\pi}{2p^{\ast}}\right)\|A\| \|X\|_p, \hskip.5cm 1<p<\infty,
$$
$$
\Big\|\sqrt{|A*X|^2 + |X|^2}\Big\|_p \leq \csc\left(
\frac{\pi}{2p^{\ast}}\right) \|A\| \|X\|_p, \hskip.5cm 1<p<\infty,
$$
and for any $\lambda>0$,
$$
\lambda^p P\left(|A*X| \geq\lambda\right)  \leq D_p\|A\|^p\left|  \left|
X\right|  \right|_p^p, \hskip.5cm 1\leq p\leq 2.
$$
These inequalities are all sharp. 
\end{corollary}

The inequalities (\ref{ortho1}),  (\ref{ortho2}) and (\ref{weakortho})  are the martingale versions of the inequalities of Pichorides \cite{Pic1},  Ess\'{e}n-Verbitsky \cite{Ess1, Vea1} and Davis \cite{Dav1}, respectively,  for harmonic and conjugate harmonic
functions.   We should note here that $$\cot\left(\frac{\pi}{2p^{\ast}}\right)< (p^*-1)$$ and that asymptotically, $$\cot\left(\frac{\pi}{2p^{\ast}}\right)\approx \frac{2}{\pi} (p^*-1),$$ as $p\to 1$ or  $p\to \infty$.
Thus orthogonality decreases the constants in Theorem \ref{Wan1}.  
 
 Motivated by the structure of the $\bR^2$-valued martingales that arise in the martingale representation for the Beurling-Ahlfors operator, the following theorem is proved in \cite{BanJan1}.

\begin{theorem}\label{thm3}
Let  $X_t=(X_t^1, X_t^2, \dots, X_t^m)$ and $Y_t=(Y_t^1, Y_t^2, \dots, Y_t^m)$ be two $\bR^m$-valued martingales on the filtration of $n$-dimensional Brownian motion with a representation as in (\ref{Brorep}). We assume $m\geq 2$. Let $Y$ satisfy 
$|K^i_s|^2=|K_s^j|^2$ for all $i, j\geq 1$ and 
$K_s^j\cdot K_s^i=0$, for $i\neq j$. Suppose $\sqrt{\frac{m+p-2}{p-1}}Y_1$ is 
differentially subordinate to $X$.  That is,
$$
d\langle\sqrt{\frac{m+p-2}{p-1}}Y_1\rangle_t=\frac{m+p-2}{p-1}|K_t^1|^2 \leq \sum_{j=1}^m
|H_t^j|^2=d\langle X\rangle_t,
$$
a.e. for all $t>0$. 
Then 
\begin{equation}
{\|Y\|}_p\leq (p-1){\|X\|}_p, \hspace{5mm} 2\leq p<\infty.
\end{equation}
\end{theorem}

Following the now standard terminology (see \cite{GetSha1}), we give the following definition for $\bR^2$ (or complex $\bC$) valued martingales. 

\begin{definition}\label{confmar} An $\bR^2$-valued martingale $Y_t=Y_t^1+iY_t^2$ with 
$Y_t^j=\int_0^t K_s^j \cdot d B_s,\ j=1, 2,$
 on the filtration of $n$--dimensional Brownian motion is said to be a conformal martingale if $K_s^1 \cdot K_s^2=0$
and 
$
|K_s^1|=|K_s^2|$, a.e. for all $s>0$. 
\end{definition}
From Theorem \ref{thm3} we have  
\begin{corollary}\label{cor1} 
Suppose $Y_t=(Y_t^1, Y_t^2)$ is a conformal martingale and $Y<<X$ where $X=(X_t^1, X_t^2)$ 
is any  $\bR^2$-valued martingale.   Then 
\begin{equation}\label{banjan}
\|Y\|_p\leq \sqrt{\frac{p(p-1)}{2}}\,\|X\|_p, \hspace{5mm} 2\leq p<\infty. 
\end{equation}
\end{corollary}

As pointed out in Borichev, Janakiraman and Volberg \cite{BorJanVol2},  the proof in \cite{BanJan1} also gives the following inequality: If $X_t=(X_t^1, X_t^2)$ is an $\bR^2$-valued conformal martingale and $Y<<X$, where $Y=(Y_t^1, Y_t^2)$ 
is any  $\bR^2$-valued martingale, then 
\begin{equation}\label{JV}
\|Y\|_p\leq \sqrt{\frac{2}{p(p-1)}}\,\|X\|_p, \hspace{5mm} 1<p\leq 2.
\end{equation}

The article  \cite{BorJanVol1}  contains a  sharp version of the inequality (\ref{banjan})  for $1<p\leq 2$. More precisely, Borichev, Janakiraman and Volberg prove the following 
\begin{theorem} 
Suppose  $X$ and $Y$ are two $\bR^2$-valued  martingales. \\ 
 (i)  {(\it Left Conformality)} Suppose $Y$ is conformal and $Y<<X$.   Then  
\begin{equation}\label{JV1}
\|Y\|_p\leq \frac{1}{\sqrt{2}}\frac{z_p}{1-z_p}\|X\|_p, \hspace{5mm} 1<p\leq 2,
\end{equation}
where $z_p$  is the least positive root in the interval $(0, 1)$ of the bounded Laguerre
function  $L_p$.  This inequality is sharp.\\ 
(ii) {(\it Right Conformality)} Suppose $X$ is conformal and $Y<<X$.  Then 
\begin{equation}\label{JV2}
\|Y\|_p\leq \sqrt{2}\frac{1-z_p}{z_p}\|X\|_p, \hspace{5mm} 2\leq p<\infty,
\end{equation}
where $z_p$  is the least positive root in the interval $(0, 1)$ of the bounded Laguerre
function  $L_p$. This inequality is sharp.
\end{theorem}

The cases of the sharp constants for $2<p<\infty$ and $1<p<2$ in (\ref{JV1}) and (\ref{JV2}), respectively, remain open.   Here we also refer the reader to the recent paper \cite{BanOse1} which contains various extensions and refinements of the results in \cite{BorJanVol2}.

The following problem was raised in \cite[p. 599]{BanWan3}
\begin{problem}
 Let
$
X_t^j=\int_0^t H_s^j \cdot d B_s,\ j=1, 2, 3,
$
be three $\bR$-valued martingales on the filtration of $n$-dimensional Brownian motion
which satisfy $H_s^1 \cdot H_s^2=0,$
and 
$
|H_s^1|=|H_s^2| \leq |H_s^3|.
$
Find the best constant $C_p$ in the inequality
$$
\Big\|\sqrt{|X^1|^2+|X^2|^2}\Big\|_p \leq C_p \|X^3\|_p,\,\,\,\,\, 1 < p < \infty.
$$
\end{problem}
The result in Borichev, Janakiraman and Volberg \cite{BorJanVol1} solves  this problem for the range  of  $1<p\leq 2$.  We take this opportunity to acknowledge the fact that our guess in \cite{BanWan3} for the best constant $C_p$ was incorrect.   

 It is well known that weak-type inequalities for martingales do not give information (at least not in any direct way) about weak-type inequalities for singular integrals.  Nevertheless,  the following problem is interesting as a martingale problem.

\begin{problem} Suppose $Y_t=Y_t^1+iY_t^2$ is a conformal martingale and $Y<<X$, where $X=(X_t^1, X_t^2)$ 
is any  $\bR^2$-valued martingale.  Find the best  constant $C_p$ in the inequality
$$
\lambda^p P\left(\sup_{t\geq 0}|Y_t|\geq \lambda\right)\leq C_p\|X\|_p^p,\,\,\,\, 1\leq p<\infty. 
$$
\end{problem}
Other variants of this question are also possible as in (\ref{sharpcon2}).  In addition, given that conformal martingales are time change of complex (2-dimensional) Brownian motion, this problem and others can be stated purely in terms of Brownian motion; see also Remark \ref{weakremark}.

\subsection{Outline of proofs} In order to illustrate Burkholder's techniques for obtaining sharp inequalities, we give an outline of the proof of Theorem \ref{BW_theorem} for $p>2$ for $\bR^2$-valued martingales as well as an outline of the proof of the inequality (\ref{ortho1}) in Theorem \ref{pic-ess} for martingales with continuous paths.  This outline follows \cite{BanWan3}.
 Let 
\[ V(x,y) = |y|^p-(p^*-1)^p|x|^p.\]
Our goal is to show that $EV(X_t,Y_t)\leq 0$ for all $t\geq 0$.  Consider Burkholder's  function $U:\bR^2\times \bR^2\to\bR$\, introduced in (\ref{u}) which we recall here again. 
\begin{equation}\label{function_U}
U(x,y) = p{(1-1/{p^*})}^{p-1}(|y|-\left(p^*-1)|x|)(|x|+|y|\right)^{p-1}.
\end{equation}
As in (\ref{sub1}), 
\begin{equation}
V(x, y)\leq U(x, y)
\end{equation}
for all $x, y\in \bR^2$. Thus to prove $EV(X_t,Y_t)\leq 0$, it suffices to prove that $EU(X_t,Y_t)\leq 0$.
 We follow the notation of \cite{BanWan3}. In particular, $U_x=\nabla U_x$, with a similar definition for $U_y$. The function $U_{xx}=\left(U_{x_i x_j}\right)$ is the matrix of second partials and similarly for $U_{xy}$ and $U_{yy}$. 
 
 The function $U$ has various structural advantages over the function $V$. In particular (see \cite{Bur45} and \cite{BanWan3}) 
 for all $x$, $y$, $h$, $k$ $\in {\bR^2}$, if $|x||y|\neq 0$, then
\begin{equation}\label{U_fact}
[U_{xx}(x,y)h]\cdot h+ 2
[U_{xy}(x,y)h]\cdot k+[U_{yy}(x,y)k]\cdot k= -c_p(A+B+C),
\end{equation}
 where $c_p>0$ is a constant depending only on $p$.  Furthermore,  for $p>2$, 
\begin{eqnarray}
A&=& p(p-1)(|h|^2-|k|^2)(|x|+|y|)^{p-2},\\ 
B &=& p(p-2)[|k|^2-(y',k)^2]|y|^{-1}(|x|+|y|)^{p-1}, \label{term_B}\\
C&=& p(p-1)(p-2)[(x',h)+(y',k)]^2|x|(|x|+|y|)^{p-3},
\end{eqnarray}
where we have used the notation $y' = y/|y|$ for  $y\ne 0$.  In addition, 
\begin{equation}\label{initialval}
U(x,y)\leq 0, \,\,\,\,\, |y|\leq |x|.
\end{equation}

The left side quantity in ($\ref{U_fact}$) is the directional concavity in 
direction $(h,k)$. That is, if $G(t) = U(x+ht,y+kt)$, then 
\[ G''(0) = [U_{xx}(x,y)h]\cdot h+ 2[U_{xy}(x,y)h]\cdot k
+[U_{yy}(x,y)k]\cdot k.\] Thus for instance $G''(0)\leq 0 $ whenever $|k|\leq |h|$. 
Burkholder uses this
property to prove the $(p^*-1)$ bound for discrete martingales. In \cite{BanWan3}, this property is explored in combination with the 
It\^o's formula and the  differential subordination as described above in the following way. 
Apply It\^o's formula to the function $U$ to get
\begin{eqnarray}\label{ito}
U(X_t,Y_t)&=&U(X_0,Y_0)+\int_0^t U_x(X_s,Y_s)\cdot dX_s\\
&+& \int_0^t U_y(X_s,Y_s)\cdot dY_s+\frac{I_t}{2},\nonumber
\end{eqnarray}
where 
\begin{eqnarray}\label{I_t}
dI_t &=& \sum_{i,j=1}^2( U_{x_i x_j}(X_t,Y_t)d\langle X^i,X^j\rangle_t\\
&+& 2U_{x_i y_j}(X_t,Y_t)
d\langle X^i,Y^j\rangle_t +U_{y_i y_j}(X_t,Y_t)
d\langle Y^i,Y^j\rangle_t).\nonumber
\end{eqnarray}
Recall  that we want to prove that 
\begin{equation}\label{toprove}
EU(X_t,Y_t)\leq 0.
\end{equation} 
Since 
$|Y_0|\leq |X_0|$, we have  $EU(X_0,Y_0)\leq 0$, by (\ref{initialval}). Since  
$$\int_0^tU_x
(X_s,Y_s)\cdot dX_s\,\,\,\,\, \text{and}\,\,\,\,\,\, \int_0^t U_y(X_s,Y_s)\cdot dY_s$$ are both martingales, their expectation
is $0$. Therefore 
\[EU(X_t,Y_t)\leq \frac{1}{2}EI_t.\]

Replacing $h_ih_j \rightarrow d\langle X_i,X_j\rangle$, $k_ik_j\rightarrow 
d\langle Y_i,Y_j\rangle$ and $h_ik_j\rightarrow d\langle X_i,Y_j\rangle$, in ($\ref{U_fact}$) and observing as in Burkholder \cite{Bur45} that the terms $B$ and $C$ are always 
non-negative, it follows that
\[ I_t \leq -p(p-1)c_p\int_0^t(|X_s|+|Y_s|)^{p-2}d(\langle X\rangle_s-\langle Y\rangle_s) \leq 0,\]
provided $Y$ is differentially subordinate to $X$. Hence $EI_t\leq 0$. This 
proves ($\ref{toprove}$) and hence the inequality in Theorem \ref{BW_theorem} for $\bR^2$-valued martingales with continuous paths.  

The proof of Theorem \ref{thm3} in \cite{BanJan1} for $m=2$ follows from a simple modification of this argument.   
Our assumptions are: $Y_t$ satisfies $d\langle Y_t^1\rangle =d\langle Y_t^2\rangle $, 
$d\langle Y_t^1,Y_t^2\rangle
=0$, and $\sqrt{\frac{p}{p-1}}Y_t^1$ is differentially subordinate to $X_t$.
As above, our goal is to show that $EU(X_t,Y_t)\leq 0$.
The above method adapted from \cite{Bur45} and \cite{BanWan3} was 
to drop  $B$ and $C$, then change
the norm square terms to quadratic variation terms and just keep $A$.  We now 
include $B$ of ($\ref{term_B}$) and verify the calculation again.
Observe that the term $$(y',k)^2 = (k_1y_1/|y|)^2 + (k_2y_2/|y|)^2 + 2
k_1k_2(y_1/|y|)(y_2/|y|)$$ in ($\ref{term_B}$) converts to
\[\left(\frac{Y_t^1}{|Y_t|}\right)^2d\langle Y^1 \rangle_t+ \left(\frac{Y_t^2}{|Y_t|}\right)^2
d\langle Y^2\rangle_t 
+2\frac{Y_t^1}{|Y_t|}\frac{Y_t^2}{|Y_t|}d\langle Y^1,Y^2\rangle_t.\] Since $d\langle Y^1\rangle =
d\langle Y^2\rangle_t$ and $d\langle Y^1,Y^2\rangle_t=0$, this gives $d\langle Y^1\rangle_t =
\frac{1}{2}d\langle Y\rangle_t$.

Given  that $|Y|^{-1}(|X|+|Y|)$ is greater than or equal to $1$, we find that 
the contribution from $B$ is bounded below by
\[p(p-2)(|X_t|+|Y_t|)^{p-2}(1/2)d\langle Y\rangle_t.\]
So,
\begin{eqnarray*}
A+B &\geq& p(p-1)(|X_t|+|Y_t|)^{p-2}(d\langle X\rangle_t -\left(\frac{p}{2(p-1)})d\langle Y\rangle_t \right)\\
&=& p(p-1)(|X_t|+|Y_t|)^{p-2}(d\langle X\rangle_t-(\frac{p}{p-1})d\langle Y_t^1\rangle)\\
&\geq& 0,
\end{eqnarray*}
where the  last inequality takes into account our assumption on differential subordination. It follows that the
term $I_t$ in (\ref{ito})  is non-positive, and therefore 
$EU(X_t,Y_t)\leq 0$ as required.  This proves the special case $m=2$ for Theorem \ref{thm3}.   

We now give an outline of the proof of the inequalities (\ref{ortho1}) and (\ref{ortho2}) in Theorem \ref{pic-ess} for real valued martingales on the filtration of $n$-dimensional Brownian motion.  Let us assume that $1<p\leq 2$. Then $p^*=p/(p-1)$, $\cot(\frac{\pi}{2p^{\ast}})=\tan(\frac{\pi}{2p})$ and 
$\csc(\frac{\pi}{2p^{\ast}})=\sec(\frac{\pi}{2p})$.
 The reverse Minkowski's
inequality implies that 
$$
\sqrt{||X||^2_p + ||Y||^2_p}\leq ||\sqrt{|X|^2 + |Y|^2}||_p.
$$
Thus inequality  (\ref{ortho1}) follows from (\ref{ortho2}) and the fact that $$\sec^2\left(
\frac{\pi}{2p}\right) = \tan^2\left(
\frac{\pi}{2p}\right) +
1.$$ To prove (\ref{ortho2}), define 
\begin{equation}
V(x,y)= |y|^p-\sec^{p}\left({\pi\over 2p}\right)|x|^p.
\end{equation}
Our task is to show that under the orthogonality and differential subordination conditions, 
$
EV(X_t, Y_t)\leq 0.
$
As above, we look for a function $U$ with $V(x, y)\leq U(x, y)$ for all $x, y\in \bR$\,.  Such function comes from  examining Pichorides' paper \cite{Pic1}. Namely, take 
\begin{equation}
U(x,y)= -\tan\left(\frac{\pi}{2p}\right)R^p\cos(p\theta),
\end{equation}
with 
$|x|=R\cos(\theta)$,  $y=R\sin(\theta)$, $-\pi/2 \leq \theta \leq \pi/2$. 
It is proved in \cite{BanWan3} that indeed $V(x,y)\leq U(x,y)$ and that $U(x_0,y_0)\leq 0$ for $|y_0| \leq |x_0|$. Furthermore, $U$ has the following property:
\begin{equation}\label{ito2}
U_{xx}(x, y)|h|^2+U_{yy}(x, y)|k|^2\leq -C(x,y) (|h|^2-|k|^2), 
\end{equation}
where $C(x, y)$ is a non-negative function for all $x, y\in \bR$.  With this it follows from (\ref{I_t}) that 
under the assumption of orthogonality and differential subordination, $EI_t\leq 0$ which gives that $EU(X_t, Y_t)\leq 0$ and proves the desired result. 

We should note here, however, that to make all of the above precise and justify the application of It\^o's formula,  one needs some approximations which are presented in Lemma 1.1 and Proposition 1.2 of \cite{BanWan3}.  For full details, we refer the reader to \cite{BanWan3, BanWan2}.  

\subsection{The Burkholder method}
Here we  summarize the basic strategy of Burkholder for finding best constants in martingale problems, using examples from above.
Suppose $X$ and $Y$ are two continuous martingales, with special properties and relations yet to be specified. We wish to find the best constant $C_p$ in the inequality
\begin{equation}
\|Y\|_p \leq C_p \|X\|_p, \hspace{5mm} 1<p<\infty.
\end{equation}
Let $V(x,y) = |y|^p-c^p|x|^p$. 
We must find the minimal $c$ so that $EV(X_t,Y_t)\leq 0$.  Written in terms of It\^o's formula, 
$$ V(X_t,Y_t) = \int_0^t dV_s + \frac{1}{2} \int_0^t d\left<V\right>_s,$$
where the first term is a martingale and the second quadratic-variation process is of bounded variation. Therefore 
$$ EV(X_t,Y_t) = \frac{1}{2}\int_0^t E d\left<V\right>_s, $$
and $EV(X_t,Y_t)\leq 0$ would follow if we can show for instance that $Ed\left<V\right>_s \leq 0$ for all $s>0$. In general, however, the quadratic term $d\left<V\right>_s$ can take both positive and negative values and its expectation is just as difficult to estimate as that of $V(X_t, Y_t)$.

Enter Burkholder. We find the minimal constant $c$ such that there exists a function $U(x,y)$ satisfying $V(x,y)\leq U(x,y)$, $U(0,0)=0$ and  
\begin{eqnarray}
d\left<U(X,Y)\right>_s &\leq& 0
\end{eqnarray}
 for all $s$, for all $(X_t,Y_t)$ satisfying the required conditions.
The last condition is equivalent to requiring that the process $U(X_t,Y_t)$ is a supermartingale. We now have
\begin{equation}
EV(X_t,Y_t)\leq EU(X_t,Y_t) \leq EU(X_0, Y_0)=U(0,0)=0
\end{equation}
as required. But also the condition $d\left<U\right>_s\leq 0$ is equivalent to requiring that $U$ is a supersolution for a family of second-order partial differential operators.  Hence Burkholder's approach essentially replaces the martingale problem with an obstacle problem in the calculus of variations setting.  It is here that Bellman functions enter.

It is often the case that this family of operators has a few extremal operators which give the general solution. Consider the special case where $X$ and $Y$ are real-valued and satisfying the subordination condition $d\left<Y\right>_s \leq d\left<X\right>_s$. Then $d\left<U\right>_s\leq 0$ implies
$$ U_{xx}d\left<X\right>_s + 2 U_{xy}d\left<X,Y\right>_s + U_{yy}d\left<Y\right>_s \leq 0.$$
It can be shown that an extremizing case is $d\left<Y\right>_s = d\left<X\right>_s$, and hence we must find a $U$ function that satisfies
$$ U_{xx} + 2 a U_{xy} + U_{yy} \leq 0,$$ where $-1\leq a\leq 1$. Again, it follows that it suffices for $U$ to satisfy 
$$ U_{xx} \pm 2 U_{xy} + U_{yy}= \left(\partial_x \pm \partial_y\right)^2U \leq 0.$$ Thus we are looking for the minimal majorant of $V$ that is biconcave in the $\pm \frac{\pi}{4}$ directions. 

Burkholder (see \cite[p. 81]{Bur45}) finds that when $c=p^*-1$, the majorant $U\geq V$ exists and equals
\begin{equation}\label{smallestbicon}
U(x, y)=\begin{cases}
\alpha_p\left(|y|-(p^*-1)|x|\right)\left(|x|+|y|\right)^{p-1},\,\,\, |y|> (p^*-1)|x| \\
V(x,y), \hspace{3mm} |y|\leq (p^*-1)|x|
\end{cases}
\end{equation}
where 
$$
\alpha_p=p\left(1-\frac{1}{p^*}\right)^{p-1}.
$$
He also shows by finding near extremals that $(p^*-1)$ is the best possible constant.

We next illustrate Burkholder's approach to the weak-type $(p,p)$ constant for orthogonal martingales. Suppose the $\bR^2$-valued martingale $(X,Y)$ is conformal. That is,  $d\left<X\right> = d\left<Y\right>$ and $d\left<X,Y\right> = 0$ and that we wish to find the best constant $D_p$ in the inequality
\begin{equation}
\sup_{\lambda>0} \lambda^p P(|Y|>\lambda) \leq D_p\|X\|_p^p, \hspace{5mm} 1\leq p<\infty.
\end{equation}
This  is equivalent to 
\begin{equation}
P(|Y|>1)\leq D_p\|X\|_p^p.
\end{equation}
Consider the function 
\begin{equation}
W(x,y) = \begin{cases}
1-c^p|x|^p, \hspace{5mm} |y|\geq 1; \\
-c^p|x|^p, \hspace{5mm} |y|<1.
\end{cases}
\end{equation}
We have to find the minimal $c$ such that $EW(X,Y)\leq 0$ whenever $(X,Y)$ is conformal. Again, we look for the minimal majorant $U\geq W$ so that $U(0,0)=0$ and such that $U(X_t,Y_t)$ is a supermartingale. This means that 
\begin{eqnarray*}
d\left<U\right>_s &=& U_{xx}d\left<X\right>_s + 2 U_{xy}d\left<X,Y\right>_s + U_{yy}d\left<Y\right>_s \\
&=& d\left<X\right>_s \left(U_{xx}+U_{yy}\right) \leq 0.
\end{eqnarray*}
Thus we want to find the least constant $c$ so that the minimal superharmonic majorant $U$ of $W$ has $U(0,0)=0$. When $1\leq p\leq 2$, $U$ is obtained by harmonically extending its boundary values into the strip $|y|<1$. Janakiraman (and Choi for $p=1$) thus finds the best constant; see \cite{Jan1}. For $2<p<\infty$, as stated in Problem \ref{problem1}, the best weak-type $(p,p)$-constant remains unknown. However we note that it is likely that Burkholder's approach is applicable in exactly the same manner and that the solution will involve finding information on the minimal superharmonic majorant of $W$.

For more on these techniques and the connections to Bellman functions, see the recent papers of
Borichev, Janakiraman and Volberg \cite{BorJanVol1, BorJanVol2}. As
already mentioned in \S1, \cite{VasVol2} is also highly recommended.

\begin{remark}
The It\^o formula was used in \cite{BanWan3} to deal specifically with martingales with continuous paths and in Wang \cite{Wan1} to deal with more general continuous-time parameter martingales. Similar use of It\^o's formula was  employed by Burkholder in \cite{Bur43}. As seen in the above presentation, It\^o's formula allows us to simplify some of Burkholder's original analysis and also leads quickly to the necessary conditions that $U$ must satisfy. Of course, finding such functions, and especially the explicit expression given by Burkholder, is quite a different matter. 
 \end{remark}

\section{Singular integrals}

In this section we present some applications of the sharp martingale inequalities in \S1 to singular integrals. While the most recent applications and those of current interest to many researchers  are to the Beurling-Ahlfors operator, there are also interesting applications to the Hilbert and Riesz transforms.  These applications raise other interesting questions that we shall mention along the way.  We begin with the Hilbert transform.

\subsection{The Hilbert transform} The most basic example of singular integrals is  the Hilbert transform on the real line defined by 
$$
Hf(x)=p.v.\frac{1}{\pi}\int_{\bR}\frac{f(y)}{x-y}dy.  
$$
The celebrated inequalities of M. Riesz and Kolmogorov assert that 
\begin{equation}\label{Riesz}
\|Hf\|_p\leq C_p\|f\|_p, \hskip.5cm 1 < p < \infty,
\end{equation}
and that for all $\lambda>0$, 
\begin{equation}\label{Kolmogorov}
\lambda m\{x\in \bR: |Hf(x)|>\lambda\}\leq C_1\|f\|_1, 
\end{equation}
where $m$ is the Lebesgue measure on $\bR$. Inequalities of this type hold for very general Calder\'on-Zygmund singular integrals as detailed in Stein \cite{Ste2}. 
 In \cite{Pic1}, Pichorides showed that the best constant in the Riesz inequality (\ref{Riesz}) is $\cot(\frac{\pi}{2p^{\ast}})$ and in \cite{Dav1}, Davis proved that the best constant in Kolmogorov's inequality (\ref{Kolmogorov}) is the constant $D_1$ given by (\ref{catalan}). It is interesting to note here, in relation to the topic of this paper, that Davis' original proof used  Brownian motion. There are now several proofs of these sharp inequalities. The Pichorides proof has been greatly simplified; see \cite[Ex. 4.1.13, p. 264]{Gra1}, as has the Davis proof; see \cite{Dur1}.  In \cite{Vea1} and \cite{Ess1}, Verbitsky  and Ess\'{e}n proved the sharp related inequality: 
 \begin{equation}\label{Essen}
 \Big\|\sqrt{|Hf|^2+|f|^2}\Big\|_p\leq \csc(\frac{\pi}{2p^{\ast}}) \|f\|_p, \hskip.5cm 1 < p < \infty.
 \end{equation}
 
Theorem \ref{pic-ess} provides martingale proofs of the inequalities (\ref{Riesz}), (\ref{Kolmogorov}) and (\ref{Essen}).  
On the other hand, we should also acknowledge here that the martingale inequalities were inspired by the analysis inequalities. However, because of the universality of the  martingale inequalities when we applied them back to analysis they provide more information, and have wider applications,  than just for the Hilbert transform. In particular,  they contain information for orthogonal harmonic functions which are not necessarily the real and imaginary parts of analytic functions. But much, much, more than that.  The martingale inequalities apply to Riesz transforms in various settings, as we shall discuss below.  In addition, we have the following Theorem of Janakiraman  \cite{Jan1} which follows from his Theorem  \ref{jan}.  It extends Davis' inequality. 
 
\begin{theorem}  For all $\lambda>0$, 
\begin{equation}\label{jan1}
\lambda^p\,m\{x\in \bR: |Hf(x)|>\lambda\}\leq D_p\|f\|_p^p,\,\,\,\, 1\leq p\leq 2,
\end{equation}
\[ D_p ={\left(\frac{1}{\pi}\int_{-\infty}^\infty \frac{{\left|\frac{2}{\pi}
\log{|t|}\right|}^p}{t^2 + 1} dt\right)}^{-1}.\]
This constant is best possible. 
\end{theorem}

The following open problem is of interest. (See the related Problem  \ref{problem1}.) 
\begin{problem}  Determine the best constant in the inequality (\ref{jan1}) for the range $2<p<\infty$.\end{problem}

Let $D$ be an open connected set in $\bR^n$.
Let $u, v\colon D\to \bR$\, be two harmonic functions.
Following Burkholder \cite{Bur48} we say that $v$ is differentially subordinate to $u$
if for all $x\in D$,
$$
|\nabla v (x)|\leq |\nabla u(x)|.
$$
We will also say that $u$ is orthogonal to $v$ if for all $x\in D$,
$$
\nabla u (x) \cdot \nabla v (x)=0.
$$ 
As in Burkholder \cite{Bur48}, we also define for $1\leq p < \infty$,
$$
\|u\|_p=\sup_{D_0}\left(\int_{\partial D_0} |u|^p d\mu\right)^{1/p},
$$
where the supremum is taken over all subdomains $D_0$ of $D$, with $\overline{D_0}\subset D$, containing a fixed
point $\xi_0$ and $\mu$ is the harmonic measure on $\partial D_0$ based at $\xi_0$.
We assume that $u(\xi_0)=v(\xi_0)=0$.
The next theorem follows from Corollary \ref{orthoweak1}.   It is the analogue of Theorem
5.1 in Burkholder \cite{Bur48} with the extra assumption of orthogonality.
As before,  orthogonality reduces the constant.

\begin{theorem}
Suppose $u$ and $v$ are two real-valued harmonic
functions in a domain $D\subset \bR^n$ which are orthogonal and with $v$ differentially
subordinate to $u$.
Then 
\begin{equation}
\|v\|_p \leq \cot\left(\frac{\pi}{2p^{\ast}}\right) \|u\|_p, \hskip.5cm 1 < p < \infty
\end{equation}
and 
\begin{equation}
\|\sqrt{|v|^2+|u|^2}\|_p \leq \csc\left(\frac{\pi}{2p^{\ast}}\right)\|u|_p, \hskip.5cm 1 < p < \infty.
\end{equation}
Furthermore,  for all $\lambda>0$, 
\begin{equation}\label{jan2}
\lambda^p\,\mu\{\xi\in \partial D_0: |v(\xi)|>\lambda\}\leq D_p\int_{\partial D_0} |u|^p d\mu,\,\,\,\, 1\leq p\leq 2.
\end{equation}
These inequalities are sharp.
\end{theorem}

If $F(z)=u(z)+iv(z)$ is analytic in the unit disc in the complex plane $\bC$, the inequalities above are again the classical inequalities of Pichorides,  Ess\'{e}n-Verbitsky and the extension of the inequality of Davis by Janakiraman.   For $p=1$, the inequality (\ref{jan2}) 
was first proved by Choi in \cite{Cho4}.  

There is also a version of Theorem \ref{osc} for harmonic functions which implies sharp versions of the classical LlogL inequality of Zygmund for conjugate harmonic functions in the disk;  see \cite{Osc8}.

\subsection{First order Riesz transforms} 
The Riesz transforms in $\bR^n$ are the natural generalizations of the Hilbert transform to higher dimensions.  Given that these operators all fall into the general category of Calder\'on-Zygmund singular integrals, it follows that  they are bounded on $L^p$, for $1<p<\infty$, and that they are weak-type (1,1).  Here we are interested in sharp bounds. For this purpose, we can define these operators
 for functions which are $C^{\infty}$ and of compact support. We denote this class by $C_0^{\infty}(\bR^n)$. The results below are all stated for such functions.  In addition, since the Riesz transforms are operators which take real valued functions to real valued functions, in dealing with their norms we could just as well assume that  the functions are real-valued. Once norm estimates for such functions are obtained, the same holds for functions taking values in the complex plane or even in a Hilbert space; see for example \cite{IwaMar1}.  We start by defining

$$
R_j f(x)=p.v.\,C_{n}
\int_{\bR^{n}} {(x_j-y_j)\over |x-y|^{n+1}}f(y)dy,\,\,\,\,\,  j=1, 2, \dots n, 
$$
where $C_n=\Gamma ({n+1\over 2})/\pi^{n+1\over 2}$ is chosen so that as Fourier multipliers 
$$\widehat{R_j f}(\xi)={i\xi_j\over |\xi|} \widehat{f} (\xi).$$ We use the normalization 
\begin{equation}\label{Fo}
\widehat{f}(\xi)=\int_{\bR^n} e^{2\pi i\xi\cdot x}f(x)dx
\end{equation} 
and 
\begin{equation}\label{Foin}
{f}(x)=\int_{\bR^n} e^{-2\pi ix\cdot\xi}\widehat{f}(\xi)d\xi
\end{equation} 
for the Fourier transform. 
For $y>0$, let 
$$
p_y(x)=\frac{C_n\, y}{\left(|x|^2+|y|^2\right)^{\frac{d+1}{2}}}
$$
be the Poisson kernel in $\bR^n$ with Fourier transform 
\begin{equation}\label{poissonfourier}
\widehat{p}_{y}(\xi)=e^{-2\pi y|\xi|}.
\end{equation} 
For any $j=1, 2, \dots, n$, set $\partial_j=\frac{\partial}{\partial x_j}$ and $\nabla=(\partial_1, \partial_2, \dots, \partial_n)$.  
With this we see that $$\widehat{\partial_j f}(\xi)=-2\pi i\,\xi_j\,\widehat{f}(\xi),$$
for any $f\in C_0^{\infty}(\bR^n)$. Since 
$$
\int_0^{\infty} (2\pi i\xi_j) e^{-2\pi y|\xi|} dy=\frac{i\xi_j}{|\xi|}
$$
it follows that 
\begin{equation}\label{rieszpoisson}
R_j f(x)=
\int_0^{\infty} {\partial_jP_y f}(x)dy,
\end{equation}
where, for any function $f$,  $P_yf(x)$ is the Poisson semigroup acting on the function.  That is, $P_y f$  is the convolution 
of $p_y$ with $f$. With this expression and the relationship between the generator of the semigroup and the semigroup itself, we see that 
\begin{equation}\label{rieszgenerator}
R_jf=\frac{\partial}{\partial x_j}(-\Delta)^{-1/2}f.
\end{equation}
In addition, we have the following semigroup representation for the vector of Riesz transforms involving the gradient and the Laplacian. 
\begin{equation}\label{rieszvector}
Rf=\left(R_1f, R_2f, \dots, R_n f\right)=\nabla(-\Delta)^{-1/2}f.
\end{equation}
While this identity can be verified with the aid of the Fourier transform, it is useful to have it in terms of the semigroup because it permits generalizations of the Riesz transforms to various other settings, see for example \cite{Arc1}, \cite{ArcLi1}, \cite{BanBau1}, \cite{DraVol1} and \cite{Gun1}.

\subsection{Background radiation} 
In their groundbreaking paper \cite{GunVor1}, Gundy and Varopoulos gave a stochastic integral representation for Riesz transforms using the so called ``background radiation'' process.  Using this representation one can transfer questions about $L^p$-boundedness of the Riesz transforms to $L^p$-boundedness of martingale transforms. From this point of view the Riesz transforms appear as basic examples of a more general class of operators which we call $T_A$.  Before we describe these operators more precisely, let us give the basic idea for this procedure following the presentation in \cite{BanLin1}.  The process can be summarized
by the following diagram
\begin{eqnarray}\label{diagram}
&&L^p(\bR^n)\longmapsto Har(\bR^{n+1}_{+})\longmapsto \M^p\longmapsto \M^p \longmapsto L^p(\bR^n)\\
&&f\longmapsto U_f(x, y)\longmapsto U_f(B_t)\longmapsto A*f_t\longmapsto E[ A*f \,\,|B_0=x].\nonumber
\end{eqnarray}
In words, with the upper half-space written as 
$$R^{n+1}_{+}=\{(x, y): x\in \bR^n, y>0\},$$ we let $U_f(x, y)=P_yf(x)$ be the harmonic extension of $f$ to $R^{n+1}_{+}$.  We compose this function with a Brownian
motion $B_t$ in $R^{n+1}_{+}$ to obtain a martingale $U_f(B_t)$.  We then transform this martingale by a
matrix $A$ to obtain a new martingale $(A*f)_t$ which is then projected by conditional expectation to finally arrive at a
function in  $L^p(\bR^n)$ which we denote by $T_{A}f$. 

We now describe the procedure in more detail. Fix $a>0$. Let $B_t$ be Brownian motion in the upper-half space  
starting on the hyperplane $\bR^n\times\{a\}$ with the Lebesgue measure
as its initial distribution. That is, we define measures $P^a$ on paths by 

$$
P^a\left(B_t\in \Theta \right)=\int_{\bR^n} P_{(x, a)}\left(B_t\in \Theta\right)dx
$$
for any Borel set $\Theta\subset \bR^{n+1}_{+}$  and where $P_{(x, a)}$ are the probability measures
associated with the Brownian motion $B_t$ starting at the point $(x, a)$. Of
course, the measures $P^a$ are no longer probability measures. If we let
$$\tau=\inf\{t>0: B_t\not \in \bR^{n+1}_{+}\}$$ be the first exit time of the Brownian motion form the upper half-space 
and use Fubini's Theorem to integrate out the
Poisson kernel  we find that (using the notation $E^a$ for the expectation associated with the measure $P^a$) 
\begin{equation}\label{unif}
E^af(B_{\tau})=\int_{\bR^n}f(x)\, dx,
\end{equation}
for all non-negative functions $f$. 

In the same way, integrating away the heat kernel and computing the Green's function for the half-line (see \cite{Ban1}), gives that 

\begin{equation}\label{occu}
E^a\int_0^{\tau}F(B_s)\,ds=2\int_0^{\infty}\int_{\bR^n} (y\wedge a) F(x, y)dx dy,
\end{equation}
for all non-negative functions $F$ on $R^{n+1}_{+}$.  Both (\ref{unif}) and (\ref{occu}) continue to hold for those $f$ and $F$ for which the integrals are finite. 

We would now like to let 
$a\to\infty$ 
so that we can use the 
Littlewood-Paley identities as in \cite{Ste2}. But since the initial
distribution of $B_0$ depends on $a$ we would have to make sense of this as a
limit of processes. In \cite{GunVor1}, Gundy and Varopoulos used time-reversal
to construct a filtered probability space and a process $\{B_t\}$ indexed by $t\in (-\infty, 0]$,
called the background radiation process. Heuristically speaking, the paths of $B_t$ are
Brownian paths which originate from $\{y=\infty\}$ at time $t=-\infty$ and
exit $\bR^{n+1}_{+}$ at $t=0$ with Lebesgue measure as distribution. Letting $E$ be
the expectation with respect to the measure associated with the background radiation
process, the identities (\ref{unif}) and (\ref{occu}) become,

\begin{equation}\label{unif1}
Ef[\left(B_0\right)]=\int_{\bR^d}f(x)\, dx
\end{equation}and 

\begin{equation}\label{occu1}
E\int_{-\infty}^{0}F(B_s)\,ds=2\int_0^{\infty}\int_{\bR^n} yF(x, y)dxdy.
\end{equation}

For a different construction using the ``entrance" law of Bessel processes, see Gundy and Silverstein \cite{GunSil1}. Here it is also interesting to point out  P. Meyer \cite[p.185]{Mey2} where he describes the duality between the killed Brownian motion
on the half-line $[0, \infty]$ and the $3$-dimensional Bessel process: ``D'une mani\`ere
intuitive, on peut donc dire que le retourn\'e du processus de Bessel issu de $\lambda_0$ est le
{\it mouvement brownien venant de l'infini et tu\'e en 0}."
The usual rules of stochastic integration and potential theory apply to this
process.  For this, see Varopoulos \cite{Var1} where details on the stochastic integrals are given. Also, a more elementary procedure can be found in \cite{Ban1}.  

\subsection{The operators $T_A$ and Riesz transforms}\label{T_A}
We now consider martingale transforms and their projections arising from functions on $\bR^n$. 
For $f\in C_0^{\infty}(\bR^n)$, let $P_yf(x)=U_f(x, y)$. If $A(x,y),x\in \bR^d, y > 0$, is an $(n+1)\times (n+1)$ matrix-valued
function, we define the martingale 
transform of $f$ by the stochastic integral
$$
A*f=\int_{-\infty}^0 \left[A(X_s,Y_s){\nabla}U_f (X_s,Y_s)\right]\cdot d B_s. 
$$
Here and below, we use ${\nabla}$ to denote the ``full" gradient of functions defined on $\bR^{n+1}_{+}$. That is, with  $\partial_0=\frac{\partial}{\partial y}$, we set
$$
{\nabla}U_f=\left(\partial_0 U_f, \partial_1 U_f, \dots, \partial_n U_f\right). 
$$
Also, here and below we use the notation $dB_s$ for 
$$
dB_s=(dY_s,\ dX_s^1,\cdots,dX_s^n). 
$$
We define the operator taking function in $\bR^n$ into functions in
$\bR^n$, called the projection of the martingale transform, by the conditional
expectation 
\begin{equation}
T_A f(x)=E[A*f|B_0=(x,0)].
\end{equation}

If $A(x, y)$ is a $(n+1)\times (n+1)$ matrix valued function on $\bR^{n+1}_{+}$, we define 
$$\|A\|=\sup_{(x, y)\in \bR^{n+1}}\|A(x, y)\|$$  where $\|A(x, y)\|=\sup\{|A(x,y)v|:v\in \bR^{n+1}, |v|\leq 1\}.$
If $\calA=\{A_i(x,y)\}_{i=1}^\infty$ is a sequence of such functions we define  $\|\calA\|$ similarly as in Corollary \ref{transforms1}. We assume $\|A\|$ and $\|\calA\|$ are both finite. 
Using the fact that the conditional expectation is a contraction in $L^p$,
for $1< p< \infty$, we have the following Corollaries which follow immediately from Corollaries 
\ref{transforms1} and \ref{transforms2}.  We  recall again the fact that  the distribution of $B_0$ is the Lebesgue
measure.

\begin{theorem}\label{TAproj} Let $\{f_i\}_{i=1}^\infty$ be a sequence of functions in $C_0^{\infty}(\bR^n)$ and 
$\calA=\{A_j(x, y)\}_{i=0}^{\infty}$ a sequence of $(n+1)\times (n+1)$ matrix-valued functions  such that $\|\calA\|<\infty$ and  $\|A_i\|\leq M$, for all $i$. 
Then, 
\begin{equation}\label{proj1}
\Bigg\|\left(\sum_{i=1}^\infty|T_{A_i} f_i (x)|^2\right)^{1/2}\Bigg\|_p \leq
(p^*-1)M
\Bigg\|\left(\sum_{i=1}^\infty |f_i|^2\right)^{1/2}\Bigg\|_p,
\end{equation}
for $ 1 < p < \infty$. Furthermore, for any  $f\in C_0^{\infty}(\bR^n)$,   
\begin{equation}\label{proj2}
\Bigg\|\left(\sum_{i=1}^\infty|T_{A_i}f(x)|^2\right)^{1/2}\Bigg\|_p\leq
(p^*-1) \|\calA\| \|f\|_p, \,\,\, 1 < p < \infty. 
\end{equation}
\end{theorem}

\begin{theorem}\label{orth-real}  Let $A(x,y)$ be an $(n+1)\times (n+1)$ matrix with real entries and suppose that for all $(x,y)\in \bR_+^{n+1}$, $[A(x,y)v]\cdot v=0$, for all 
$v\in \bR^{n+1}$.
Then
\begin{equation}\label{proj3}
\|T_A f\|_p \leq \cot\left(\frac{\pi}{2p^{\ast}}\right)\|A\| \|f\|_p,\,\,\,\,\,1 < p < \infty
\end{equation}
and
\begin{equation}\label{proj4}
\|\sqrt{|T_A f|^2+|f|^2}\|_p \leq \csc\left(\frac{\pi}{2p^{\ast}}\right)\|A\| \|f\|_p, \,\,\,\,\,1 < p < \infty.
\end{equation}
These inequalities are sharp.
\end{theorem}

Before we give the consequences of the above representation for $R_j$ as $T_{A_j}$, we state the following 
interesting Littlewood-Paley identity which follows exactly as in the proof given below 
for the representation of $R_j$ by first removing the conditional expectation and then 
using the occupation-time formula (\ref{occu1}). 

\begin{theorem}\label{prop1} For all $f, g\in C_0^{\infty}(\bR^n)$ and all $(n+1)\times (n+1)$ matrix-valued functions $A(x, y)$ in $\bR^{n+1}$ with $\|A\|<\infty$, we have
\begin{eqnarray}
\int_{\bR^d} T_A f(x)\, g(x)\, dx &=& E\left(\int_{-\infty}^0 \left[A(X_s,Y_s){\nabla}U_f (X_s,Y_s)\right]\cdot{\nabla}U_g(X_s,Y_s)\,ds\right)\nonumber\\
&=&2\int_0^\infty \int_{\bR^n} y \left[A(x, y){\nabla}U_f (x, y)\right]\cdot{\nabla}U_g(x, y)dx\,dy.
\end{eqnarray}
\end{theorem}

The following corollary is immediate from this, H\"older's inequality and the fact that $\|T_Af\|_p\leq \|A\|(p^*-1)\|f\|_p$. 
\begin{corollary}\label{corofprop}For all $f, g\in C_0^{\infty}(\bR^n)$ and all $(n+1)\times (n+1)$ matrix-valued functions $A(x, y)$ in $\bR^{n+1}$ with $\|A\|<\infty$, we have
\begin{equation*}
2\Bigg|\int_0^\infty \int_{\bR^n} y \left[A(x, y){\nabla}U_f (x, y)\right]\cdot{\nabla}U_g(x, y)dx\,dy\Bigg|\leq \|A\|(p^*-1)\|f\|_p\|g\|_q,
\end{equation*}
for all $1<p<\infty$.  Here $q$ is the conjugate exponent of $p$.  In addition, if $A$ has the orthogonality property as in Theorem \ref{orth-real} for all $(x, y)\in \bR^{n+1}$, then
\begin{equation*}
2\Bigg|\int_0^\infty \int_{\bR^n} y \left[A(x, y){\nabla}U_f (x, y)\right]\cdot{\nabla}U_g(x, y)dx\,dy\Bigg|\leq \|A\|\cot\left(\frac{\pi}{2p^{\ast}}\right)\|f\|_p\,\|g\|_q
\end{equation*}
\end{corollary}

With the matrix 
$$
A(x, t)=\frac{{\nabla}U_f(x, t) \otimes{\nabla}U_g(x, t)}{|{\nabla}U_f(x, t)||{\nabla} U_g(x, t)|}, 
$$
the first inequality in the Corollary gives 

\begin{corollary}\label{nazvoldualtyforharmart} For all $f, g\in C_0^{\infty}(\bR^n)$, 
\begin{equation} 2\int_0^{\infty}\int_{\bR^n} y\Big|{\nabla}U_f(x, t)\Big|\, 
\Big|{\nabla} U_g(x, t)\Big|\, dxdt\leq (p^*-1)\|f\|_p\|g\|_q,
\end{equation}
for all $1<p<\infty$. 
\end{corollary}

This is a ``Poisson extension" version of Corollary \ref{nazvoldualtybymart} below.

\begin{remark}\label{convolutionpoisson}
It is also interesting here to record the following expression for the operator $T_A$ as a convolution with a kernel.  More precisely, for any $(n+1)\times (n+1)$ matrix $A(x, y)$ with $\|A\|<\infty$ we have 
$$
T_A f(x)=\int_{\bR^n}K(x-\tilde{x})f(\tilde{x})d\tilde{x},
$$
where 
$$
K(x)=2\int_0^\infty \int_{\bR^n} y A(\overline x, y){\nabla}p_y (\overline{x})\cdot {\nabla} p_y (x-\overline{x})d\overline{x} dy.
$$
\end{remark}

\begin{problem}\label{weak1-poisson}
While we know the boundedness of these convolution operators from the martingale transforms, it would be interesting to study their properties as Calder\'on-Zygmund operators, including their weak $L^1$-boundedness which does not follow from the corresponding martingale inequalities.  
\end{problem}

Now let $A_j=(a_{\ell m}^j),j=1,2,\cdots,n$, be the $(n+1)\times (n+1)$
matrix given by
\begin{equation}\label{Riesmat}
a_{\ell m}^j=\begin{cases} 1&\text{$\ell=1,\ m=j+1$}\\
-1&\text{$\ell=j+1,\ m=1$}\\
0&\text{otherwise}.
\end{cases}
\end{equation}
Then
\begin{equation}\label{Rieszrep}
A_j * f=\int_{-\infty}^0 {\partial_j U_f} (X_s,Y_s)dY_s-
\int_{-\infty}^0 {\partial_0 U_f} (X_s,Y_s)dX_s^j.
\end{equation}

It follows from \cite{GunVor1} that with this $A_j$,  $T_{A_j}f=R_j f$.  We give the proof of this  
important fact to illustrate these computations which are used in many places in the literature. 
Let $g\in C_0^{\infty}(\bR^n)$.  By the It\^o formula (since $U_g$ vanishes at $\infty$), 
$$
g(B_0)=\int_{-\infty}^0 {\nabla} U_g (X_s,Y_s)\cdot dB_s.
$$
Thus, using the properties of the conditional expectation, the fact that
the distribution of $B_0$ is the Lebesgue measure, (\ref{unif1}) above, and that the Green's function for the process is $2y$ (the identity (\ref{occu1}) above),  and recalling that 
$$
\widehat{{\partial_j U_f}}(\xi,y)=-2\pi i\xi_j \,e^{-2\pi y|\xi|}\,\widehat{f}(\xi),
$$
and 
$$
\widehat{{\partial_0 U_g}} (\xi,y)=-2\pi |\xi|\, e^{-2\pi y|\xi|}\,\widehat{g}(\xi),
$$
we have,
\begin{eqnarray*}
 &&\int_{\bR^d} E\left(\int_{-\infty}^0 
{\partial_j U_f}(X_s,Y_s)dY_s\Big|\, B_0 = x\right)g(x)dx\\
&=& E \left( g(B_0) \int_{-\infty}^0 {\partial_j U_f} (X_s,Y_s)
dY_s \right)\nonumber\\
&= &E\left(\int_{-\infty}^0 {\partial_j U_f} (X_s,Y_s)
{\partial_0 U_g}(X_s,Y_s)ds\right)\nonumber\\
&=&\int_0^\infty \int_{\bR^d} 2y {\partial_j U_f}
(x,y) {\partial_0 U_g} (x,y)dxdy.\nonumber\\
\end{eqnarray*}
If we continue with this string of identities and apply Plancherel's theorem we see that the last quantity above equals
\begin{eqnarray}\label{firstorderriesz}
&&\int_0^\infty 2y \int_{\bR^d} \widehat{{\partial_j U_f}}
(\xi,y)\overline{\widehat{{\partial_0 U_g}}} (\xi,y)d\xi\,dy\nonumber\\
&=&  8\pi^2\int_0^\infty y \int_{\bR^d} i\xi_j\,|\xi|\,e^{-4\pi y|\xi|}\widehat{f}(\xi) \overline{\widehat{g}(\xi)}d\xi\,dy\nonumber\\
&=& 8\pi^2 \int_{\bR^d} i\xi_j\,|\xi|\,\widehat{f}(\xi)
 \widehat{g}(\xi)\left(\int_0^\infty y\,  e^{-4\pi y|\xi|}\,dy \right) d\xi\nonumber\\
&=&\frac{1}{2}\int_{\bR^d}\widehat{R_j f}(\xi) \overline{\widehat{g}(\xi)} \, d\xi
=\frac{1}{2}\int_{\bR^d}{R_j f}(x){g}(x) \, dx.
\end{eqnarray}
In a similar way, we can prove that 

\begin{equation*}
-\int_{\bR^d} E\left(\int_{-\infty}^0 
{\partial_0 U_f}(X_s,Y_s)dX_s^j\Big|\, B_0=x\right)g(x)dx=\frac{1}{2}\int_{\bR^d}{R_j f}(x){g}(x) dx.
\end{equation*}

We observe that the matrix $A_j$ is has the orthogonality property of Theorem \ref{orth-real} and that $\|A_j\|=1$.  Therefore  (\ref{proj3}) and (\ref{proj4}) give 
\begin{corollary} For all $f\in C_0^{\infty}(\bR^n)$ and all $1<p<\infty$, 
\begin{equation}\label{rieszpic}
\|R_j f\|_p \leq  \cot\left(\frac{\pi}{2p^{\ast}}\right)\|f\|_p\
\end{equation}
and
\begin{equation}\label{rieszess}
\Big\|\sqrt{|R_j f|^2+|f^2|}\Big\|_p \leq \csc\left(\frac{\pi}{2p^{\ast}}\right)\|f\|_p,
\end{equation}
\end{corollary}
These are the inequalities of Pichorides \cite{Pic1} and Ess\'{e}n-Verbitsky \cite{Ess1, Vea1} 
for the Riesz transforms proved in \cite{BanWan3}. 

The inequality (\ref{rieszpic}) follows, as pointed out by Iwaniec and Martin in \cite{IwaMar1}, from the Calder\'on-Zygmund method of rotations. These authors also proved that the inequality is sharp by showing that the Riesz transforms are Fourier multiplier extensions of the Hilbert transform; see \cite{IwaMar1} for details. The inequalities for orthogonal martingales also lead to inequalities for Riesz transforms on compact Lie groups with the same constants; see for example \cite[Theorem 1]{Arc1}. In the same way it follows that the constant in (\ref{rieszess}) is also best possible.

If we take the matrices $\tilde{A}_j=(a_{\ell m}^j)$ with $a_{\ell m}^j=1$,
if $\ell=1, m=j+1$, we have $T_{\tilde{A}_j}f={1\over 2} R_j f$.
For this sequence of matrices $\|\overrightarrow{A}\|=1$ in (\ref{proj2}).  This gives 
 
\begin{corollary}\label{rieszvector} For all $1 < p < \infty$, 
\begin{equation}\label{nablaries}
\Bigg\|\nabla(-\Delta)^{-1/2}f\Bigg\|_p=
\Bigg\|\left(\sum_{j=1}^n |R_j f|^2\right)^{1/2}\Bigg\|_p
\leq 2(p^*-1)\|f\|_p.
\end{equation}
\end{corollary}

This inequality is proved in Iwaniec and Martin \cite{IwaMar1} for $2\leq p<\infty$ with the
constant $\sqrt{\pi}\cot ({\pi\over 2p})$ replacing the $2(p^*-1)$. 
As $p\to \infty$, the Iwaniec-Martin constant is slightly better than $2(p^*-1)$.

There is now a large literature showing that various $L^p$-constants for 
operators in $\bR^n$ are independent of the dimension $n$.  That the constant in Corollary \ref{rieszvector} is independent of the dimension has been known for many years.  For some of this literature, we refer the reader to \cite{Ban1, Ban1.1, BanWan3, DraVol1, DuoRub1, IwaMar1, 
Pis1, Ste4, Ste5}.  Given all the technology available these days to study Riesz transform, we believe the following problem is interesting and that its solution may lead to new techniques that could be useful in other problems.

\begin{problem} Find the best constant $C_p$ in the inequality 
\begin{equation}
\Bigg\|\left(\sum_{j=1}^n |R_j f|^2\right)^{1/2}\Bigg\|_p
\leq C_p\|f\|_p, \quad 1<p<\infty.
\end{equation}
The obvious conjecture, of course, is that $C_p=\cot\left(\frac{\pi}{2p^{\ast}}\right)$. 

\end{problem}

We next consider the Ornstein-Uhlenbeck operator $L=\frac{1}{2}\Delta-x\cdot\nabla$ in $\bR^n$ 
equipped with the Gaussian measure
$$
d\mu=\frac{e^{-\frac{|x|^2}{2}}}{(2\pi)^{n/2}}\,dx
$$
and the Gaussian Riesz transforms 

\begin{equation}\label{mayerriesz}
R_{g}f=\nabla(-L)^{-1/2}f,
\end{equation}
for $f\in C_0^{\infty}(\bR^n)$. Using techniques from probabilistic Littlewood-Paley-Stein
theory, P.A. Meyer \cite{Mey1} proved that $\|R_g(f)\|_p\leq C_p\|f\|_p$, for $1<p<\infty$, where the constant does not depend on the dimension. The  inequalities remain true on the ``truly" infinite
dimensional Wiener space equipped with Wiener measure.  In \cite{Pis1}, Pisier gave an alternative proof of Meyer's result using the Calder\'on-Zygmund method of rotations.  From Pisier's result, it follows that $C_p=O(p)$, as $p\to\infty$, and that $C_p=O(1-p)^{-3/2}$, as $p\to 1$.  In \cite{Gun2}, (see also \cite{Gun1}),  Gundy adapts the martingale representation from \cite{GunVor1} to prove Meyer's theorem. A different  proof can be found in \cite{Gut1}.   By adapting the proof of (\ref{nablaries}) in \cite{BanWan3} to prove an  inequality for Riesz transforms on the $n$--dimensional spheres in $\bR^n$,  Arcozzi \cite{Arc1} proves Meyer's theorem with $C_p\leq 2(p^*-1)$, for all $1<p<\infty$.  Asymptotically, (as $p\to\infty$ or 1)  this bound is best possible as shown in \cite{Lar1}.  Arcozzi also treats more general compact groups.  See also \cite{ArcLi1}.    

The literature on Riesz transform for more general diffusions,  and for Brownian motion on manifolds, is very large and it would be impossible for us to review it here. In addition to the already mentioned examples presented in \cite{Arc1} where Burkholder's inequalities play a crucial role, we refer the reader to the recent papers of X.D. Li \cite{Li1, Li2, Li3} where the martingale inequalities are use to obtain Riesz transform inequalities for Riemannian manifolds and generalized Ornstein-Uhlenbeck operator on abstract Wiener space yielding the same bounds of $2(p^*-1)$ as in (\ref{nablaries}).  In particular, see Corollaries 1.5 and 1.6, page 254, of \cite{Li3}. 

As the reader no doubt has noticed, while we prominently featured the weak-type inequalities for martingales and the Hilbert transform, nothing has been mentioned about such results for the Riesz transform. Unfortunately, the weak-type inequalities (which in general follow from the Calder\'on-Zygmund theory) do not follow from the martingale inequalities due to the simple fact that  the conditional expectation does not preserve weak-type estimates.  We should also point out here that due to the Cauchy-Riemann equations and the It\^o formula, in the case of the Hilbert transform the conditional expectation plays no role and the stochastic integral representation is exact.  It is for this reason that weak-type inequalities for the Hilbert transform follow from those on orthogonal martingales. For more on this, see \cite{Ban1}.
Related to the discussion in this paragraph, we have the following problem. 

\begin{problem}\label{Riesztranweak}  
Find the optimal constant $C_p$ in the inequality
\begin{equation}\label{weakriesztrans}
\lambda^p m\{x\in \bR^n: |R_jf|>\lambda\}\leq C_p\|f\|_p^p,\,\,\,\,\, 1\leq p<\infty, 
\end{equation}
for all $\lambda>0$.
\end{problem}

Of considerable interest is the case $p=1$.  For this it is not even known if the inequality holds with a constant independent of the dimension $n$.  To the best of our knowledge, the best available result is that of Janakiraman \cite{Jan} which proves that $C_1\leq C\log(n)$, where $C$ is independent of $n$.  This result follows from a more general results for singular integrals with certain homogeneous kernels.  The proof in \cite{Jan} is via the Calder\'on-Zygmund machinery with various modifications.

We also refer the reader to  \cite{DraVol1} where Bellman functions techniques are used to obtain various dimension free estimates and where one  finds  the following interesting statement (p.~171) : ``We hope that the properties of the Bellman function could also be utilized in a way to obtain dimension free estimates of the weak type 1-1."  We also share this hope and belief but as the authors of \cite{DraVol1} conclude, ``so far this has eluded us."

\begin{remark}
Another problem of considerable interest is to investigate the weak-type inequality (\ref{weakriesztrans}) for $p=1$ for the Riesz transforms on Wiener space.  This amounts to proving the inequality for the Gaussian Riesz transforms in $\bR^n$ (with respect to the Gaussian measure $\mu$) with a constant independent $n$. While some estimates are known, (see \cite{FabGutSct1}), this problem seems to be wide open. 
\end{remark}

Here is the problem  more precisely. 

\begin{problem} Let 
$R_{g}f=\nabla(-L)^{-1/2}f$ be the Gaussian Riesz transforms in $\bR^n$ defined by (\ref{mayerriesz}). Prove that there is a constant $C$ independent of $n$ such that 
\begin{equation}
\lambda \mu\{x\in \bR^n: |R_{g}f|>\lambda\}\leq C\|f\|_1,
\end{equation}
for all $\lambda>0$.
\end{problem}

To further whet the readers appetite, we mention that the martingale transform techniques presented in this paper can be used to study the boundedness of Riesz transforms on manifolds under curvature assumptions.  This has been done in \cite{Li3} by extending the 
Gundy-Varopolous representation to manifolds and applying the martingale inequalities following \cite{BanWan3}.    
Li's results hold for Riesz transforms on complete Riemannian manifolds $(M, g)$ with metric $g$ for the Ornstein-Uhlenbeck operator $L=\Delta_M-\nabla{\varphi}\cdot \nabla$ with  volume measure  $d\mu(x)=e^{-\varphi(x)}\sqrt{det(g(x))}\,dx$ and  under the assumption that $Ric(L)=Ric(M)+\nabla^2\varphi\geq 0.$
%
%
%

As in the case of the sphere in $\bR^n$ treated by Arcozzi in \cite{Arc1} to obtain the Meyer's (\cite{Mey1}) Riesz transforms inequality for the Ornstein-Uhlenbeck operator
on the classical Wiener space, the fact that the inequalities in Li's  results have constants independent of the dimension gives an extension of P. Meyer's theorem to the abstract Wiener space defined by Gross in \cite{Gro1}.  

We have already mentioned that as $p\to 1$ or $p\to\infty$, asymptotically,  the constant $2(p^*-1)$ is best possible in the classical Meyer's inequality; see \cite{Lar1}. More precisely, it is proved that $\|R\|_p$ grows, 
within constant factors,  like $1/(p-1)$, as $p\to 1$, and like $p$, as $p\to\infty$. 
For more on this subject and references to some of the literature on Riesz transforms on manifolds,  see  \cite{Li1, Li2, Li3}.

\subsection{Multipliers of Laplace transform-type; Poisson semigroup} We now consider the case when the matrix $A=a(y){I}$, where ${I}$ is the  $(n+1)\times(n+1)$ identity matrix and $a(y)$ is a function defined on $(0, \infty)$  with $\|a\|_{\infty}<\infty$. It follows exactly as in (\ref{firstorderriesz}) that the operator $T_A$, which this time we just  denote by $T_a$, is the Fourier multiplier operator 
\begin{equation}\label{laplaceFour1}
\widehat{T_af}(\xi)=\left(16\pi^2|\xi|^2\int_{0}^{\infty} y a(y)\,e^{-4\pi y|\xi|} dy\right) \widehat{f}(\xi). 
\end{equation}
As in (\ref{rieszpoisson}), we can conveniently re-write this as 
\begin{equation}\label{laplaceFour2}
T_af(x)=\int_0^{\infty} ya(y) \partial_0^2 P_{2y} f(x) dy.
\end{equation}
These operators have been studied by many authors including by those of \cite{Mey0}, \cite{Ste0} and \cite{Var1}.   From Theorem \ref{TAproj}, we have the following theorem first proved in  \cite[p. 123]{Ban1.1} but with better constants. 

\begin{theorem}\label{LaplaceTran} Let $\{f_j\}_{j=1}^\infty$ be a sequence of functions in $C_0^{\infty}(\bR^n)$ and 
$\{a_j (y)\}_{i=0}^{\infty}$ a sequence of functions such that $\|a_j\|_{\infty}\leq M<\infty$, for all $j$. 
Then, 
\begin{equation}\label{proj1-1}
\Bigg\|\left(\sum_{j=1}^\infty|T_{a_j} f_j (x)|^2\right)^{1/2}\Bigg\|_p \leq
(p^*-1)M
\Bigg\|\left(\sum_{j=1}^\infty |f_j|^2\right)^{1/2}\Bigg\|_p, 
\end{equation}
for $1 < p < \infty$. Furthermore, set $$a(y)=\left(\sum_{j=1}^{\infty}|a_j(y)|^2\right)^{1/2}$$ and suppose that $\|a\|_{\infty}<\infty$. 
The for any  $f\in C_0^{\infty}(\bR^n)$, 
\begin{equation}\label{proj3-1}
\Bigg\|\left(\sum_{i=1}^\infty|T_{a_j}f(x)|^2\right)^{1/2}\Bigg\|_p\leq
(p^*-1) \|a\|_{\infty} \|f\|_p, \,\,\, 1 < p < \infty. 
\end{equation}
\end{theorem}

This leads to the following corollary for the so called ``imaginary powers" of the Laplacian. 

\begin{corollary} 
Suppose 
$$
a(y)=\frac{(2y)^{-2i\gamma}}{\Gamma\left(2(1-i\gamma)\right)}. 
$$
Then 
$$
T_af(x)=(-\Delta)^{i\gamma}f(x)
$$
and  
\begin{equation}\label{Hyt1}
\|(-\Delta)^{i\gamma}f(x)\|_p\leq \frac{(p^*-1)}{|\Gamma\left(2-2i\gamma\right)|}\|f\|_p,\,\,\,\,\,\, 1<p<\infty. 
\end{equation}
\end{corollary} 

For other interesting and more recent work related to these multipliers, see Hyt\"onen \cite{Hyt1}, \cite{Hyt2}, and the many references contained therein to related applications of martingale inequalities in UMD spaces. 
In particular, the inequality (\ref{Hyt1}) is proved in \cite{Hyt2} for Hilbert space valued functions as a corollary of more general results for UMD Banach spaces; (see Corollary 5.1, p.~354).  However, note that our constant here is better (twice better) than the constant given there.  This improvement comes from the fact that we used the full identity matrix $I$ while in \cite{Hyt2} the matrix $\tilde{I}=(a_{jk})$ which has $a_{11}=1$ and 
$a_{jl}= 0$, for all other $j,k$, is used.  With $\tilde{I}$ one obtains only half of $(-\Delta)^{i\gamma}$ and hence the function $a(y)$ above has to be multiplied by 2.  
\begin{remark} Theorem \ref{LaplaceTran} has a version for Laplace transform-type multipliers in terms of the projections of space-time Brownian martingale transforms discussed in the next section. In particular, the imaginary powers of the Laplacian $(-\Delta)^{i\gamma}$ can be obtained in an even more direct way from those projections and the constant in (\ref{Hyt1}) can be improved; see (\ref{laplaceheat}) and (\ref{Hyt2}) below. 
\end{remark}

\begin{remark} Other interesting Fourier multipliers arise if we take matrices of the form $a(y)A$, where $A$ has constant coefficients and $a$ is a bounded function on $(0, \infty)$.   For example, with $A=A_j$, $j=1, 2,\dots, n$, as in 
(\ref{Riesmat}), which gives the first-order Riesz transforms, we obtain the multipliers

\begin{equation}\label{RiesLaplace1}
\widehat{T_jf}(\xi)=\left(16\pi^2i\xi_j|\xi|\int_{0}^{\infty} y a(y)\,e^{-4\pi y|\xi|} dy\right) \widehat{f}(\xi),
\end{equation}
$j=1, 2, \dots, n$,  
which  can then be written as 
\begin{equation}\label{RiesLaplace2}
T_jf(x)=2\int_0^{\infty} ya(y) \partial_j\partial_0P_{2y} f(x) dy.
\end{equation}
For these operators, assuming that $a$ is real valued,  we have even the better estimate from Theorem \ref{orth-real}, 

\begin{equation}\label{RiesLaplace3}
\|T_jf\|_p\leq  \cot\left(\frac{\pi}{2p^{\ast}}\right)\|a\|_{\infty}\|f\|_p,
\end{equation}\label{RiesLaplace3}
for any $1<p<\infty$.   We also note here that from (\ref{laplaceFour1}) and (\ref{RiesLaplace1}), $T_j(f)(x)=T_{a}R_jf(x)$, where $R_j$ is the $jth$ Riesz transform and $T_a$ is the operator in (\ref{laplaceFour2}). 
\end{remark}

\subsection{Second order Riesz transforms}
In this section we present the space-time martingale approach introduced in the joint paper with M\'endez \cite{BanMen1} to study the norms of second order Riesz transforms. This  is a modification of the Gundy-Varopoulos background radiation process approach.  The novelty in \cite{BanMen1} is the use
of space-time Brownian motion and ``heat" martingales to obtain better
estimates for singular integrals of second-order. It is interesting to
observe that this construction is technically much simpler than the
Gundy-Varopoulos \cite{GunVor1} construction. That the author of this
paper missed this simple heat martingale construction for so many years
after his work with Gang in \cite{BanWan3} is one of those facts of life that he must just accept and learn to live with!
 
The second-order Riesz transforms in $\bR^n$ are defined by iteration of the first-order Riesz transforms. Thus, proceeding via the Fourier transform we have  
$$\widehat{R_j^2f}(\xi)=\frac{-\xi_j^2}{|\xi|^2}\widehat{f}(\xi)\,\,\,\,\,\, \text{and}\,\,\,\,\,\widehat{R_jR_kf}(\xi)=\frac{-\xi_j\xi_k}{|\xi|^2}\widehat{f}(\xi).$$  These operators, just like the first-order ones, can be written in terms of the background radiation process used in \S3.2 above. In fact, it follows exactly  as in the proof of ({\ref{firstorderriesz}) that 
 if we define the matrix $$A_{jk}=(a_{\ell m}^{jk}),\
j,k=1,2,\cdots, n,\ j\not= k,$$ by 
$$
{a}_{\ell m}^{jk}=\begin{cases} -1&\text{$\ell=k+1,\ m=j+1$}\\
-1&\text{$\ell=j+1,\ m=k+1$}\\
0&\text{otherwise},
\end{cases}
$$
we have $T_{A_{jk}}f=R_j R_k f$.
With the matrices $\tilde{A}_{jj}=(\tilde a_{\ell m}^j),\ j=1,2,\cdots,n$, defined by
$$
\tilde{a}_{\ell m}^j=\begin{cases} -1,&\text{$\ell=j+1,\ m=j+1$}\\
\ 1,&\text{$\ell=m,\ \ell \not= j+1,\ m\not= j+1$},
\end{cases}
$$
we have $T_{A_{jj}}f=R_j^2 f$.
Since $\|A_{jk}\|=1$ and $\|A_{jj}\|=1$, we obtain the following estimates for $R_jR_k$ from (\ref {proj1}).  
This was proved in \cite{BanWan3}. 
\begin{theorem}\label{second1} Let  $f\in C_0^{\infty}(\bR^n)$. 
For all $1 < p < \infty$,
$$
\|R_j R_k f\|_p \leq (p^*-1)\|f\|_p,\ j\not= k,
$$
for all $j\not= k$,
and 
$$
\|R_j^2 f\|_p \leq (p^*-1)\|f\|_p,
$$
for all $j=1, 2,\dots, n$. 
\end{theorem}

\subsection{The  Beurling-Ahlfors operator and the Iwaniec Conjecture}\label{IwBa-conj} 
The Beurling-Ahlfors operator is the singular integral operator (Fourier multiplier) 
on the complex plane $\bC$ (or $\bR^2$) defined for $f\in C_0^{\infty}(\bC)$ by 
\begin{equation}\label{BAdefinition}
Bf(z) = -\frac{1}{\pi}\textrm{p.v.}\int_{\bC}\frac{f(w)}{(z-w)^2}dm(w). 
\end{equation}
In terms of the Fourier multiplier, 
$\widehat{Bf}(\xi)=\frac{\overline{\xi}^2}{|\xi|^2}\hat f(\xi)$.  This means that 
we can write it in terms of second order Riesz transform as 
\begin{equation}\label{BAriesz} 
B=R_2^2-R_1^2+2iR_1R_2.
\end{equation}

This operator is of fundamental importance in several areas of analysis  including  PDE's and 
  the geometry of quasiconformal mappings \cite{ Ast1, AstIwaMar1, DonSul1, Iwa1, Iwa2, Iwa3, IwaMar0, IwaMar1,  PetVol1}.  As a Calder\'on-Zygmund singular integral, it is bounded on $L^p(\bC)$, for $1<p<\infty$. The computation of its operator norm $\|B\|_{p}$ has been an open problem for almost thirty years now. 
In  \cite{Leh1}, Lehto showed that $\|B\|_{p} \geq  p^*-1$. Inspired in part by the celebrated Gehring-Reich
 conjecture \cite{GerRei} on the area distortion of quasiconformal mappings in the plane 
 (proved by K. Astala  \cite{Ast1}), T. Iwaniec  \cite{Iwa1}  made the following 
 
  \begin{conjecture}\label{Iwaniec} For all $f\in C_0^{\infty}(\bC)$ and $1<p<\infty$, it holds that  
  \begin{equation}\label{Iwaniec-Equation}
  \|Bf\|_{p}\leq (p^*-1)\|f\|_p.
  \end{equation}
 \end{conjecture}

 From Theorem \ref{second1} and formula (\ref{BAriesz}), we immediately get the following theorem proved in \cite{BanWan3}. 
 
 \begin{theorem} For all $f\in C_0^{\infty}(\bC)$ and $1<p<\infty$, we have 
\begin{equation}\label{BA4} 
\|Bf\|_{p}\leq 4(p^*-1)\|f\|_p.
\end{equation}
\end{theorem}

Let us observe that the matrix that gives the full operator $B$ as a $T_A$ operator ($T_{\mathbb{B}}$ in this case) is given by 
\begin{equation}
\mathbb{B}=\left[\begin{array}{ccc} 0 & 0 & 0 \\ 0 & 2 & -2i \\ 0 & -2i & -2 \end{array}\right],
\end{equation}
When acting on real vectors (vectors whose components are all real) this matrix has norm $2\sqrt{2}$ but when acting on complex vectors has norm $4$. This observation and a  more direct application of the inequality \eqref{proj1} give the following result also proved in \cite{BanWan3}. 

\begin{theorem} For all $f\in C_0^{\infty}(\bC)$ with $f:C\to \bR$\, and $1<p<\infty$, we have 
\begin{equation}\label{BA4-1}
\|Bf\|_{p}\leq 2\sqrt{2}(p^*-1)\|f\|_p.
\end{equation}
\end{theorem}

The following question immediately arises from \eqref{BA4} and \eqref{BA4-1}.

\begin{question}\label{question-background-smallest-B}
Is it possible to find a matrix of lower norm that also represents the Beurling-Ahlfors operator, thereby improving the bound in (\ref{BA4})? 
\end{question}

To answer this question we set $R_0=Id$ and let $\M^{(n+1)\times (n+1)}(\bC)$ denote the space of $(n+1)\times (n+1)$ matrices with complex entries and let 
$$
\Phi:\M^{(n+1)\times (n+1)}(\bC)\to \left\{\sum_{i,j}^{n} a_{ij}R_iR_j,  a_{ij}\in  \bC\right\}
$$
be the mapping $A\to T_A$.  This is a surjective linear mapping,
and so a basis for the kernel can be computed. 
In fact, when $n=2$, we let 
$$
K_1=\left[\begin{array}{ccc} -1 & 0 & 0 \\ 0 & 1 & 0 \\ 0 & 0 & 1 \end{array}\right], \quad\quad  K_2=\left[\begin{array}{ccc} 0 & 0 & 0 \\ 0 & 0 & 1 \\ 0 & -1 & 0 \end{array}\right] 
$$
and 
$$
K_3=\left[\begin{array}{ccc} 0 & 1 & 0 \\ 1 & 0 & 0 \\ 0 & 0 & 0 \end{array}\right], \quad\quad K_4=\left[\begin{array}{ccc} 0 & 0 & 1 \\ 0 & 0 & 0 \\ 1 & 0 & 0 \end{array}\right]. 
$$
It is proved in \cite{BanLin1} that $$ker(\Phi)=span\left\{K_1, K_2, K_3, K_4\right\}.$$  From this we can prove the following theorem  which completely answers the above question.  We refer the reader to \cite[pp.~237 \& 252]{BanLin1} for complete proofs. 

\begin{theorem} \label{background-smallest-B}
\begin{equation}
\inf_{\alpha_1, \alpha_2, \alpha_3, \alpha_4\in \bC}\|\mathbb{B}+\alpha_1K_1+\alpha_2K_2+\alpha_3K_3+\alpha_4K_4\|=\|\mathbb{B}\|=4.
\end{equation}
\end{theorem}

Thus using the background radiation process of  Gundy-Varopoulos, the bound $4(p^*-1)$ cannot be improved simply by picking a matrix of smaller norm that also represents  $B$. Nevertheless, this bound, proved in \cite{BanWan3}, was the first estimate on the $L^p$-norm of $B$ 
with an explicit constant involving the $(p^*-1)$ conjectured bound and the result ignited considerable interest on the conjecture and on the probabilistic techniques used to derive the bound.   

\subsection{The Nazarov-Volberg estimate} In \cite{NazVol2}, Nazarov and Volberg improved the bound in \eqref{BA4} to 

\begin{theorem} For all $f\in C_0^{\infty}(\bC)$ and $1<p<\infty$, we have
\begin{equation}
\|Bf\|_{p}\leq 2(p^*-1)\|f\|_p.
\end{equation}
\end{theorem}

 The proof involved a nice application of techniques from Bellman functions.  As it turns out, however, the existence of the  Bellman function in this case  also depends on Burkholder's inequalities in \cite{Bur39} for Haar martingales.  Before we give the key inequality of Nazarov and Volberg, we give a definition of the second-order Riesz transforms in terms of the heat semigroup. If we set 
\begin{equation}\label{heat}
h_t(x)=\frac{1}{(2\pi t)^{n/2}}e^{-\frac{|x|^2}{2t}}
\end{equation}
whose Fourier transform is given by
$
\widehat{h}_t(\xi)=e^{-2\pi^2 t|\xi|^2},
$
we see that 
\begin{eqnarray}\label{semigroupriesz}
\widehat{R_j^2f}(\xi)=\frac{-\xi_j^2}{|\xi|^2}\widehat{f}(\xi)&=&{\frac{1}{2}}\int_0^{\infty}(-4\pi^2 
\xi_j^2)e^{-2\pi^2 t|\xi|^2}\hat{f}(\xi)\,dt\\
&=&{\frac{1}{2}}\int_0^{\infty}e^{-2\pi^2 t|\xi|^2}(-4\pi^2 
\xi_j^2)\hat{f}(\xi)\,dt\nonumber\\
&=& {\frac{1}{2}}\int_0^{\infty}\widehat{\partial_j^2H_t f}(\xi)\,dt\nonumber\\
&=&{\frac{1}{2}}\int_0^{\infty}\widehat{H_t(\partial_j^2f)}(\xi) \, dt\nonumber, 
\end{eqnarray}
where for any function $f\in C_0^{\infty}(\bR^n)$,  $H_tf(x)$ is the heat semigroup with generator 
$-\frac{1}{2}\Delta$ acting on the function $f$. That is, the convolution 
of $h_t$ and  $f$, 
$$
H_t(f)(x)
=\int_{\bR^{n}}h_{t}(x-y)f(y)\, dy. 
$$
\normalsize
Thus 

\begin{eqnarray}\label{rieszheat1}
R_j^2(f)(x)={\frac{1}{2}}\int_0^{\infty}{\partial_j^2H_t f}(x) \, dt
&=&{\frac{1}{2}}\int_0^{\infty}{H_t(\partial_j^2f)}(\xi) \, dt,\\\nonumber\\
&=&- {\partial_j^2}(-\Delta)^{-1}f(x)\nonumber \\\nonumber\\
&=&-(-\Delta)^{-1} {\partial_j^2f}(x)\nonumber .
\end{eqnarray}
Similarly, 
\begin{eqnarray*}
R_j R_k(f)(x)=\frac{1}{2}\int_0^{\infty} {\partial_{j k}^2H_t f}(x) \, dt
&=&\frac{1}{2}\int_0^{\infty}H_t\left(\partial_{j k}^2 f\right)(x) \, dt\\\\
&=&-\partial_{j k}^2(-\Delta)^{-1}f(x)\\\\
&=&-(-\Delta)^{-1}\partial_{j k}^2f(x).
\end{eqnarray*}
If we restrict ourselves to $\bR^2=\bC$, we can verify that 
\begin{eqnarray}\label{BAlaplacian}
B(f)(z)=\frac{1}{2}\int_0^{\infty} \partial^2 H_tf (z) \, dt=
-\left(\partial^{2} (-\Delta)^{-1}f(z)\right)
\end{eqnarray}
where here and for the rest of the paper we use the notation 
\begin{equation}\label{partials-1}
\partial f= {\partial f\over
 \partial x_1}-i\ {\partial f\over \partial x_2}, \,\,\,\,\,\,\,\,\,  \overline\partial f= {\partial f\over \partial x_1}+i\ {\partial f\over \partial x_2}
\end{equation}
for the Cauchy--Riemann operators $\partial$ and $\overline\partial$. (Note that our notation for $\partial$ and $\overline\partial$. differs from the ``standard notation" by a factor of a half.)
By Plancherel's theorem, for all $f, g\in C_0^{\infty}(\bR^n)$, 
\begin{eqnarray}\label{littlewood1}
&&\int_{0}^{\infty}\int_{\bR^n}\partial_j H_t f(x)\partial_j H_t({g})(x) dx dt\\
&=&\int_{0}^{\infty}\int_{\bR^n}\widehat{\partial_j
 H_t f}(\xi)\overline{\widehat{\partial_j H_t({g})}}(\xi) d\xi dt\nonumber\\
 &=&\int_{0}^{\infty}\int_{\bR^n}(4\pi^2 \xi_j^2)e^{-2\pi^2 t|\xi|^2}e^{-2\pi^2 t|\xi|^2}\hat{f}(\xi)\, \overline{\widehat{g}}(\xi)\,d\xi dt\nonumber\\
&=&\int_{0}^{\infty}\int_{\bR^n}(4\pi^2\xi_j^2)e^{-4\pi^2 t|\xi|^2}\hat{f}(\xi)\, \overline{\widehat{g}}(\xi)\,d\xi dt\nonumber\\
&=&-\int_{\bR^{n}} \widehat{R_{j}^{2}f}(\xi) \, \overline{\widehat{g}}(\xi) \,d\xi
=-\int_{\bR^{n}} R_{j}^{2}f(x) \, {g}(x) \,dx\nonumber.
\end{eqnarray}

\begin{remark}
It is interesting to note here that this Littlewood-Paley identity can also be proved by first integrating  by parts and then using the self-adjointness of $H_t$, the semigroup property  $H_t$ on $L^2$, and the representation of $R_j^2$ given in (\ref{rieszheat1}).  A similar identity can be written down for $R_jR_k$.  This then leads to more general expressions for second order Riesz transforms on manifolds and semigroups where one does not have the Fourier transform so easily available.  For more on this, we refer the reader to the forthcoming paper \cite{BanBau1} where in addition to second order Riesz transforms on manifolds, results are given also for second order Riesz transforms of Schr\"odinger operators. 
\end{remark}

The following is the key Littlewood-Paley inequality of Nazarov and Volberg \cite{NazVol2};  see also Dragi\v cevi\'c and Volberg \cite{DraVol1, DraVol4, DraVol2, DraVol3}. We state it here as in \cite{NazVol2} for $\bR^2$ but a similar inequality holds in $\bR^n$.

\begin{theorem}\label{NazVoldual}
For all $f, g\in C_{0}^{\infty}(\bR^2)$, 
$2\leq p<\infty$, 
\begin{eqnarray}\label{littlewood2}
\int_0^{\infty} \int_{\bR^2}\Big[\Big|\partial_1H_tf(x)
\Big|\Big|\partial_1H_t(g)(x)\Big| &+& \Big|\partial_2H_tf(x)
\Big|\Big|\partial_2H_tg(x)\Big|\Big]dx \, dt\nonumber\\\nonumber\\
&\leq & (p^{*}-1) \|f \|_{p} \|g\|_{q}.
\end{eqnarray}
\end{theorem}
We remind the reader that when comparing the formulas in this paper to those of Nazarov and Volberg, and Dragi\v cevi\'c and Volberg, one needs to keep in mind that in this paper we use the  heat kernel for $\frac{1}{2}\Delta$, while they use the heat kernel for $\Delta$.  This changes the formulas by a factor of $2$.

Nazarov and Volberg prove (\ref{littlewood2}) by applying Green's theorem to the function 
$$
b(x, t)=B\left(|H_t{f}(x)|^{p}, |H_t{g}(x)|^{p}, H_t{f}(x), H_t{g}(x)\right),
$$
where $B(X, Y, \xi, \eta)$ is a ``Bellman'' function for the domain 
$$D_{p}=\{(X, Y, \xi, \eta)\in \bR\times 
\bR\times\bR^{2}\times\bR^{2}: \|\xi\|^{p}<X, \|\eta\|^{p}<Y\}.$$

The construction (existence) of this Bellman function ultimately depends on 
Burkholder's inequality for the Haar system, hence their approach, while avoiding the stochastic integrals used in \cite{BanWan3} and bringing new insights and ideas into the problem, is not independent of martingales.  

Duality, (\ref{littlewood1}) and (\ref{littlewood2}) led Nazarov and Volberg \cite{NazVol2} to the following 
\begin{theorem}
\begin{equation}\label{Rieszonehalf}
\|(R^2_1-R_{2}^{2})f\|_p\leq (p^*-1)\|f\|_p, \,\,\,\,\, \|2R_1R_2f\|_p\leq (p^*-1)\|f\|_p, 
\end{equation}
for any $1<p<\infty$.  
\end{theorem}
From this Nazarov and Volberg obtained the following improvement of (\ref{BA4}).
\begin{corollary}
\begin{equation}\label{BA2} 
\|B\|_{p}\leq 2(p^*-1), \,\,\,\,\, 1<p<\infty.
\end{equation}
\end{corollary}

\subsection{Space-time Brownian motion and the operators $\mathcal{S}_{A}$}\label{S_A}
The martingale transform techniques of \cite{BanWan3} based on the
background radiation process of Gundy and Varopoulos do not give the
Nazarov-Volberg estimate (\ref{BA2}). The Nazarov-Volberg heat kernel
approach suggests, however, that one should look into the possibility
of deriving a ``probabilistic" Littlewood-Paley approach to the second
order Riesz transforms based on space-time Brownian motion. As it turns
out, such approach is successful and leads to improvements on
(\ref{BA2}) and to many other applications.  This approach, which first appeared in \cite{BanMen1}, is what we now discuss.

We again follow the procedure depicted in the diagram (\ref{diagram}). Fix a large $T>0$ and consider the space-time Brownian motion $Z_t=(B_t, T-t)$, $0<t\leq T$, started at  $(x, T)\in \bR^n \times (0, \infty)$, where $B_t$ is Brownian motion in $\bR^n$ with initial distribution the Lebesgue measure $m$. For the rest of the paper, unless otherwise explicitly stated, $n\geq 2$. Thus the process $Z_t$ starts on the hyperplane $\bR^n\times \{T\}$ with
initial distribution $m \otimes
\delta_T$. $P_x$ and $E_x$ will denote the probability and expectation for processes starting at the point $x$ and  let
     $P^T$ denote the ``probability" measure associated with the
process with initial distribution initial distribution $m \otimes
\delta_T$ and denote  by $E^T$ the
corresponding expectation. For any function $f\in C_0^{\infty}(\bR^n)$, we  consider the heat extension function 
$$V_f(x, t)=H_tf(x)=E_x[f(B_t)].$$
 It then follows that
 $$V_f(Z_t)=V_f(B_t, T-t),\,\,\,\,\,0<t<T$$
  is a martingale (a ``heat" martingale). 
The It\^o formula gives 
\begin{equation}\label{itomart}
V_f(Z_t)-V_f(Z_0)=\int_0^t\nabla_x V_f(Z_s)\cdot dB_s,\hskip.4cm  0<t\leq T,
\end{equation}
where 
$$\nabla_x V_f(\cdot)=\Big(\partial_1
V_f(\cdot),\partial_2 V_f(\cdot)\dots, \partial_d V_f(\cdot)\Big)
$$ 
denotes the gradient of $V_f$ in the $x$-variable.  That is, what is referred to in some places as the horizontal gradient. 
Furthermore, if $A(x, t)$ is a $n\times n$ matrix-valued function defined for $(x, t)\in \bR^{n+1}_{+}$ (with real or complex entries), 
\begin{equation}\label{itotransform}
(A*V_f)_t = \int_0^t [A(Z_s)\nabla_x V_f(Z_s)]\cdot dB_s, \hskip.4cm 0<t\leq T,
\end{equation}
is the martingale transform of  $V_f(Z_t)$.  As before the martingale $(A*V_f)_t$ is 
differentially subordinate to the martingale $\|A\|\,V_f(Z_t)$.
We notice that  for any $f$ and $t>0$,  (recall $h_t$ is as in (\ref{heat})), 
\begin{eqnarray}\label{lebesguedis}
 \int_{\bR^n}E_{x}f(B_t)\, dx&=&\int_{\bR^n}\int_{\bR^n}f(x) h_t(x-y) dxdy\\
 &=&\int_{\bR^n} f(x)\, dx.\nonumber
\end{eqnarray}
This gives that the distribution of the random variable $B_T$ under $P^T$ is the Lebesgue measure and this turns the $L^p$-norm of the random variables $f(B_T)$ into the $L^p(\bR^n)$ norm of the function $f$.   Defining the projection of $A*V_f$ on
$\bR^n$ by 
\begin{eqnarray}\label{heatproj1}
\mathcal{S}_{A}^T f(x) &=& E^T\left[ \,(A*V_{f})_T\,\big|\,
Z_T=(x,0)\, \right]\\
&=&E^T\left[ \,(A*V_{f})_T\,\big|\,
B_T=x\, \right]\nonumber,
\end{eqnarray}
 we obtain the family of operators 
$$\{\mathcal{S}_{A}^T, \hskip.1cm  A\in \mathcal{M}^{n\times n},\hskip.1cm T>0\}$$ 
defined on $C_0^{\infty}(\bR^n)$, where as before $\mathcal{M}^{n\times n}$ is the collection of all $n\times n$ matrix-valued functions on $\bR^{n+1}_{+}$ with real or complex entries and with $\|A\|<\infty$. 

The following Theorem  is proved exactly as Proposition 2.2 in  \cite{BanMen1}.  The reader should compare this with Theorem  \ref{prop1} in \S3.2.
 
\begin{theorem}\label{prop2} For all $f\in C_0^{\infty}(\bR^n)$, $\mathcal{S}_A^{T}f\to \mathcal{S}_Af$ in $L^2$, as $T\to\infty$, and for all $f, g\in C_0^{\infty}(\bR^n)$ and all $n\times n$ matrix-valued functions $A(x, t)$ in $\bR^{n+1}$ with $\|A\|<\infty$, we have
\begin{eqnarray}\label{prop2eq}
\lim_{T\to\infty}\int_{\bR^n} \mathcal{S}_A^Tf(x)\, g(x)\, dx&=&\int_{\bR^n} \mathcal{S}_Af(x)\, g(x)\, dx\nonumber\\
&=&\int_0^{\infty}\int_{\bR^n} \left[A(x, t)\nabla_x  V_f(x, t)\right]\cdot \nabla_x V_g(x, t)\, dx\, dt.
\end{eqnarray}
\end{theorem}

By Corollary \ref{transforms1}, for any $1<p<\infty$ and any $f\in C_0^{\infty}(\bR^n)$,
\begin{eqnarray}\label{STbound}
 \Big\|\mathcal{S}_A^Tf\Big\|_p&\leq& \|A\| (p^*-1)\Big\|\int_0^T\nabla_x V_f(Z_s)\cdot dB_s\Big\|_p\nonumber\\\nonumber\\
 &= &\|A\| (p^*-1)\Big\|V_f(Z_T)-V_f(Z_0)\Big\|_p,\nonumber\\
 &\leq & 2\|A\| (p^*-1)\|f\|_p,
\end{eqnarray}
which is valid for all $T>0$ and the right hand side does not depend on $T$.  
With a more careful argument it is shown in \cite[p.~984]{BanMen1} 
that $$\lim_{T\to\infty}\Big\|V_f(Z_T)-V_f(Z_0)\Big\|_p\leq \|f\|_p$$ which then gives that 
\begin{equation}\label{Sbound}
 \Big\|\mathcal{S}_Af(x)\Big\|_p \leq \|A\| (p^*-1)\|f\|_p,\,\,\,\,\,\,1<p<\infty. 
\end{equation}

In fact, from the arguments in \cite{BanMen1} and  Corollary \ref{transforms1}, we obtain the same result as Theorem \ref{TAproj} which we restate here to summarize.

\begin{theorem}\label{SAproj} Let $\{f_i\}_{i=1}^\infty$ be a sequence of functions in $C_0^{\infty}(\bR^n)$ and 
$\{A_i(x, y)\}_{i=0}^{\infty}$ a sequence of $n\times n$ matrix-valued functions  such that $\|A_i\|\leq M$, for all $i$. Then,  for all $1<p<\infty$,  
\begin{equation}\label{SAproj1}
\Bigg\|\left(\sum_{i=1}^\infty|\mathcal{S}_{A_i} f_i (x)|^2\right)^{1/2}\Bigg\|_p \leq
(p^*-1)M
\Bigg\|\left(\sum_{i=1}^\infty |f_i|^2\right)^{1/2}\Bigg\|_p. 
\end{equation}
Furthermore, suppose $\calA=\{A_i(x,t)\}_{i=1}^\infty$ is a sequence of $n\times n$
matrix-valued functions and $f\in C_0^{\infty}(\bR^d)$.   Then, 
\begin{equation}\label{SAproj2}
\Bigg\|\left(\sum_{i=1}^\infty|\mathcal{S}_{A_i}f(x)|^2\right)^{1/2}\Bigg\|_p\leq
(p^*-1) \|\calA\| \|f\|_p,
\end{equation}
where $\|\calA\|<\infty$ and defined as in Corollary \ref{transforms1}.
\end{theorem}

As in \S3.1, Theorem \ref{prop2} and  inequality (\ref{Sbound}) give 
 
\begin{corollary}\label{nazvolheatcor} For all $f, g\in C_0^{\infty}(\bR^n)$ and all $n\times n$ matrix-valued 
functions $A(x, t)$ in $\bR^{n+1}_{+}$ with $\|A\|<\infty$,  we have
\begin{equation*}
\Bigg|\int_0^{\infty}\int_{\bR^n} \left[A(x, t)\nabla_x V_f(x, t)\right]\cdot \nabla_x V_g(x, t)\, dx dt\Bigg|\leq \|A\|\,(p^*-1)\|f\|_p\|g\|_q,
\end{equation*}
for all $1<p<\infty$.
\end{corollary}
 
 The choice of the matrix

$$
A(x, t)=\frac{\nabla_xV_f(x, t) \otimes\nabla_xV_g(x, t)}{|\nabla_xV_f(x, t)||\nabla_xV_g(x, t)|}
$$
gives 

\begin{corollary}\label{nazvoldualtybymart} For all $f, g\in C_0^{\infty}(\bR^n)$, 
\begin{equation}\int_0^{\infty}\int_{\bR^n} \Big|\nabla_xV_f(x, t)\Big|\, \Big|\nabla_xV_g(x, t)\Big|\, dxdt\leq (p^*-1)\|f\|_p\|g\|_q,
\end{equation}
for all $1<p<\infty$. 
\end{corollary}
This is the Dragi\v cevi\'c-Volberg inequality proved in \cite{DraVol2} for $\bR^2$ using Bellman functions; see their Theorem 1.4. This inequality implies the Nazarov-Volberg inequality (\ref{littlewood2}). We also refer the reader to  \cite{Hyt3} and \cite{PetSlaBre1} for an even more general version of this inequality for functions $f:\bR^n\to \bR^m$ which also follow from these methods.

We now give the outline of the proof of the equality (\ref{prop2eq}) in Theorem \ref{prop2}; the reader can see \cite{BanMen1} for the full details.    The first equality below is the fact that the distribution of $B_T$ is the Lebesgue measure, the second is the It\^o formula (\ref{itomart}) applied to $V_g$.  

\begin{eqnarray*}
\int_{\bR^n} \mathcal{S}_A^Tf(x)g(x)dx&=
&E^T\left(g(B_T)\int_0^T \Big[A(Z_s)\nabla_xV_f(Z_s)\Big]\cdot dB_s\right)\\
&=& E^T\left(V_g(Z_0)\int_0^T \Big[A(Z_s)\nabla_xV_f(Z_s)\Big]\cdot dB_s\right)\\
&+&E^T\left(\int_0^T \nabla_xV_g(Z_s)\cdot dB_s
\int_0^T \Big[A(Z_s)\nabla_xV_f(Z_s)\Big]\cdot dB_s\right).
\end{eqnarray*}

We now observe that the first quantity in the last equality is actually zero.  Indeed, by the definition of $E^T$, this quantity is simply  
\begin{eqnarray*}
&&\int_{\bR^n} E_x\left(H_T {g}(B_0)\int_0^T \Big[A(Z_s)\nabla_xV_f(Z_s)\Big]\cdot dB_s\right) dx\\
&=& \int_{\bR^n} H_T {g}(x) E_x\left(\int_0^T \Big[A(Z_s)\nabla_xV_f(Z_s)\Big]\cdot dB_s\right) dx\\
&=&0,
\end{eqnarray*}
by the martingale property of the stochastic integral. 
On the other hand, the It\^o isometry gives that 
\begin{eqnarray}\label{uniformbound}
&&\Big|E^T\left(\int_0^T \nabla_xV_g(Z_s)\cdot dB_s
\int_0^T \Big[A(Z_s)\nabla_xV_f(Z_s)\Big]\cdot dB_s\right)\Big|\\
&=&\Big|\int_0^T E^T\left(\left[A(Z_s)\nabla_xV_f(Z_s)\right]\cdot \nabla_xV_g(Z_s)\right)ds\Big|\nonumber\\&=&\Big|\int_0^T \int_{\bR^n} E_x\left[A(B_s, T-s)\nabla_xV_f(B_s, T-s)\right]dx\,ds\Big|\nonumber\\
&=&\Big|\int_0^T \int_{\bR^n} \left[A(x, T-s)\nabla_xV_f(x, T-s)\right]\cdot \nabla_xV_g(x, T-s)\,dx\,ds\Big|\nonumber\\
&=&\Big|\int_0^T \int_{\bR^n} \left[A(x, t)\nabla_xV_f(x, t)\right]\cdot \nabla_xV_g(x, t)\,dx\,dt\Big|.\nonumber
\end{eqnarray}

By H\"older's inequality and (\ref{STbound}), we have 
\begin{eqnarray*}
&&\Big|E^T\left(\int_0^T \nabla_xV_g(Z_s)\cdot dB_s
\int_0^T\Big[ A(Z_s)\nabla_xV_f(Z_s)\Big]\cdot dB_s\right)\Big|\\
&\leq& 4(p^*-1)\|A\| \|f\|_p\|g\|_q,
\end{eqnarray*}
independent of $T$. 
Thus the right hand side of (\ref{uniformbound}}) is uniformly bounded independent of $T$. This proves that 
\begin{eqnarray}\label{dual1}
&&\lim_{T\to\infty}\int_{\bR^n} \mathcal{S}_A^Tf(x)\, g(x)\, dx\nonumber \\
&=&\int_0^{\infty}\int_{\bR^n} \left[A(x, t)\nabla_xV_f(x, t)\right]\cdot \nabla V_g(x, t)\, dx\, dt,
\end{eqnarray}
as asserted by (\ref{prop2eq}). 

For these operators we also have the analogue of Remark \ref{convolutionpoisson} as well as a problem similar to Problem \ref{weak1-poisson}. 
\begin{remark}
As in the case of the operators $T_A$ obtained by projections of background radiation martingale transforms, the operators $\mathcal{S}_A$ can also be written as convolutions operators, this time in terms of the heat semigroup.  More precisely,  for any $n\times n$ matrix $A(x, t)$ with $\|A\|<\infty$ we have 
$$
\mathcal{S}_Af(x)=\int_{\bR^n}K_A(x-\tilde{x})f(\tilde{x})d\tilde{x},
$$
where 
$$
K_A(x)=\int_0^\infty \int_{\bR^n} A(\overline x, t)\nabla_xh_t(\overline{x})\cdot \nabla_x h_t(x-\overline{x})d\overline{x} dy.
$$
\end{remark}

When $A=A(t)$ is independent of $x$, it follows from (\ref{dual1}) and Plancherel's theorem that 
\begin{equation}\label{fourier-A(t)}
\widehat{\mathcal{S}_Af}(\xi)=\left(4\pi^2 \int_{0}^{\infty} [A(t)\xi]\cdot \xi e^{-4\pi^2 t|\xi|^2}\,dt\right) \widehat{f}(\xi).
\end{equation}
Furthermore, if $A$ is constant we have  
\begin{equation}\label{fourierA}
\widehat{\mathcal{S}_Af}(\xi)=\frac{A\xi\cdot\xi}{|\xi|^2}\,\widehat{f}(\xi).
\end{equation}

\subsection{Space-time Brownian motion and the Beurling-Ahlfors operator} 

If  we now consider $\bR^2$ and take the matrix 
\begin{equation}\label{Riesz1sq}
A_{11}=\left[\begin{array}{cc}1 & 0 \\ 0 & -1\end{array}\right],
\end{equation}
it follows from  (\ref{littlewood1}) and (\ref{prop2eq}) (or simply from (\ref{fourierA}) ) that 

$$\mathcal{S}_Af=R_2^2f-R_1^2f.$$
With
\begin{equation}\label{Riesz12sq}
A_{12}=\left[\begin{array}{cc}0 & -1 \\ -1& 0\end{array}\right],
\end{equation}
we have 

$$\mathcal{S}_Af=2R_1R_2f.$$  
Since both of these matrices have norm 1, the Nazarov-Volberg Theorem \ref{Rieszonehalf} and Corollary \ref{BA2} follow from (\ref{Sbound}).  In the same way, if 
$$A=\left[\begin{array}{cc}a & -b \\ -b & -a\end{array}\right],$$
with $a, b\in\bR$\, such that $a^2+b^2\leq 1$, we get the following bound (proved in \cite{BanMen1}) which incorporates both bounds  in (\ref{Rieszonehalf}).

\begin{corollary}  For all $f\in C_0^{\infty}(\bR^n)$ and $1<p<\infty$,
\begin{equation}\label{rieszab}
\|a(R_2^2-R_1^2)f+2bR_1R_2f\|_p\leq (p^*-1)\|f\|_{p} ,
\end{equation}
\end{corollary}

The matrix 
\begin{equation}\label{BAmatrix}
\calB=A_{11}+iA_{12}=\left[\begin{array}{cc}1 & -i \\ -i & -1\end{array}\right]
\end{equation}
which gives $\mathcal{S}_{\calB}f=Bf$ has norm $\|\calB\|=2$ when acting on vectors with complex entires.  That is,
$\|\calB\|=\sup \{\|\calB z\|_{\bC^2}: z\in \bC^2, \|z\|_{\bC^2} \leq 1\}=2$.  On the other hand, 
$\|\calB\|=\sup \{\|\calB x\|_{\bR^2}: x\in \bR^2, \|x\|_{\bR^2} \leq 1\}=\sqrt{2}$.  This then gives that 
\begin{equation}\|Bf\|_p\leq \sqrt{2}(p^*-1)\|f\|_p,
\end{equation}
  for smooth functions of compact support with $f:\bC\to\bR$.  However, unlike the Riesz transforms which take real valued functions to real valued functions, the Beurling-Ahlfors operator takes real valued functions to complex valued functions. Hence, we cannot convert this estimate to cover any function $f:\bC\to\bC$ with the usual Hilbert space techniques which apply to the Riesz transforms.

With $B=R_2^2-R_1^2+2iR_1R_2$, we denote by $\Re(Bf)$ the real part of the complex-valued function $Bf$. Note that for us this is not the same as what is called in some papers ``the real part of $B$" which in our notation is the operator $R_2^2-R_1^2$.  We also introduce the conjugation operator $\tau(f)=\bar{f}$ defined for any function $f:\bC\to\bC$.  One checks easily that $\tau B\tau=\overline B$, where $\overline B=R_2^2-R_1^2-2iR_1R_2$.  We summarize various estimates from  \cite{BanMen1} and improvements from those in  \cite{BanLin1} which can be obtained using the space-time Brownian motion in the next two theorems.  (The inequalities in Theorem \ref{BA-essen} and the connections with Ess\'en's inequality were first discussed in \cite[pp 228--229]{BanLin1}.)  We leave the computation of the norms of the relevant matrices to the interested reader.

\begin{theorem} Let $f\in C_0^{\infty}(\bR^n)$ and $1<p<\infty$.  
Suppose $\{v_j\}_{j=1}^{n}$ is a sequence of scalars such that $|v_j|\leq 1$ for all $j$.  Then \\
\begin{enumerate}
\item [(i)] $
\big\|\sum_{j=1}^d v_jR^2_jf\big\|_p\leq (p^*-1)\|f\|_p,$\\
\end{enumerate}
\noindent Furthermore, if $n$ is even, say $n=2m$, then\\
\begin{enumerate}
\item [(ii)] $
\big\|\sum_{j=1}^{m}v_{2j}R_{2j-1}R_{2j}f\big\|_p\leq \frac{1}{2}(p^*-1)\|f\|_p. 
$
\end{enumerate}
\end{theorem}

\begin{theorem}\label{BA-essen}
\noindent Let $f\in C_0^{\infty}(\bR^2)$ and $1<p<\infty$. Then \\
\begin{enumerate}
\item[(i)] $\big\|\sqrt{|Bf|^2+|{\tau B}\tau f|^2+|\sqrt{2}f|^2}\big\|_p\leq \sqrt{6}(p^*-1)\|f\|_p$,\\ 
\item[(ii)] $\big\|\sqrt{|Bf|^2+|{\tau B}\tau f|^2}\big\|_p\leq 2(p^*-1)\|f\|_p,$ \\
\item[(iii)] $\big\|\Re(Bf)\big\|_p\leq \sqrt{2} (p^*-1)\|f\|_p$.
\end{enumerate}
\end{theorem}

From (i), (ii) in the previous Theorem and Minkowski inequality it follows that

\begin{equation}\label{essen-minkowski-1}
\|Bf\|_p^2 +\|B\overline f\|_p^2+2\|f\|_p^2 \leq 6(p^*-1)^2\,\|f\|_p^2, \quad 1<p\leq 2,
\end{equation}
and 
\begin{equation}\label{essen-minkowski-2}
\|Bf\|_p^2 +\|B\overline f\|_p^2 \leq 4(p^*-1)^2\,\|f\|_p^2, \quad 1<p\leq 2, 
\end{equation}
for all $f\in C_0^{\infty}(\bR^2)$.  From this we get the following interesting looking corollary.

\begin{corollary} For all $f\in C_0^{\infty}(\bR^2)$ and $1<p\leq 2$, at least one of the following inequalities holds:
\begin{equation}
\|Bf\|_p \leq \sqrt{2}(p^*-1)\,\|f\|_p, 
\end{equation}
\begin{equation}
\|B\overline f\|_p \leq \sqrt{2}(p^*-1)\,\|f\|_p. 
\end{equation}
\end{corollary}
Note that a similar conclusion can be made with the bound $\sqrt{3(p^*-1)^2-1}$ if we use (\ref{essen-minkowski-1}) instead of (\ref{essen-minkowski-2}). 

Recall that the Ess\'en inequality from (\ref{Essen}) for the Hilbert transforms reads

 \begin{equation}
 \Big\|\sqrt{|Hf|^2+|f|^2}\Big\|_p\leq \csc(\frac{\pi}{2p^{\ast}}) \|f\|_p, \hskip.5cm 1 < p < \infty.
 \end{equation}
 Since the Hilbert transform anti-commutes with the dilation operator $(\delta_{-1} f)(x)= f(-x)$, 
the Ess\'en inequality is trivially equivalent to 
\begin{equation*}
 \Big\|\sqrt{|Hf|^2+|\delta_{-1}H\delta_{-1}f|^2 +|\sqrt{2}f|^2}\Big\|_p\leq \sqrt{2}\csc(\frac{\pi}{2p^{\ast}}) \|f\|_p, \hskip.5cm 1 < p < \infty
 \end{equation*}
and this motivates the inequality (i) in Theorem \ref{BA-essen}.  We note that for $p=2$ the Fourier transform gives the inequality (i) with $\sqrt{6}$ replace by $\sqrt{4}=2$, instead.  Thus, as usual, our inequality is not sharp even at $p=2$.  The obvious conjecture is that the best bound in (i) should be $2(p^*-1)$ and $\sqrt{2}(p^*-1)$ in (ii).  For better bounds related to Theorem \ref{BA-essen}, see \cite{Jan0}.

In \cite{DraVol2}, Dragi\v cevi\'c and Volberg reproved the inequality (\ref{rieszab}) with $a=\cos(\theta)$ and $b=\sin(\theta)$, $\theta\in [0, 2\pi]$ and use it to obtain the following improvements on the bounds for $B$ on both real and complex valued functions.  Their results is as follows. 
\begin{equation}\label{BAtau1}
\|Bf\|_p\leq \sigma(p)\,p\|f\|_p \, \,\,\,\,\, \, f:\bC\to\bR
\end{equation}
and 
\begin{equation}\label{BAtau2}
\|Bf\|_p\leq \sqrt{2}\,\sigma(p)\,p\|f\|_p,\,\,\,\,\, \, f:\bC\to\bC,
\end{equation}
where 
$$
\sigma(p)={\left(\frac{1}{2\pi}\int_0^{2\pi}|\cos(\theta)|^p
d\theta\right)}^{-\frac{1}{p}}, 
$$
and $\sigma(p)\to 1$ as $p\to\infty$. 

 In the \cite{GeiSmiSak1}, Geiss, Montgomery-Smith and  Saksman used arguments similar to those used  by Bourgain in \cite{Bou1} to show that Burkholder's UMD property is  equivalent to the boundedness of the Hilbert transform to prove that the $L^p(\bR^n)$-norms of  $2R_jR_k$,  and $R_j^2-R_k^2$, $j\not= k$, are bounded below by $(p^*-1)$.  Together with the above upper bound estimates this shows that 
 $$\|2R_jR_k\|_{p}=(p^*-1)$$
 and 
 $$\|R_j^2-R_k^2\|_{p}=(p^*-1)$$
 for $j\not =k$. 
(The question of the computation of the norm of these operators was first raised in \cite[p. 260]{BanLin1}; see \S5.2 below for more on this.)   

This somewhat surprising result gives the first examples of  singular integrals whose $L^p$ norms are exactly those of martingale transforms. 
The result also shows that in the plane the real and imaginary components of the Beurling-Ahlfors operator $B$  have norm equal to $(p^*-1).$  The proof in \cite{GeiSmiSak1} adapts to other combinations of $R_j$ and $R_k$ (and much, much, more; see \cite{BanOse2}).   A different method for proving the sharpness of these bounds which avoids the Bourgain method altogether based on so called ``laminates" is presented  in \cite{NicVol1}; see also \cite{VasVol2}.  
 
 In \cite{Cho1}, K.P. Choi uses the Burkholder method to identify the best constant in the martingale transforms where the predictable sequence $v_k$ takes values in $[0, 1]$ instead of $[-1, 1]$, as in the work of Burkholder. While this constant is not as explicit as the $p^*-1$ constant of Burkholder, one does have a lot of information about it. 
 \begin{theorem}\label{Choi} Let $f=\{f_n, n\geq 0\}$ be a real-valued martingale with difference sequence $d=\{d_k, k\geq 0\}$.   Let $v\ast f$ be the martingale transform of $f$ by a predictable sequence $v=\{v_k, k\geq 0\}$ with values in $[0,1]$.  Then 
\begin{equation}\label{choi}
\|v\ast f\|_p\leq c_p\|f\|_p, \,\,\, 1<p<\infty,
\end{equation}
with the best constant $c_p$ satisfying
$$
c_p=\frac{p}{2}+ \frac{1}{2}\log\left(\frac{1+e^{-2}}{2}\right) +\frac{\alpha_2}{p}+\cdots
$$
where $$\alpha_2=\left[\log\left(\frac{1+e^{-2}}{2}\right)\right]^2+\frac{1}{2}\log\left(\frac{1+e^{-2}}{2}\right)-2\left(\frac{e^{-2}}{1+e^{-2}}\right)^{2}. $$
 \end{theorem}

As in the proof of Burkholder's inequalities, Choi's proof adapts to stochastic integrals and  in particular it follows from his proof and the space-time Brownian motion (heat martingale) representation used in this section for second order Riesz transforms that an upper bound for $\|R_j^2\|_p$ is the constant in Choi's Theorem \ref{Choi}.  However, much more interesting is the fact  that this bound is sharp.  This  and more is proved in the forthcoming paper \cite{BanOse2} of Os\c{e}kowski and the author. In particular, the following theorem  is a special case of the results in \cite{BanOse2}.  

\begin{theorem}
Let $J$ be a nonempty subset of $\{1,2, \dots, d\}$, $J\neq \{1,2, \dots, d\}$.  Then 
\begin{equation}\label{Oscbound1}
\Big\|\sum_{j\in J}R_j^2\Big\|_p=c_p, \hskip.3cm 1<p<\infty. 
\end{equation}
\end{theorem}
This result once again shows that singular integrals (as ``tame" as they appear to be in comparison to martingales) can achieve the same norms as those for martingales.   The proof  is an adaptation of the techniques of Geiss, Montgomery-Smith and  Saksman. 

As observed by Choi, 
\begin{equation}
c_p\approx \frac{p}{2}+\frac{1}{2}\log\left(\frac{1+e^{-2}}{2}\right),
\end{equation}
with this approximation becoming better for large $p$. 
It also follows trivially from Burkholder's inequalities that  (even without knowing explicitly the best constant $c_p$) 
\begin{equation}\label{choibound}
\max\left(1, \frac{p^*}{2}-1\right)\leq c_p\leq \frac{p^*}{2}
\end{equation}
In the same way,  the fact that the best constant in (\ref{Oscbound1}) has the same bounds follows trivially even without knowing its value or that it is the same as the Choi's constant.  For example, in $\bR^2$, $\|R_1^2-R_2^2\|_p=p^*-1$ and $R_1^2+R_2^2=-I$ give the bound in (\ref{choibound}) for the best constant in  
(\ref{Oscbound1}). 

It follows from the result of Geiss, Montgomery-Smith and  Saksman that for the family of operators $\{\mathcal{S}_{A}, \,\, A\in \mathcal{M}_{n\times n}\}$ the bound 
 \begin{equation}\label{BAend}
 \|\mathcal{S}_{A}\|_{p}\leq \|A\|(p^*-1) 
 \end{equation}
 cannot be improved in general.  Thus, via this general approach it is not possible to improve the bound 2$(p^*-1)$ from martingale inequalities without a more careful study of the structure of the martingale transform $\calB*X$ arising from the matrix in (\ref{BAmatrix}).   This possibility (already observed in \cite[p.~599]{BanWan3}) motivated the results in \cite{BanJan1} and leads to an improvement on the norm bound for $B$.  The idea is to find this additional structure and apply Corollary \ref{cor1}.  
 Again, set $f=f_1+if_2$. We consider the martingale 
$$X_t=\int_0^t \nabla_xV_f(Z_s)\cdot dB_s$$ and its martingale transform 
$$\calB *X_t=\int_0^t \calB\nabla_x V_f(Z_s)\cdot dB_s$$
where $\calB$ is the Beurling-Ahlfors matrix as in (\ref{BAmatrix}). We can easily check that  $\calB*X_t=(Y_t^1, Y_t^2)$ is the $\bR^2$--valued martingale with
$$
Y_t^1=\int_0^t\left( A_{11}
\nabla_xV_{f_1}(B_s) - A_{12}\nabla_x V_{f_2}(B_s)\right)\cdot dZ_s
$$
and 
$$
Y_t^2=\int_0^t\left( A_{11}\nabla_x V_{f_2}(B_s) + A_{12}\nabla_x V_{f_1}(B_s)\right)\cdot dZ_s.
$$
From here, we easily verify that $\calB*X_t$ is an $\bR^2$--valued conformal martingale as in Definition \ref{confmar} in \S2.2.  Indeed, 

\begin{eqnarray*}
\langle Y^1\rangle_t &=& \int_0^t \left|\left( A_{11}
\nabla_x V_{f_1}(B_s) - A_{12}\nabla_x V_{f_2}(B_s)\right)\right|^2ds \\
&=& \int_0^t\big({(\partial_1 V_{f_1}(B_s)-\partial_2 V_{f_2}
(B_s))}^2 +{(-\partial_2 V_{f_1}(B_s)-\partial_1 V_{f_2}
(B_s))}^2\big)ds \\ 
&=& \int_0^t {|\overline{\partial}V_f(B_s)|}^2 ds,
\end{eqnarray*}
where 
\begin{eqnarray*} 
|\overline{\partial} V_f|^2 = |(\partial_1+i\partial_2)V_f|^2
= {|\nabla_x V_f}^2- 2(\partial_1V_{f_1}\partial_2V_{f_2}
-\partial_2V_{f_1}\partial_2V_{f_2}).
\end{eqnarray*}
Next, computing the same for $Y^2$ we find that  
\begin{eqnarray*}
\langle Y^2\rangle_t &=& \int_0^t \left|\left(A_{11}
\nabla_x V_{f_1}(B_s) + A_{12}
\nabla_x V_{f_2}(B_s)\right)\right|^2ds \\
&=& \int_0^t\big({(\partial_2 V_{f_1}(B_s)+\partial_1 V_{f_2}
(B_s))}^2  +{(\partial_1 U_{f_1}(B_s)-\partial_2 U_{f_2}
(B_s))}^2\big)ds \\ 
&=& \int_0^t {|\overline{\partial}V_f|(B_s)|}^2 ds. 
\end{eqnarray*}
Hence $\langle Y^1 \rangle_t=\langle Y^2\rangle_t $, for all $t\geq 0$.  

In the same way, one verifies (see \cite{BanJan1}) that 
 $\langle Y^1,Y^2\rangle_t=0$ for all $t\geq 0$.  Thus, $\calB*X_t$ is a conformal martingale.  
 Since $\calB*X_t<<2X_t$, we can apply Corollary \ref{cor1} to conclude that 
 $$
 \|\calB*X_T\|_p\leq \sqrt{2p(p-1)}\|X_T\|_p, \,\,\,\,\,\, 2\leq p<\infty.
 $$
 
  This gives the following result proved in \cite{BanJan1}.
  
  \begin{theorem}\label{banjan} Suppose $2\leq p<\infty$.  For all $f\in C_0^{\infty}(\bC)$, $f:\bC\to\bC$, 
  \begin{eqnarray}\label{BAsecondbest}
 \|Bf\|_p&=&\lim_{T\to\infty}\|\mathcal{S}_{\calB}f\|_p\nonumber\\\nonumber\\
 &\leq &\sqrt{2p(p-1)}\lim_{T\to\infty}\|V_f(Z_T)-V_f(Z_0)\|_p\\\nonumber \\
 &=&\sqrt{2p(p-1)}\|f\|_p.\nonumber
 \end{eqnarray}
 If instead we restrict to $f:\bC\to\bR$,  we obtain the bound 
$$ \|Bf\|_p\leq \sqrt{p(p-1)}\|f\|_p,$$ for  $2\leq p<\infty$.
 \end{theorem}
As remarked in \cite{BanJan1},  
we note that the bound ${\|B\|}_{p}\leq 
\sqrt{2(p^2-p)}$, $2\leq p<\infty$,  is already asymptotically better than 
($\ref{BAtau2}$). To see this,  divide both terms by $\sqrt{2}(p-1)$ and raise this to the power
$p$ and let $p\rightarrow\infty$. The $\sigma(p)$ term diverges and the estimate from (\ref{BAsecondbest})
converges to $\sqrt{e}$. 

Since $\|B\|_{2\to 2} =1$, we can use   
the Riesz-Thorin interpolation theorem and our estimate 
(\ref{BAsecondbest}) to improve the bound for all $p$. The following 
result is proved in \cite{BanJan1}. 

\begin{theorem}\label{BAbest} For all $f\in C_0^{\infty}(\bC)$, $f:\bC\to\bC$,

\begin{equation}\label{main_result2}
\|Bf\|_p\leq 1.575 \,(p^*-1)\,{\|f\|}_p, \,\,\,\,\,\, 1<p<\infty.
\end{equation}
\end{theorem}

This theorem represents the best known estimate for the norm of the Beurling-Ahlfors operator and the best progress toward the Iwaniec's conjecture as of now.  However, the significance of this result is more than just the fact that we have come numerically closer to the desired conjecture than previous estimates.  Other arguments in the literature up to this point essentially estimate the norm of $B$ by estimating the norm of $R_1^2-R_2^2$ and $R_1R_2$, individually, and  adding them up. This is what is done, for example, to obtain the bound $2(p^*-1)$.  Such approach will not even provide the best constant for $p=2$. The estimate in Theorem \ref{BAbest} for the first time treats the operator $B$ as a single unit and takes into account some (but perhaps not all) the interactions between the martingales representing the second-order Riesz transforms and the operator $B$  itself. 

\begin{remark}
As in the case of the first order Riesz transforms, sharp weak-type $(1,1)$ inequalities are unknown for second-order Riesz transforms and for the Beurling-Ahlfors operator.  Of great  interest (see \cite[Ch. 4]{AstIwaMar1}) is the case of the Beurling-Ahlfors operator.  For this it is shown in \cite{BanJan2} and \cite{Gil1} that a lower bound for the best weak-type $(1,1)$ constant is $\frac{1}{\log2}$  and it is conjectured in \cite{BanJan2} that this is best possible. 
\end{remark}

\begin{remark}  There are also several papers in the literature which study the behavior of the $L^p$-constants for powers (iterations) of the
Beurling-Ahlfors  operator.  For some of this literature we refer the reader to \cite{Dra1} and \cite{DraPetVol1}. 
\end{remark}

\begin{remark}\label{smallest-heat} Finally, we should also remark here that just as in the case of the background radiation process (see Theorem \ref{background-smallest-B}),  it is not possible to pick a $2\times 2$ matrix which represents the operator $B$ and which has a norm smaller than $2$ using the space-time Brownian motion employed in this section.  In fact, it is easy to see that in this case the kernel of the operator $A\to \mathcal{S}_{A}$ is the linear span of the orthogonal matrix 
\begin{equation}
K=\left[\begin{array}{cc}0 & -1 \\ 1& 0\end{array}\right]
\end{equation}
and an easy computation shows that 
\begin{equation}\label{space-time-smallest-B}
\inf_{\alpha \in \bC}\|\calB+\alpha K\|=\|\calB\|=2,
\end{equation}
where $\calB$ is the $2\times 2$ matrix in (\ref{BAmatrix}) that represents the operator $B$. 

\end{remark}

\subsection{Multipliers of Laplace transform-type; heat semigroup}
Returning to (\ref{fourier-A(t)}), if  $A(t)=a(t){I}$, where ${I}$ is the identity $n\times n$ matrix, we again obtain the Laplace-type transform multipliers as in \S3.3
\begin{equation}\label{laplaceheat}
{\mathcal{S}_af}(x)=-\int_0^{\infty} a(t) \,\Delta H_{2t}f(x)dt=-\int_0^{\infty} a(t) \,\frac{\partial}{\partial t}{H_{2t}f(x)}dt
\end{equation}
with Fourier transform 
\begin{equation}\label{fourier-A(t)-1}
\widehat{\mathcal{S}_af}(\xi)=\left(4\pi^2 |\xi|^2\int_{0}^{\infty} a(t) e^{-4\pi^2 t|\xi|^2}\,dt\right) \widehat{f}(\xi)
\end{equation}
and a result similar to Theorem \ref{LaplaceTran} holds but this time with better constants.  More precisely, we have 

\begin{corollary}
Suppose 
 $$a(t)=\frac{t^{-i\gamma}}{\Gamma(1-i\gamma)}.$$
 Then 
 $$
\mathcal{S}_af(x)=(-\Delta)^{i\gamma}f(x) 
 $$ 
 and
 
\begin{equation}\label{Hyt2}
\|(-\Delta)^{i\gamma}f(x)\|_p\leq \frac{(p^*-1)}{|\Gamma\left(1-i\kappa\right)|}\|f\|_p,\,\,\,\,\,\, 1<p<\infty.   
\end{equation}
\end{corollary}
We should now compare this bound  with (\ref{Hyt1}). For this, consider the ratio of the Gamma factor in (\ref{Hyt1}) and the one in (\ref{Hyt2}).  Using the doubling property of the Gamma function we see that 
\begin{equation}\label{gamma2}
\frac{|\Gamma\left(2-2i\gamma \right)|}{|\Gamma\left(1-i\gamma\right)|}=
\frac{2}{\sqrt{\pi}}\Big|\Gamma\left(\frac{3}{2}-i\gamma\right)\Big|. 
\end{equation}
At $\gamma=0$ we see that  this ratio is $1$.  On the other hand, for all $z\in \bC$
 the product formula gives that 
$$
\Gamma(z)=\frac{e^{-\kappa z}}{z}\prod_{n=1}^{\infty}\left(1+\frac{z}{n}\right)^{-1}\, e^{z/n},
$$
where $\kappa\approx 0.57721$ is the Euler-Mascheroni constant.  From this it follows trivially (by examining each factor separately) that the right hand side of (\ref{gamma2}) is  a decreasing function of $\gamma$ on $(0, \infty)$.  Since $\overline{\Gamma({z})}=\Gamma({\overline z})$, we see that the above ratio is always smaller than $1$ for all $\gamma\neq 0$. (The basic identities used above for the Gamma function can all be found in \cite{AbrSte1}.) Thus the bound in (\ref{Hyt2}) is better than the bound in (\ref{Hyt1}).  To the best of our knowledge this bound,  which appeared for the first time in \cite[Eq. (7.3)]{Hyt3}, is the best available in the literature.    This raises the following challenging problem.

\begin{problem}  Find the best constant in the inequality (\ref{Hyt2}).  That is, what is the norm of the operators $(-\Delta)^{i\gamma}$ on $L^p(\bR^n)$, for $1<p<\infty$? 
\end{problem}

With $A(t)=a(t)\tilde{I}$, where $\tilde{I}$ is an $n\times n$ matrix as in the Laplace transform multipliers of \S3.3 (1 in the first entry and 0 else) and $a$ is a bounded function on $\bR^{+}$, we have the operator
\begin{equation*}
{\mathcal{S}_af}(x)=-\int_0^{\infty} a(t) \,\partial_1^2 H_{2t}f(x) dt,
\end{equation*}
with similar versions for $\partial_j$ and even for the Cauchy-Riemann operators $\partial$ and $\overline{\partial}$ in $\bC$. 

With the matrix 
\begin{equation}\label{BAmatrix-1}
A_{a}=\left[\begin{array}{cc}a(t) & -i a(t) \\\\-i a(t) & -a(t)\end{array}\right],
\end{equation}
where $a$ is a bounded function on $(0, \infty)$, we obtain the multipliers on $\bC$
\begin{equation}\label{BA-Laplace-1}
\mathcal{S}_{A_{a}}f(x)=-\int_0^{\infty} a(t) \,\partial^2 H_{2t}f(x) dt,
\end{equation}
where the operator $\partial$ is the Cauchy-Riemann operator as in (\ref{partials-1}).  From this it follows that 

\begin{equation}\label{fourier-A(t)Laplace}
\widehat{\mathcal{S}_{A_a}f}(\xi)=\left(4\pi^2 \overline{\xi}^2\int_{0}^{\infty} a(t) e^{-4\pi^2 t|\xi|^2}\,dt\right) \widehat{f}(\xi).
\end{equation}
Combining with (\ref{fourier-A(t)-1}) we have 
$$
\mathcal{S}_{A_{a}}f(x)=B\circ\mathcal{S}_{A_{a}}f(x).
$$
As in (\ref{RiesLaplace3}), we obtain 
\begin{equation}\label{BA-Laplace-2}
\|B\circ\mathcal{S}_{A_{a}}f\|_p\leq 2(p^*-1)\|a\|_{\infty}\|f\|_p, 
\end{equation}
for all $1<p<\infty$. 
\begin{remark}\label{remarkBALaplace}
Similar versions of (\ref{BA-Laplace-2}) can be obtained for compositions of second order Riesz transforms with Laplace transforms-type multipliers by replacing $\partial^2$ in (\ref{BA-Laplace-1}) with $\partial_j^2$ or $\partial_j\partial_k$.  

Finally, we should also note here that with the function $a$ taking real values the better bound of $\sqrt{2p(p-1)}$ for $2\leq p<\infty$ given in Theorem \ref{banjan} above holds for the operator $B\circ\mathcal{S}_{A_{a}}$. 
\end{remark}

\subsection{Beurling-Ahlfors in $\bR^n$ and the Iwaniec-Martin Conjecture}
From the Fourier transform of $B$ in (\ref{BAlaplacian}), we know that in 
terms of the Laplacian and the Cauchy-Riemann operator $\partial$ we have  $B=-{\partial^2}(-\Delta)^{-1}$.  From this it follows (see \cite{AstIwaMar1}) that the Iwaniec conjecture is equivalent to proving that $\|\partial f\|_p\leq (p^*-1)\|\overline\partial f\|_p$, for all $f\in C_0^{\infty}(\bC)$.  In their study of {\it Quasiconformal 4-manifolds}  \cite{DonSul1},  Donaldson and Sullivan defined a version of $B$ in higher dimensions acting on differential forms in terms of the Laplacian, 
the Hodge operator $\delta$, and its adjoint $\delta^*$. Recall that the  $k$--form 
$$
\omega(x)=\sum_I \omega_I dx_I,\quad dx_I=dx_{i_1} \wedge \ldots \wedge dx_{i_k},
$$
 in $\bR^n,\ k=1,2,\ldots,d,$ is in $L^p (\bR^n, \wedge^k)$ if
$$
\| \omega\|_{L^p (\bR^n,\wedge^k)} 
= \Big\|\left(\sum_I|\omega_I|^2\right)^{1/2}\Big\|_p < \infty.
$$
 We set 
$$
L^{p}(\bR^{n}, \wedge)=\oplus_{k=0}^{n}L^p (\bR^n, \wedge^k).
$$
The Donaldson--Sullivan ``signature" operator is defined by 
$$S\omega=(\delta\delta^*-\delta^*\delta)\circ (-\Delta)^{-1}\omega,$$
where the Laplacian acts on forms by acting on its coefficients.  
This is again a Calder\'on-Zygmund singular integral operator and as such it follows from \cite{Ste2} that   
$$
S: L^p (\bR^n, \wedge)\to L^p (\bR^n, \wedge),
$$
 for all $1<p<\infty$.  The $L^p$ operator norm $\|S\|_{p}$ is directly connected to the regularity of quasiregular mappings,
as well as conditions for a closed set to be removable under such maps. Identification
of the norm, as in the two dimensional case,  would also have implications for the existence of minimizers of
conformally invariant energy functionals and regularity of solutions to the generalized
Beltrami system, see \cite{AstIwaMar1, IwaMar1, IwaMar2}.

In $\bR^2$, acting on one forms, $S$ reduces to $B$, up to a sign. 
In \cite{IwaMar1},  Iwaniec and Martin proved that 
$$(p^*-1)\leq \|S\|_{p}\leq C (n+1)p^2,$$
where  $C$ is a universal constant independent of $n$ and made the far reaching 
\begin{conjecture}\label{IwaMar1}
For all $n\geq 2$, \,\,
$\|S\|_{p}=(p^*-1), \,\,1<p<\infty.$  
 \end{conjecture}

The operator $S$ has a representation as a sparse, 
block--diagonal matrix of second-order Riesz transforms.  In \cite{BanLin1}, this representation is used to give a 
representation of $S$ in terms of (``harmonic") martingale using the background radiation processes as in \S3.2 above.  From this and Burkholder's martingale inequalities as presented in \S2.1,  the following estimate is obtained in \cite{BanLin1}
$$
\|S\|_{p}\leq \begin{cases} \left({n}+2\right)(p^*-1),
&\text{$2\leq n\leq 14,\text{ and even}$}\\ 
{(n+1)}(p^*-1),&\text{$3\leq n\leq 13,\text{ and odd}$}\\
{\left({4n\over 3}-2\right)(p^*-1)},&\text{otherwise}.
\end{cases}
$$

 Using the space-time Brownian motion representation in \S3.3 and the exact same proofs as in \cite{BanLin1}, one obtains the bound $(\frac{2n}{3}-1) (p^*-1)$, for  $n\geq 15$, and other similar improvements  for $n$ below 15.  
By a much more careful analysis of the arguments in \cite{BanLin1}  and using again the space-time martingale as above, Hyt\"onen  \cite{Hyt3} has improved the bound to $\|S\|_{p}\leq (\frac{n}{2}+1) (p^*-1)$, for all $n\geq 2$. Other improvements of the bounds in \cite{BanLin1} using Bellman function techniques were obtained in \cite{PetSlaBre1}.

The resolution of Conjecture \ref{IwaMar1} seems out of reach at this point to this author. The following less ambition problem seems more plausible,  but even this has eluded us so far. 

\begin{problem}\label{Banconj}
Prove that for all $n\geq 2$, $\|S\|_{p}\leq C(p^*-1)$, $1<p<\infty$, for some constant $C$ independent of the dimension $n$. 
\end{problem}
 
 Finally, we mention the very recent applications of the space-time martingale inequalities by X.-D. Li \cite{Li2} (see also \cite{Li1, Li3}) to establish  the weak $L^p$-Hodge decomposition theorem and to prove the $L^p$-boundedness of the Beurling-Ahlfors operator on complete non-compact Riemannian manifolds with the so called non-negative Weitzenb\"ock curvature  operator.  In these papers, Li shows that the formulas we discussed above for Riesz transforms in $\bR^n$ when properly formulated continue to holds on manifolds under very general conditions.  In this context Theorem 3.4 in \cite{Li2} is particularly interesting as it shows that the martingale representation for the Beurling-Ahlfors operator in \cite{BanMen1} (which is a corollary of Theorem \ref{prop2} above) continues to hold on manifolds. The proof of such a  representation on manifolds, while much more technical, retains many of the features of the proof for the classical Beurling-Ahlfors operator in $\bR^2$ as given in \cite{BanMen1}.  These applications are yet another example of the power of Burkholder's ideas and their range of applications in areas of mathematics that on the surface seem far removed from the martingale transforms of Theorem \ref{bur66}.

\section{L\'evy processes and Fourier multipliers}
 When we view ``heat" martingales as those arising by composing the heat extension of the function with space-time Brownian motion and ``harmonic" martingales as those arising by composing the Poisson extension of the function with killed Brownian motion, the natural question arises:  Is it possible to use  other symmetric stable processes of order 
$0<\alpha<2$ to investigate the quantity $\|B\|_{p}$?  This question was raised in \cite[p.~989]{BanMen1}. Even more, is there a 
theory  similar to that of Brownian motion that would lead to martingale transforms and Fourier multipliers arising from more general 
L\'evy processes and their semigroups?  Some answers to these questions are provided in  \cite{BanBieBog1, BanBog1} where a general class of multipliers is obtained by transformations  of the Gaussian and jump parts of L\'evy processes.  We call these L\'evy multipliers.

\subsection{L\'evy multipliers} Consider a measure  $\nu\geq 0$ 
on $\bR^n$ with $ \nu(\{0\})=0$ and
 \begin{equation}\label{levymu1}
 \int_{\bR^n} \frac{|x|^2}{1+|x|^2} \,d\nu(x)<\infty.
\end{equation}
A measure with these properties is called  a  L\'evy measure.  For any finite Borel measure  $\mu\geq
0$ on the unit sphere
$\uS\subset \bR^n$ and functions $\varphi:\bR^n\to \bC$, $\psi:\uS\to \bC$ with $\|\phi\|_\infty\leq 1$ and $\|\psi\|_\infty\leq 1$, 
we consider the (multiplier) function 

\begin{equation} 
\M\left(\xi\right)=
\frac{
      \int_{\bR^n} \Big(1 - \cos (\xi \!\cdot\! x) \Big)\varphi\left(x\right)d\nu(x)+ 
     \frac{1}{2} \int_{\uS} |\xi\!\cdot\! \theta|^2 \psi\left( \theta \right) d\mu(\theta)
     }
     {
      \int_{\bR^n} \Big(1 - \cos (\xi \!\cdot\! x) \Big)d\nu(x)+ 
      \frac{1}{2}\int_{\uS} | \xi\!\cdot\! \theta|^2 d\mu (\theta)
     }.
\end{equation}
To emphasize the connections to the L\'evy-Khintchine formula given below, note that  
\begin{eqnarray*} 
\M\left(\xi\right)&=&
\frac{
      \int_{\bR^n} \Big(\cos (\xi \!\cdot\! x)
     -1 \Big)\varphi\left(z\right)d\nu(x) - 
\frac{1}{2}{A}\xi\cdot \xi
     }
     {
      \int_{\bR^n} \Big(\cos (\xi \!\cdot\! x)-1\Big)d\nu(x) - 
\frac{1}{2}{B}\xi\cdot \xi
     },
\end{eqnarray*}
where
\begin{eqnarray*}
 {A} = \left[ \int_{\uS} \psi\left(\theta\right)
  \theta_{i}\theta_{j}\, d\mu(\theta)
 \right]_{i,j=1 \ldots d} \quad \mbox{and} \quad
 {B} = \left[ \int_{\uS}  \theta_{i}\theta_{j}\, d\mu(\theta) \right]_{i,j=1 \ldots d}
\end{eqnarray*}
and where both ${A}$ and ${B}$ are $n\times n$ symmetric matrices and ${B}$ is
non-negative definite.   We observe that $\|\M\|_{\infty}\leq 1$. We call $\M$ a  {\it L\'evy multiplier}.
For reasons that will become clear later, we call $\varphi$ a {\it L\'evy jumps transformation} 
function and $\psi$ a  {\it L\'evy Gaussian transformation} function.
\begin{theorem}\label{LevyLp-bound} The Fourier multiplier   
$\widehat{\mathcal{S}_{\M}}f(\xi)=\M(\xi)\hat{f}(\xi)$ on $L^2(\bR^n)$ extends to an operator on $L^p(\bR^n)$, $1<p<\infty$, with
\begin{equation}
\|\mathcal{S}_{\M} f\|_p\leq (p^*-1)\|f\|_p.
\end{equation}
\end{theorem}

As it turns out, this class of multipliers also includes the second-order Riesz transforms $R_1R_2$ and hence, again, by the result of Geiss, Montgomery-Smith and Saksman \cite{GeiSmiSak1}, the constant $(p^*-1)$ cannot be improved in general.  

Theorem \ref{LevyLp-bound} is proved in \cite{BanBog1} for symmetric $\nu$ and with $\mu=0$ and in \cite{BanBieBog1} for general $\nu$
and $\mu$.  By a symmetrization and approximation argument, the general case reduces to the special case of $\nu$ symmetric and $\mu=0$.  While we will not give the details here, we present some ideas on how these Fourier multipliers arise from 
L\'evy processes.

L\'evy processes provide a rich class of stochastic processes which generalize several of the basic processes in probability, including Brownian motion, Poisson and compound Poisson processes, stable processes and other processes subordinated to  Brownian motion. They have been widely used in many areas of pure and applied mathematics, including stochastic control, financial mathematics, potential analysis, geometry and PDE's. We refer the reader to \cite{ConTan1}, \cite{OksSul1} and to the survey article \cite{App1} where some of these connections and applications are discussed.  Here we are interested in projections (conditional expectations) of martingale transforms arising by modifying the symbol $\rho$.  These operators lead  to Theorem \ref{LevyLp-bound}.

Recall that a L\'evy process  $\{X_t\}$ in $\bR^n$ is  a stochastic process with independent and stationary 
increments which is stochastically continuous. That is, for all $0<s< t<\infty$, Borel sets $\Theta\subset \bR^n$,  
$P_{0}\{\,X_t-X_s\in \Theta\,\}=P_{0}\{\,X_{t-s}\in \Theta\,\}$, and 
for any given sequence  of  ordered times $0<t_1<t_2<\cdots <t_m<\infty,$ the random variables  $X_{t_1}-X_0,\,\,  X_{t_2}-X_{t_1}, \dots, X_{t_m}-X_{t_{m-1}}$
are independent. Furthermore,   for all $\varepsilon>0$, 
$\lim_{t\to s}P_{0}\left\{\,|X_t-X_s |>\varepsilon\,\right\}=0.$
The celebrated  L\'evy-Khintchine formula
\cite{Sat1} 
guarantees the existence of a triple 
$\left( b, B, \nu \right)$ 
such that the characteristic function of the process is given by    
$E_0\left[\, e^{i \xi \cdot X_t}\,\right]  = e^{t
\rho(\xi)},$ 
where 
\begin{equation}\label{levykin}
\rho(\xi)= i b\cdot \xi  - \frac{1}{2} B\xi\cdot \xi +
\int_{\bR^n}\left[\,  e^{i\,\xi \cdot x}- 1 - i (\xi\cdot y) \,\bI_{B(0, 1)} (x) \,\right]\,
d\nu(x).
\end{equation}
Here,  $b \in \bR^n$,   $B$ is a non-negative $n \times n$ symmetric matrix, 
$\bI_{B(0,1)}$ is the indicator function of the unit ball $B(0, 1)\subset\bR^n$ and $\nu$ 
is a measure  on $\bR^n$ satisfying (\ref{levymu1}).

The triplet $\left( b , B, \nu \right)$, referred to here as a {\it L\'evy triplet},  is called the characteristics of  the process and  the measure $\nu$ is called the {\it  L\'evy measure} of the process. Conversely, given $\left( b , B, \nu \right)$ with such properties  there is  L\'evy process corresponding to it.  We will refer to $\rho(\xi)$ as the {\it L\'evy symbol}. The L\'evy triplet $\left( 0 , I, 0\right)$, where I is the $n\times n$ identity matrix gives the standard {\it Brownian motion} in $\bR^n$ and $\left( 0 , B, 0 \right)$ gives more general {\it Gaussian processes} with covariance $b_{jk}\min(s, t)$, where $B=(a_{jk})$.   Brownian motion plus drift  $X_t=bt+B_t$ 
arises from $\left( b , I, 0 \right)$. 

The  Poisson process $\pi_{\lambda}(t)$ of intensity $\lambda$ arises from $(0, 0, \lambda \delta_1)$, where $\delta_1$ is the Dirac delta at 1. If we let  $Y_1, Y_2, \dots$ be i.i.d. with distribution 
$\nu$ and independent of $\pi_{\lambda}(t)$, we get the {\it compound Poisson process} 
\begin{equation}\label{poisson1}
X_t=Y_1+Y_2+\cdots +Y_{ \pi_t(\lambda)}=S_{\pi_{\lambda}(t)}.
\end{equation}
By independence, 
\begin{eqnarray*}
E[e^{i\xi\cdot X_t}]=\sum_{m=0}^{\infty} P\{
\pi_{\lambda}(t)=m\}\, E[e^{i\xi\cdot S_m}]&=& \sum_{m=0}^{\infty}\frac{e^{-\lambda t}(\lambda t)^m}{m!} \left(\widehat{\nu}(\xi)\right)^m\\\\
&=&e^{\left(\lambda t\left(\widehat{\nu}(\xi)-1\right)\right))}.
\end{eqnarray*}
Hence, $X_t$ is a L\'evy process with L\'evy symbol 
\begin{equation}\label{compound}
 \rho(\xi)=\lambda\int_{\bR^n}(e^{i x\cdot\xi}-1) d\nu(x).
\end{equation}

A class of L\'evy processes which has been widely studied is rotationally invariant (symmetric) stable processes.  These are self-similar processes with L\'evy symbols
$
\rho(\xi)=-|\xi|^\alpha$, $0<\alpha\leq 2,
$
and L\'evy measures 
$$d\nu^{\alpha}
=\frac{c_{\alpha, n}}{|x|^{n+\alpha}}dx$$
 for $0<\alpha<2$ and $\nu^{\alpha}=0$ for $ \alpha=2.$
Here, ${c_{\alpha, n}}$ is a normalizing constant depending only on $\alpha$ and $n$. For  $\alpha=1$, this gives the  Cauchy  
process and  $\alpha\to 2$, gives Brownian motion.

We assume that $f\in C_0^{\infty}(\bR^n)$. 
The semigroup of the L\'evy process $\{X_t\}$ with L\'evy symbol $\rho$ acting on $f$ is given by 
\begin{eqnarray}\label{semigroup}
P_tf(x)=E_0[f(X(t)+x)] &=&E_0 \left(\int_{\bR^n}e^{-2\pi i(X_t+x)\cdot\xi}\,\widehat{f}(\xi)d\xi\right)\nonumber\\
&=&\int_{\bR^n}e^{t\rho(-2\pi\xi)}\,e^{-2\pi ix\cdot\xi}\widehat{f}(\xi)d\xi.
\end{eqnarray}
Differentiating  with respect to $t$ gives that the infinitesimal generator of the semigroup for L\'evy process is the Fourier integral operator 
\begin{equation}
\mathcal{A}f=\int_{\bR^n} \rho(-2\pi\xi) e^{-2\pi ix\cdot\xi}\,\widehat{f}(\xi)d\xi.
\end{equation}
 With $\rho(\xi)=-|\xi|^{\alpha}$  we obtain the fractional 
 Laplacian  $\mathcal{A}=-(-\Delta)^{\alpha/2}$. 

To more clearly illustrate the origins of Theorem \ref{LevyLp-bound} and to avoid several technical points,  let us assume that  the semigroup is self-adjoint on $L^2(\bR^n)$. As is well known (see \cite{App}) this happens  if and only if the L\'evy process $X_t$ is symmetric.  That is, if and only if $P\{X_t\in \Theta\}=P\{X_t\in -\Theta\}$ for all Borel sets $\Theta\in \bR^n$.  This leads to  $T_t$ being self-adjoint if and only if 
\begin{equation}\label{self-adjoint}
\rho(\xi)=-{\frac{1}{2}}\xi\cdot B\xi +\int_{\bR^n}
\Big(\cos(z\cdot \xi)-1\Big)d\nu(z),
\end{equation}
where $B$ a symmetric matrix and is $\nu$ a symmetric ($\nu(\Theta)=\nu(-\Theta)$) L\'evy measure.  
It is then clear that for symmetric L\'evy measures $\nu$, the 
Fourier multipliers $\M(\xi)$ in Theorem \ref{LevyLp-bound} are obtained from  
{\it ``transformations"} (or {\it ``modulations"}) of the above L\'evy symbol normalized by the symbol itself.   To make the connection to martingales as clearly as possible, we illustrate two different instances of these multipliers. The case when $\nu=0$ (the purely gaussian case under the assumption that B is strictly positive definite) and the case when $B=0$ (the purely compound Poisson case). For the first, we observe that from (\ref{semigroup}) we have 
$$
\widehat{P_tf}(\xi)=e^{t\rho(-2\pi \xi)}\widehat{f}(\xi) =e^{-2\pi^2t B\xi\cdot\xi}\widehat{f}(\xi).
$$
If as in (\ref{fourier-A(t)}) we consider the operator this time with Fourier multiplier defined by 
\begin{equation}\label{fourier-gaussian-A(t)}
\widehat{\mathcal{S}_Af}(\xi)=\left(4\pi^2 \int_{0}^{\infty} [A(t)\xi]\cdot \xi e^{-4\pi^2 t\B\xi\cdot\xi}\,dt\right) \widehat{f}(\xi)
\end{equation}
we see that it arise from the projections of the martingale transform this time defined by 

\begin{equation}\label{proj-gaussin}
\mathcal{S}_{A}^T f(x) = E^T\left[ \int_0^T A(T-s)\nabla_x P_{T-s}f(X_s)\cdot dB_s\,\big|\,
X_T=x\, \right],
\end{equation}
where this time $X_t$ is the diffusion (Gaussian process) associated with the operator 
$$
\mathcal{A}=\frac{1}{2}\sum_{j, k=1}^{n} b_{jk}\frac{\partial^2}{\partial x_j\partial x_k}. 
$$
If the matrix $A$ is constant as in (\ref{fourierA}),  this gives the multipliers 

\begin{equation}\label{fourierGaussian}
\widehat{\mathcal{S}_Af}(\xi)=\frac{A\xi\cdot\xi}{B\xi\cdot\xi}\,\widehat{f}(\xi)
\end{equation}
and under the assumption that $|A\xi\cdot\xi| \leq |B\xi\cdot\xi|$ for all $\xi\in \bR^n$, we get that these operators again have 
$$
\|\mathcal{S}_Af\|\leq (p^*-1)\|f\|_p, \quad 1<p<\infty. 
$$
Thus for the Gaussian case there is (essentially) no change from the Brownian motion case except for the replacement of the identity matrix $I$ by the symmetric matrix $B$.    

We now consider the case when 
\begin{equation}\label{poissonsymbol}
\rho(\xi)=\int_{\bR^n}
\Big(\cos(z\cdot \xi)-1\Big)d\nu(z),
\end{equation}
where $\nu$ is symmetric and finite.   The process $X_t$ is then a symmetric compound Poisson process as in (\ref{poisson1}). With $T>0$ and $Z_t=(X_t, T-t)$, $0<t<T$,  we set $V_f(Z_t)=V_f(X_t, T-t)$ where  $V_f(x, t)=P_tf(x)$.  Then $V_f(Z_t)=V_f(X_t, T-t)$, $0<t<T$ is a martingale (see \cite{BanBog1}, Lemmas 1-4)  and it follows from the general It\^o formula (\cite{DelMey1, Pro1}) that 
\begin{eqnarray*}
V_f(Z_t)- V_f(Z_0)&=&\sum_{0<s\leq t}[V_f(X_{s}, T-s)-V_f(X_{{s}^{-}}, T-s)]\\
&-& \int_0^t \int_{\bR^n} [V_f(X_{s^{-}}+z, T-s)- V_f(X_{s^{-}}, T-s)] d\nu(z)ds.
\end{eqnarray*}
Let $\varphi:\bR^d \to \bC$ be such that $\|\varphi\|_{\infty}\leq 1$.  Consider the new martingale
\begin{eqnarray*}
\varphi*V_f(Z_t)&=&\sum_{0<s\leq t}[V_f(X_{s}, T-s)-V_f(X_{{s}^{-}}, T-s)]\varphi(\Delta X_s)\\
&+& \int_0^t \int_{\bR^n} [V_f(X_{s^{-}}+z, T-s)- V_f(X_{s^{-}}, T-s)] \varphi(z)d\nu(z)ds
\end{eqnarray*}
on $0<t<T$.  
Then 
$$
[\varphi*V_f(Z)]_t=\sum_{0<s\leq t}|V_f(X_{s}, T-s)-V_f(X_{{s}^{-}}, T-s)|^2|\varphi(\Delta X_s)|^2
$$
and it follows that $\varphi*V_f(Z_t)<<V_f(Z_t)$. 
From Theorem  \ref{Wan1} it follows that the  ``martingale transform" $\varphi* V_f(Z_t)$ is in $L^p$ for all $1<p<\infty$ and that 
\begin{equation}\label{contraction-p}
\|\varphi* V_f(Z_t)\|_p\leq (p^*-1)\|f\|_p.
\end{equation}   As in the case of space-time Brownian motion we define the family of operators in 
$\bR^n$ by 
$$\mathcal{S}_{\varphi}^T f(x) = E^T\big [ \,\varphi\star V_{f}\,\big|\,
Z_T=(x,0)\,\big ].$$    
As in the case of the Brownian motion, the boundedness of the martingale transform $V_f(Z_t) \to \varphi*V_f(Z_t)$, (inequality (\ref{contraction-p})), the contraction in $L^p$ of the conditional expectation gives that 

\begin{equation}
\|\mathcal{S}_{\varphi}f\|_p\leq (p^*-1)\|\varphi\|_{\infty}\|\,\|f\|_p.
\end{equation}

The following proposition identifies the Fourier multiplier for the operator $\mathcal{S}_{\varphi}^T$ and hence we obtain the special case of Theorem \ref{LevyLp-bound} proved in  \cite{BanBog1}.  

\begin{proposition}\label{Lp-levybound}
The operators $\mathcal{S}_{\varphi}^T$ are Fourier multipliers on $L^2(\bR^n)$ and 
$$\widehat{S^{T}_{\varphi}\,f}(\xi)=\M_T(\xi)\hat{f}(\xi),$$
where   
 \begin{equation}\label{multipler1}
\M_T(\xi)=\left(e^{2T\rho(\xi)}-1\right)\frac{1}{\rho(\xi)}\int_{\bR^n}\left(1-\cos(z\cdot\xi)\right)\varphi(z)d\nu(z).
\end{equation} 
As $T\to \infty$, the operators $\mathcal{S}_{\varphi}^T$ generate Fourier multipliers $\mathcal{S}_{\varphi}$, with 
$\widehat{S_{\varphi}\,f}(\xi)=\M(\xi)\hat{f}(\xi)$ and  
$$\M(\xi)=\frac{1}{\rho(\xi)}\int_{\bR^n}\left(\cos(z\cdot\xi)-1\right)\varphi(z)d\nu(z).$$
\end{proposition}

We outline the proof here for the convenience of the reader.  For (\ref{multipler1}) we use the following Littlewood-Paley type identity.  

\begin{lemma} Integrating againts a function $g\in C_0^{\infty}(\bR^n)$ gives 
\begin{eqnarray*}
&&\int_{\bR^n} S^{T}_{\varphi} f(x)\overline{g(x)} dx=E[{\varphi*V_f(Z_T)}\overline{V_g(Z_T)}]\\
&=&
\int_0^T\int_{\bR^n}\int_{\bR^n}\int_{\bR^n}\Big\{\overline{V_g(x+y+z, T-s)-V_g(x+y, T-s)}\Big\}\times \\
&&\Big\{V_f(x+y+z, T-s)-V_f(x+y, T-s)\Big\}\varphi(z)d\nu(z) p_T(dy)\,\, dx\, ds
\end{eqnarray*}
\end{lemma}

\begin{proof} 
We recall again that $$\widehat{V_f}(\xi, t)= \widehat{P_tf}(\xi)=e^{t\rho(\xi)}\hat{f}(\xi)$$ 
We then, by Fubini's theorem, perform the integration with respect to  $\left\{dx\,p_T(dy)\right\}$ first and change variables $t=T-s$  to obtain 
\begin{eqnarray*}
&&\int_0^T\int_{\bR^n}\int_{\bR^n}\Big\{\overline{V_g(x+z, T-s)-V_g(x, T-s)}\Big\}\times\\
&& \Big\{V_f(x+z, T-s)-u_f(x, T-s)\Big\}\,\, dx\,\, \varphi(z)d\nu(z)\, ds\\
&=&\int_0^T\int_{\bR^n}\int_{\bR^n}\Big\{\overline{V_g(x+z, t)-V_g(x, t)}\Big\}
 \Big\{V_f(x+z, t)-V_f(x, t)\Big\} dx\varphi(z)d\nu(z)dt\\
 &&\int_0^T\int_{\bR^n}\int_{\bR^n}\Big\{\overline{e^{iz\cdot\xi}\hat{g}(\xi)-\hat{g}(\xi)} \Big\}e^{t\rho(\xi)}\Big\{e^{iz\cdot\xi}\hat{f}(\xi)-\hat{f}(\xi)\Big\}e^{t\rho(\xi)}d\xi\varphi(z)d\nu(z) dt\\
 &&\int_0^T\int_{\bR^n}\int_{\bR^n}|e^{iz\cdot\xi}-1|^2 \hat{f}(\xi)\overline{\hat{g}(\xi)}e^{2t\rho(\xi)}d\xi\varphi(z)d\nu(z) dt
\end{eqnarray*}
\begin{eqnarray*}
&=&\int_{\bR^n}\int_{\bR^n}2\Big(1- \cos(z\cdot\xi)\Big)\left(e^{2T\rho(\xi)}-1\right)\frac{1}{2\rho(\xi)}\hat{f}(\xi)\overline{\hat{g}(\xi)}d\xi\varphi(z)d\nu(z)\\
&=&\int_{\bR^n}\left(e^{2T\rho(\xi)}-1\right)\frac{1}{\rho(\xi)}\left(\int_{\bR^n}  \Big(1- \cos(z\cdot\xi)\Big)\varphi(z)d\nu(z)    \right) \hat{f}(\xi)\,\overline{\hat{g}(\xi)}\, d\xi,\\
\end{eqnarray*}
which identifies the multiplier. 
\end{proof}

The proof of Theorem \ref{LevyLp-bound} follows from this proposition by an approximation argument to extend to general symmetric L\'evy measures and by a symmetrization arguments for measures and functions.  We refer the reader to   \cite{BanBieBog1} and \cite{BanBog1} for full details. 

\subsection{Examples} 
\begin{example}\label{gaussian-modulation}
If $n=2$, $\nu=0$ and $\mu$ is point  mass at $1,i,e^{-i\pi/4},e^{i\pi/4}$
and $\psi(1) = 1$, $\psi(e^{-i\pi/4}) = i$, $\psi(i) = -1$ and $\psi(e^{i\pi/4}) = -i$, we get 
\begin{eqnarray*}
 \int_{\uS} \abs{\xi\cdot\theta}^{2}\psi(\theta)d\mu(\theta) = \xi_{1}^2 - \xi_2^2 +
 i(\xi_1\frac{1}{\sqrt{2}} - \xi_2\frac{1}{\sqrt{2}})^2 - i(\xi_1\frac{1}{\sqrt{2}} + \xi_2\frac{1}{\sqrt{2}})^2  = \overline{\xi}^2,
\end{eqnarray*}
  and 
\begin{eqnarray*}
 \int_{\uS} \abs{\xi\cdot\theta}^{2}d\mu(\theta) = \xi_{1}^2 + \xi_{2}^2 +
 (\xi_1\frac{1}{\sqrt{2}} - \xi_2\frac{1}{\sqrt{2}})^2 + (\xi_1\frac{1}{\sqrt{2}} + \xi_2\frac{1}{\sqrt{2}})^2= 2\abs{\xi}^2. 
\end{eqnarray*} 
Thus 
\begin{equation*}
\M(\xi) = \frac{\overline{\xi^2}}{2\abs{\xi}^2}
\end{equation*}
and hence
\begin{equation}\label{gaussian}
\mathcal{S}_{\M}f=\frac{1}{2}Bf,
\end{equation}
which leads to the  estimate $\|B\|_{p}\leq 2(p^*-1)$ as in (\ref{BA2}).

If instead $\mu$ is point  mass at $1,i$ and $\psi(1) = 1$,  $\psi(i) = -1$, we have
$$
\M(\xi) = \frac{{\xi_1^2-}\xi_2^2}{\abs{\xi}^2},
$$
and $T_{\M}f=R_2^2f-R_1^2f$. If $\mu$ is point  mass at $e^{-i\pi/4},e^{i\pi/4}$  and $\psi(e^{-i\pi/4}) = i$ $\psi(e^{i\pi/4}) = -i$, we obtain $T_{\M}f=2R_1R_2$. Hence, by the result of Geiss, Montgomery-Smith and Saksman \cite{GeiSmiSak1},  the constant in Theorem \ref{LevyLp-bound} cannot be improved in general. 
\end{example}

As in the matrix representation for the second we can ask the question

\begin{question}  Exactly as in Theorem (\ref{background-smallest-B}) and equation (\ref{space-time-smallest-B}), one may ask: Is it possible to pick a better $\psi$ and $\mu$ such that the $\frac{1}{2}$ in (\ref{gaussian}) can be replaced with $1$? \end{question} 

Unfortunately, the answer to this question is again ``no" as shown by the following proposition from \cite{BanBieBog1}.  
This result should be compared with those in Theorem \ref{background-smallest-B} and Remark \ref{smallest-heat}.

\begin{proposition}\label{gaussian-modulations-smallest} 
 If $\mu$ is a finite measure on the circle $\uS^1$ in $\R^2$ and $\psi:\uS^1\to \bC$ with $\|\psi\|_{\infty}\leq 1$ and 
\begin{equation} 
 \frac{\int_{\uS^1}\abs{\xi\cdot\theta}^{2}\psi(\theta)d\mu(\theta)}{\int_{\uS^1}\abs{\xi\cdot\theta}^{2}d\mu(\theta)}
        = \frac{\overline{\xi^2}}{c\abs{\xi}^2} \,, \quad \xi \in \bR^2 \setminus \{0\}, 
\end{equation}
 then $|c| \geq 2$.
\end{proposition}

In the above example (\ref{gaussian-modulation}) the modulations are perform on the Gaussian part of the L\'evy processes. If, again with $n=2$, we perform the modulations on the jump part, this leads to another martingale representation for the Beurling-Ahlfors operator. 
\begin{example}  Assume $\mu=0$ and for   $0<\alpha<2$, consider the stable measure in $\bR^2$ (in polar coordinates)
$$d\nu^{\alpha}(r, \theta) = r^{-1-\alpha}dr\,d\theta$$ and a bounded function $\varphi:\uS\to \bC$ extended to all of $\bC$ by $\varphi(z)=\varphi\left(z/|z|\right)$. Using polar coordinates it follows  that   
\begin{eqnarray*}
 \int_{\R^2} (1-\cos(\xi\cdot z))\varphi(z) d\nu^{\alpha}(z) 
& =& \int_{\uS}\int_0^\infty (1-\cos(r\theta\cdot \xi))\varphi(r\theta) r^{-1-\alpha}\,dr\,d\theta \\
 &=& \int_{\uS} \abs{\xi\cdot\theta}^{\alpha}\varphi(\theta) \int_0^{\infty} \frac{1-\cos(s)}{s^{1+\alpha}} ds d\theta\\ 
 &=& c_{\alpha}\int_{\uS} \abs{\xi\cdot\theta}^{\alpha}\varphi(\theta)d\theta.
\end{eqnarray*}
Thus 
\begin{equation}\label{jump}
\M(\xi) =\frac{\int_{\uS} \abs{\xi\cdot\theta}^{\alpha}\varphi(\theta)\,d\theta}{\int_{\uS} \abs{\xi\cdot\theta}^{\alpha}\,d\theta}.
\end{equation}
Setting $\theta=(\cos(t),\sin(t))$ and taking 
$\varphi(\cos(t),\sin(t)) = e^{-i2t}$, a computation (see \cite{BanBieBog1}) yields 
\begin{equation*}
 \M(\xi) =\frac{\alpha}{\alpha+2}\frac{\overline{\xi}^2}{\abs{\xi}^2},
 \end{equation*}
and hence
 \begin{equation}\label{jump}
 \mathcal{S}_{\M}f=\frac{\alpha}{\alpha+2}Bf.
 \end{equation}
Again, we can recover from this the bound $2(p^*-1)$ by letting $\alpha\to 2$. 
\end{example}

\begin{example}
 Finally, if we take $\mu=0$, $\varphi=0$, 
$$\lambda=\delta_{(1,0,\ldots,0)}+\delta_{(-1,0,\ldots,0)}+
\cdots+\delta_{(0,0,\ldots,1)}+\delta_{(0,0,\ldots,-1)}\,.$$
and consider the L\'evy measure (in polar coordinates) $$\nu^{\alpha}(dr, d\theta)=r^{-1-\alpha}dr\,
d\lambda(\theta), \hskip.5cm 0<\alpha<2,$$ 
(the symmetric $\alpha$-stable L\'evy
process with independent coordinates) we have
\begin{eqnarray*}
 \int_{\uS} |\xi\!\cdot\!  \theta|^\alpha d\lambda(\theta)=
C_{\alpha} |\xi_1|^\alpha+\cdots+|\xi_d|^\alpha 
\end{eqnarray*}
Let $\psi(z_1,\ldots,z_d)=1$ if $z_k=0$ for $k\neq j$ and 
$z_j\neq 0$, 
and let $\psi=0$ otherwise. (That is,  observe only the jumps of the jth coordinate
process.) 
The multiplier is

\begin{equation*}
  \M(\xi)=\frac{|\xi_j|^\alpha}{|\xi_1|^\alpha+\ldots+|\xi_d|^\alpha}\,,\quad
  \xi=(\xi_1,\ldots,\xi_n)\in \bR^d\,.
\end{equation*}
These are the Marcinkiewicz-type multipliers as in \cite[p. 110]{Ste2}
\end{example}
\begin{remark}
From formulas  (\ref{gaussian}) and  (\ref{jump}) we see that the Beurling-Ahlfors operator can be obtained by either  {\it L\'evy jump transformations} or {\it L\'evy Gaussian transformations} and that from either of these the bound $2(p^*-1)$ follows.  The obvious question arising from our second example is: Can one obtain a similar representation where one could let $\alpha\to\infty$?  The obvious answer is that because of the probabilistic restriction of $0<\alpha\leq 2$, this does not seem possible in any direct way.  However, it may be that using additive L\'evy processes one could free oneself of this restriction and obtain further improvements.  At this point this is only speculation (perhaps wishful thinking)  on this author's part. 
\end{remark}

The calculations leading to (\ref{jump}) suggest the following conjecture which would imply the desired sharp bound for the Beurling-Ahlfors operator. 

 \begin{conjecture} Let $\varphi:\uS\to \bC $, $\|\varphi\|_{L^{\infty}}\leq 1$.  For any $0<r<\infty$, and any $n\geq 2$, set 
 \begin{equation}
 \M(\xi) = \frac{\int_{\uS} \abs{\xi\cdot\theta}^{r}\varphi(\theta)\,d\sigma(\theta)}{\int_{\uS} \abs{\xi\cdot\theta}^{r}\,d\sigma(\theta)}, 
 \end{equation}
where $\sigma$ denotes the surface measure on the unite sphere $\uS$ of $\bR^n$.   Then 
 \begin{equation}
\|\mathcal{S}_{\M}\|_p\leq p^*-1. 
\end{equation}
 \end{conjecture}

 \section{Burkholder, Iwaniec and Morrey}
In this section we discuss how the biconcavity properties of Burkholder's function $U$ lead to connections with a problem  of considerable interest in the calculus of variations commonly referred to as Morrey's conjecture and to a conjecture which implies the Iwaniec conjecture. 

\subsection{Rank-one convexity and quasiconvexity}  As we have already mentioned, it is well known (\cite{AstIwaMar1}) that proving $\|Bf\|_p\leq (p^*-1)\|f\|_p$, $1<p<\infty$, is equivalent to proving that 
 \begin{equation}\label{BA-2}
 \|\partial f\|_p\leq (p^*-1)\|\overline{\partial}f\|_p, \,\,\,\, 1<p<\infty, 
 \end{equation}
  for all $f\in C_0^{\infty}(\bC)$. 
 
 Viewed in terms of the function $V$ in (\ref{v}), (\ref{BA-2}) is the same as proving that
 \begin{equation}\label{BA-V} 
 \int_{\bC}V(\overline\partial f, \partial f)dm(z)\leq 0,\,\,\,\,  f\in C_0^{\infty}(\bC). 
 \end{equation}
 Since by (\ref{sub1}) the Burkholder function $U$ of (\ref{u}) satisfies  $V(z, w)\leq U(z, w)$ for all $w, z\in \bC$, it is natural to make the following 
 \begin{conjecture}\label{BA-Bur-U}  For all $f\in C_0^{\infty}(\bC)$, 
 \begin{equation}\label{BA-U} 
 \int_{\bC}U(\overline\partial f, \partial f)dm(z)\leq 0.  
 \end{equation}
 \end{conjecture}
 This conjecture arose naturally in \cite{BanWan3} which uses, for the first time,  the function $U$ in connections with the norm of $B$. The conjecture is written down in \cite{BanLin1} as {\it Question 1} and we refer the reader to that paper where several other related questions and problems are stated. 
 The conjectured inequality (\ref{BA-U}) and the convexity properties (listed below) satisfied by the function $U$ lead to other unexpected connections and applications of Burkholder's ideas.   
 
 Denote  by $\M^{n\times m}$ the set of all $n\times m$ matrices with real entries.  
 The function $\Psi:\M^{n\times m}\to \bR$  is said to be  {rank-one convex} if for each
$A, B \in \M^{n\times m}$ with rank $B = 1$, the function
\begin{equation}\label{rank-one}
h(t)=\Psi(A+tB), \, \, \,\, t\in \bR
\end{equation}
is convex.  The function is said to be {quasiconvex} if  it is locally integrable and for each $A\in \M^{n\times n}$, 
bounded domain $\Omega\subset \bR^n$ and each compactly supported Lipschitz function $f : \Omega\to \bR^m$, we have 
  \begin{equation}\label{quasiconvex}
  \Psi(A)\leq \frac{1}{|\Omega|}\int_{\Omega} \Psi\left(A+Df(x)\right)\,dx, 
  \end{equation}
where $Df$ is the Jacobian matrix of $f=(f_1,\ldots,f_m)$. That is, 
\[
Df=\begin{pmatrix}
\partial_1 f_1&\ldots&\partial_n f_1\\
\vdots&\ddots&\vdots\\
\partial_1 f_m&\ldots&\partial_n f_m
\end{pmatrix}.
\]

These properties arise in many problems in the calculus of variations, 
especially in efforts to extend the so called ``direct method" techniques from convex energy functionals to non-convex functionals.  They were introduced by C.B. Morrey (see \cite{Mor2}) and further developed by J. Ball \cite{Bal1}.  For more (much more) on the relationship
 between these properties and their consequences in the direct method of the calculus of variations, we refer the reader to Dacoronga \cite{Dac1}.  If $n = 1$ or $m = 1$, then $\Psi$ is quasiconvex or rank-one convex if and only if it is convex. If $m\geq 2$ 
and $n\geq 2$, then convexity $\Rightarrow$ quasiconvexity $\Rightarrow$ rank-one convexity. (See \cite{Dac1} where the notion of polyconvexity  which lies ``between" convexity and quasiconvexity is also discussed.)  
 In 1952, Morrey \cite{Mor1} conjectured that rank-one convexity does not imply quasiconvexity when
both $m$ and $n$ are at least 2. In 1992,  \v{S}ver\'ak \cite{Sve2} proved that this is indeed the case if $m\geq 3$  and $n\geq 2$. The cases $m=2$ and $n\geq 2$ remain open.  One of the difficulties with these notions of convexity  
 is that it is in general very difficult to construct nontrivial, interesting examples of such functions.

 Enter Burkholder's function $U$. It is proved in \cite{Bur45} that for all $z,\ w,\ h,\ k\in \bC$ with  $|k|\leq |h|$, the function  
$t\to U (z+th,\ w+tk)$  is concave in $\bR$ , or equivalently that $t\to -U (z+th,\ w+tk)$ is convex in $\bR$ . The concavity property of $t\to U (z+th,\ w+tk)$ is crucial in the proof of the properties in (\ref{sub1})--(\ref{sub3}). Properly interpreted, this convexity property of $U$ is equivalent to rank-one convexity. 

Let us explain this in more detail. First recall that if $n=m$ and  $h',k'\in\bR^n$, $h'\otimes k'$ denotes the $n\times n$ matrix $h'k'^*$.  That is,
if $h'=(h_1', h_2', \dots, h_n')$ and $k'=(k_1', k_2', \dots, k_n')$ their tensor product is the matrix
$$
h'\otimes k'=\left[\begin{array}{ccccc}h_1'k_1' & h_1'k_2'  &\ldots & h_1'k_n' \\h_2'k_1' & h_2'k_2'
 &\ldots & h_2'k_n' \\\vdots & \vdots & \ddots&\vdots\\h_n'k_1' & h_n'k_2'&\dots & h_n'k_n'\end{array}\right].
$$

By  \cite[p.~100]{Dac1}, the rank-one convexity of the function $\Psi:\M^{n\times n}\to \bR$   is equivalent to the function
\begin{equation}
t\mapsto \Psi(A+th'\otimes k')
\end{equation}
being convex in $t$ for every $A\in \M^{n\times n}$ and for every $h',k'\in\bR^n$.  Restricting now to $n=2$, we define 
 the function 
$
\Gamma\colon \M^{2\times 2}\to \bC\times \bC$
by 
$$
\Gamma\left(A\right)=(z, w),
$$ where 
$$
A=\left(\begin{array}{cc}a & b \\c & d\end{array}\right),
$$
$z=(a+d)+i(c-b)$ and  $w= (a-d)+i(c+b)$.  We then set 
$\Psi_{U}= -U\circ \Gamma.$  This gives
\begin{eqnarray*}
\Psi_{U}\left(A\right)&=& -\alpha_p \{[(a+d)^2+(c-b)^2]^{1/2}-(p-1)[(a-d)^2\\
&+& (c+b)^2]^{1/2}\}\{[(a+d)^2+(c-b)^2]^{1/2}\\
&+& [(a-d)^2+(c+b)^2]^{1/2}\}^{p-1}.
\end{eqnarray*}
Then for any two vectors $h'=(h_1', h_2')$ and $k'=(k_1', k_2')$ in $\bR^2$ we have 
$$
h'\otimes k'=\left(\begin{array}{cc}h_1'k_1' & h_1'k_2' \\h_2'k_1' & h_2'k_2'\end{array}\right)
$$
and 
\[
\Psi_{U}(A+th'\otimes k')=-U (z+th, w+tk),
\]
where $h=(h_1, h_2)\in \bR^2$ and $k=(k_1, k_2)\in \bR^2$ are such that 
\[
\begin{cases}h_1=h'_1 k'_1 +h'_2 k'_2\\
h_2=h'_1 k'_2-h'_2 k'_1\\
k_1=h'_1 k'_1-h'_2 k'_2\\
k_2=h'_2 k'_1+h'_1 k'_2.\end{cases}
\]
Observing that $|h|=|k|$, we see that the rank-one convexity of $\Psi_{U}$ follows from Burkholder's convexity property of $t\to -U (z+th,\ w+tk)$, for $z,\ w,\ h,\ k\in \bC$ with  $|k|\leq |h|$. The above argument follows \cite{BanLin1}. For further clarity and insight into this argument, see \cite{Iwa3}.

Now, if $f=f_1+if_2 \in C_{0}^{\infty}(\bC)$, then 
\begin{equation}
 Df=\begin{pmatrix}
     \partial_1f_1 &  \partial_2f_1  \\
     \partial_1f_2 & \partial_2f_2
\end{pmatrix}
\end{equation}
and 
\begin{equation}
\Psi_{U} \left(Df\right)
=-U\left({\overline \partial f},\ {\partial f}  \right).
\end{equation}
Thus quasiconvexity of $\Psi_{U}$ (at $0\in \bR^{2\times 2}$) is equivalent to 
\begin{equation}\label{quasi} 
 0\leq -\int_{\text{supp f}}U(\overline\partial f, \partial f)dm(z),
 \end{equation}
which is equivalent to (\ref{BA-U}).  Thus the  following remarkable question encompassing both problems arises. 
\begin{question} Is the Burkholder  function $U$ also quasiconvex in the sense that $\Psi_U$ is quasiconvex?
\end{question}
\begin{remark} This is a ``win-win" question and its resolution would be of great interest.  In the positive it would imply Iwaniec's 1982 conjecture, and in the negative it would solve Morrey's 1952 problem  for the important case $n=m=2$.   
In either case, there is a great theorem here but we just simply do not know which one holds.  Of course, while we believe it is unlikely, it could also be true that \eqref{BA-V} holds while \eqref{BA-U} does not.  This would even be better. 
\end{remark}

\begin{remark}
In their very recent paper \cite{AstIwaPraSak}, K. Astala, T. Iwaniec, I. Prause and E. Saksman prove that Burkholder's function is quasiconvex (quasiconcave in their notation) when tested
on certain deformations of the identity.  This result  already has many interesting consequences.  We refer the reader to their paper for precise statements and details of results. 
\end{remark}

The article by A. Baernstein and S. Montgomery-Smith \cite{BaeSmt1} presents various connections between the function $U$ and another function $L$ used by Burkholder to prove sharp weak-type inequalities for martingales and harmonic functions under the assumption of differential subordination, \cite[p. 20]{Bur46}.  This function $L$ was subsequently, and independently, rediscovered by  \v{S}ver\'ak in \cite{Sve3}.  For more on these connections, we refer the reader to  \cite[pp.~518-523]{AstIwaMar1}, \cite{Iwa3}, \cite{Sve1, Sve2, Sve3} and \cite{VasVol2}.  

We observe here that as shown by Burkholder the function $U$ is not the smallest majorant of $V$ satisfying the important property that for 
all $z,w,h,k\in\bC$ with $|k|\leq |h|$, the mapping
$t\mapsto U(z+th, w+hk)
$
is concave on $\bR$ which 
 is a key property for the proof of his inequalities.
Indeed, as was already pointed out in (\ref{smallestbicon}),  the smallest majorant of $V$ with this property \cite[p. 81]{Bur45}  is 
\begin{equation}\label{minimalU}
\tilde U(z,w)=\begin{cases} V(z,w), & \text{if $|w|\leq (p^*-1)|z|$}\\
U(z,w),&\text{if $|w| > (p^*-1)|z|$}\end{cases}
\end{equation}
for $2\leq p<\infty$ and with $U$ and $V$ interchanged for $1<p\leq 2$.

It is interesting to compare the conjectured inequality (\ref{BA-U}) for $U$ and $\tilde U$ by calculating with the  ``extremals" used by Lehto \cite{Leh1} to show that $\|B\|_{p}\geq (p^*-1)$.  Suppose $2<p<\infty$, $0<\theta<1$ and consider the functions
\[
f_\theta(z)=\begin{cases} z\ |z|^{-2\theta/p},&\text{for $|z| < 1$}\\
\overline z^{-1},&\text{for $|z|\geq 1.$}\end{cases}
\]
Then a computation gives 
\[{
\|\partial f_\theta\|_{p}\over \|\overline{\partial}f_\theta\|_{p} }=\left(
{(p-1)(p-\theta)^p\over (p-1)\theta^p+(1-\theta)p^p}\right)^{1/p}.
\]
Since this holds as $\theta\uparrow 1$, one finds that $\|B\|_{p}\geq p-1$, for $2<p<\infty$. Using the same family of functions, it is  shown in  \cite{BanLin1} that for all $0<\theta<1$,
\begin{equation}\label{Leh-tildeU}
\int_{\bC}U\left( {\overline{\partial}f_\theta},
{{\partial}f_\theta}\right)dm(z)=0
\end{equation}
while 
\begin{equation}\label{Leh-U}
\int_{\bC}\tilde U\left({\overline{\partial}f_\theta},
{\partial f_\theta}\right)dm(z)=\pi[p(1-1/p)^{p-1}-(p-1)^{p-1}]<0. 
\end{equation}

\subsection{Riesz transforms and the Burkholder function $U$, revisited} To the best of our knowledge the problem of determining the norm of $R_jR_k$ was first raise in \cite[p. 260]{BanLin1} and a question (Question 3) similar to Conjecture \ref{BA-Bur-U} is in fact raised there.   While we now know that $2\|R_jR_k\|_p=(p^*-1)$, this question remains of interest.  However,  given what we currently know, the question needs to be updated and reformulated. The proof given in \cite{BanMen1} that $2\|R_jR_k\|_p\leq (p^*-1)$ and $\|R_j^2-R_k^2\|_p\leq (p^*-1), j\not= k$, can be stated in terms of the functions $V$ and $U$ from \eqref{v} and \eqref{u} as follows.

\begin{theorem} For $f\in C_0^{\infty}(\bR^n)$, consider the martingale 
\begin{equation}
X_t=\int_0^t\nabla_x V_f(Z_s)\cdot dB_s,\hskip.4cm  0<t\leq T,  
\end{equation}
where $V_f(x, t)$ is the heat extension $f$ to $\bR^{n+1}_{+}$ (as in \eqref{itomart}).  Then for 
$j\not =k$, the following inequalities hold:
\begin{equation}\label{U-inqR_jR_k}
\int_{\bR^n} V(f, 2R_jR_k f)\, dx \leq  \lim_{T\to\infty}E^T\,[U(X_T, A_{jk}*X_T)]\, \leq 0
\end{equation}
and 
\begin{equation}\label{U-inq(R_j^2-R_k^^2)}
\int_{\bR^n} V(f, (R_j^2-R_k^2) f)\, dx \leq \lim_{T\to\infty}E^T\,[U(X_T, \tilde{A}_{jk}*X_T)]\, \leq 0,
\end{equation}
where  $A_{jk}*X_T$ and $\tilde{A}_{jk}*X_T$ are the martingale transforms of $X_T$ corresponding to the operators $2R_jR_k$ and $R_j^2-R_k^2$, respectively, as constructed in \S3.4. 
\end{theorem}

The proof of this theorem consist, basically,  of two steps:  
\begin{itemize}
\item[(i)] We first remove the conditional expectation (the trivial step) on the projection operators $\mathcal{S}_{A}$ to obtain 
\begin{equation}\label{V-inqR_jR_k}
\int_{\bR^n} V(f, 2R_jR_k f)\, dx \leq  \lim_{T\to\infty}E^T\,[V(X_T, A_{jk}*X_T)]
\end{equation}
and 
\begin{equation}\label{V-inq(R_j^2-R_k^^2)}
\int_{\bR^n} V(f, (R_j^2-R_k^2) f)\, dx \leq  \lim_{T\to\infty}E^T\,[V(X_T, \tilde{A}_{jk}*X_T)]. 
\end{equation}
\item[(ii)]  We then use the Burkholder machinery to prove that  
\begin{equation}\label{VU-inqR_jR_k}
E^T\,[V(X_T, A_{jk}*X_T)] \leq E^T\,[U(X_T, A_{jk}*X_T)]\leq 0
\end{equation}
and 
\begin{equation}\label{VU-inq(R_j^2-R_k^^2)}
E^T\,[V(X_T, \tilde{A}_{jk}*X_T)] \leq E^T\,[U(X_T, \tilde{A}_{jk}*X_T)]\leq 0.  
\end{equation}
\end{itemize} 

In fact, under the assumption that $\|A\|\leq 1$, this is the basic strategy used in \S\ref{T_A} and \S\ref{S_A} for the operators $T_A$ and $\mathcal{S}_{A}$, not just for Riesz transforms. 

One ``deficiency" as of now in the investigations of the upper bounds for the $L^p$--norms of the Beurling-Ahlfors operator and the second order Riesz transforms is that there seems to be only one successful approach, the martingale approach.  This approach has ``done its job" for the Riesz transforms but has (so far) falling short for the Beurling-Ahlfors operator. We believe it would be of great interest to find a non-martingale approach to these problems. It is in this context that we raise the following question (as in \cite[p. 260]{BanLin1}) which even though it will not give any new information on the constants for $2R_jR_k$ and $R_j^2-R_k^2$ than what we already have, it may shed new light on Conjectures \ref{Iwaniec} and  \ref{BA-Bur-U} and related problems.  

\begin{question}\label{Ban-Lin-type}  Fix $f\in C_0^{\infty}(\bR^n)$, $n\geq 2$, $j\not =k$.  Is it the case that 
\begin{equation}\label{U-inqR_jR_k-NoPro} 
\int_{\bR^n} U(f,\, 2R_jR_kf)\,dx\leq 0 
\end{equation}
and 
\begin{equation}\label{U-inq(R_j^2-R_k^2)-NoPro}
\int_{\bR^n} U(f,\, (R_j^2-R_k^2)f)\,dx\leq 0,
\end{equation}
and that these inequalities can be proved without martingales? 
\end{question} 

Since $R_jR_k=\frac{\partial^2 }{\partial x_j\partial x_k}\Delta^{-1}$ with a similar definition for $R_j^2-R_k^2$, the inequalities \eqref{U-inqR_jR_k-NoPro} and 
\eqref{U-inq(R_j^2-R_k^2)-NoPro} can be stated as 

\begin{equation}\label{U-inqR_jR_k-NoPro-1} 
\int_{\bR^n} U(\Delta f,\, 2\frac{\partial^2 f}{\partial x_j\partial x_k} )\,dx\leq 0 
\end{equation}
and 
\begin{equation}\label{U-inq(R_j^2-R_k^2)-NoPro-1}
\int_{\bR^n} U(\Delta f,\, (\frac{\partial^2f}{\partial^2 x_j}-\frac{\partial^2f}{\partial^2 x_k}) )\,dx \leq 0. 
\end{equation}

It may be that the recent methods of Volberg and his collaborators for finding Burkholder functions 
via Monge-Amp\`ere equation \cite{VasVol0} and the  ``laminates" 
method for proving that the bound $(p^*-1)$ for the 
operator norm of $R_jR_k$ and $R_j^2-R_k^2$ (and their perturbations) \cite{NicVol1} is sharp, can shed some light on 
this question. 

If the function $U$ is replaced 
by the function $\tilde U$ in \eqref{minimalU} (which is convex in $w$ for fixed $z$) the inequalities \eqref{U-inqR_jR_k-NoPro} and \eqref{U-inq(R_j^2-R_k^2)-NoPro} can  be reduced to the martingale case.  Bu this again produces no new techniques.  This was pointed out to us by P. Janakiraman in a private communication.

\subsection{Quasiconformal mappings and the Burkholder function $U$} 
We give a brief account of some recent developments in which quasiconformal mappings (also nonlinear hyperelasticity) and  Burkholder's theory on sharp martingale inequalities share common problems of compelling interest. (We refer the reader to \cite{Iwa3} and \cite{AstIwaMar1} for details.) By definition, a weakly differentiable mapping $f :\Omega \rightarrow \bR^n$ in a domain  $\Omega\subset \bR^n$ (also referred to as hyperelastic deformation) is said to be $K$-quasiregular, $1\leq K < \infty ,$ if its Jacobian matrix  $D\!f(x) \in \M^{n\times n}$ (deformation gradient) satisfies the distortion inequality
\begin{equation}\label{distor}
 |D\!f(x)|^n  \leqslant K\, \det D\!f(x),\;\quad \textrm{where}\;\;\;|D\!f(x)|  = \max_{|v| =1} \; |D\!f(x)v|. 
\end{equation}
 
 The $L^p$-integrability of the derivatives of $K$-quasiregular mappings relies on a general inequality which is opposite to the distortion inequality in an average sense. More precisely, 
\begin{equation} \label{1}
 \int_{\mathbb R^n} \left \{\,|D\!F(x)|^n \;-\; K\, \det D\!F(x) \,\right \}\cdot |DF(x)|^{p-n}\; \textrm d x \;\geqslant  \; 0, 
\end{equation}
for all mappings $F \in \W^{1,p}(\bR^n, \bR^n)$ with the Sobolev exponents $p$ in a certain interval $ \alpha(n,K)  < p < \beta(n,K)$, where $ \alpha(n,K)  < n <\beta(n,K)$. 
Iwaniec (\cite{Iwa3}) conjectured that the largest such interval is
\begin{equation}\label{2}
\alpha (n,K)  = \frac{nK}{K+1} < p < \frac{nK}{K-1} = \beta(n,K).
\end{equation}
Iwaniec (see again \cite[pp. 518-523]{AstIwaMar1} and \cite{Iwa3}) then observed that in dimension $n=2\,$ the integrand in  (\ref{1}) is none other than the Burkholder's function $U$ (modulo constant factor), thus rank-one convex for all exponents $p$ in (\ref{2}). Inspired by Burkholder's results he proved, in every dimension $n\geqslant 2$, that (\ref{2}) defines precisely the range of the exponents $p$  for which the integrand in (\ref{1}) is rank-one convex; see \cite{Iwa3}. Now, it may very well be that Iwaniec's $n$-dimensional analogue of Burkholder's integral is also quasiconvex and, conjecturally, that (\ref{1}) holds for all $p$ in the range (\ref{2}). This would give a completion of the  $L^p$-theory of quasiregular mappings in space. 

While it is not clear at this point that martingale techniques will
produce the conjectured sharp bound $(p^*-1)$ for $\|B\|_{p}$ which motivated many of the applications presented in this paper, it seems likely that
the resolution of Iwaniec's conjecture will somehow involve the
Burkholder function $U$ and/or his ideas originally developed for his martingale inequalities. We again refer the reader  to \cite{AstIwaPraSak} which sheds some new light on this speculation.

In higher dimensions it is plausible
that Burkholder's vision and his sharp martingale inequalities will
contribute to the further development of the  $L^p$-theory of quasiregular
mappings with far reaching applications to geometric function theory in
$\bR^n$ and, in particular, mathematical models of nonlinear
hyperelasticity.  What is certainly clear is that as of now all approaches (stochastic integrals and
Bellman functions) which have produced concrete bounds close to the
conjectured bound for $\|B\|_{p}$ and for other operators, either completely rest upon or have been heavily influenced by the fundamental ideas of Burkholder originally conceived to prove sharp martingale inequalities. These  ideas have led to deep and
surprising connections in areas of analysis, geometry and PDE's which on the surface seem far removed from martingale theory. We hope this article will further elucidate some of this connections and that it will  serve as a starting point for further research explorations in this area of mathematics. Finally,  we hope that the many problems and conjectures listed throughout the paper can be resolved, by these or other techniques, in the near future.

\section{UCLA/Caltech 1984, Urbana 1986}
I first met Don Burkholder and heard him speak about his work on sharp martingale inequalities during a visit he made to UCLA in early 1984.  I was then a graduate student writing a thesis (published in \cite{Ban1}) under the direction of Rick Durrett on the $L^p$-boundedness of the Riesz transforms based on the Gundy-Varopoulos representation of these operators \cite{GunVor1}. This approach led to the more general ``$T_A$" operators discussed in \S3 whose $L^p$-boundedness reduces to the 
$L^p$-boundedness of martingale transforms for stochastic integrals. The bounds for the stochastic integrals  came from the  Burkholder-Gundy square function inequalities \cite{BurGun1}.   While I certainly did not follow Burkholder's lecture in its entirety, it was clear that in obtaining the $(p^*-1)$-constant Burkholder had bypassed the square function inequalities used in his 1966 paper \cite{Bur11}.  With some trepidation, given that I had never met Burkholder before and that about the only thing I knew about him was that he was a towering figure in probability and analysis, I decided to approach him with some questions. In particular, I wanted  to ask him about the possibility of bypassing the square function inequalities to prove the boundedness of the stochastic integrals needed in the applications to the Riesz transforms.  With the usual kindness and welcoming that characterizes all  his interactions with everyone he meets (from world famous mathematicians to unknown graduate students like myself at that time), he listened with interest and was very encouraging.  

In September of 1984 I moved from UCLA to Caltech as a postdoc where I learned the T. Iwaniec $(p^*-1)$ conjecture ({\it Conjecture} \ref{Iwaniec} below) from the late Tom Wolff.  I immediately went back to the operators $T_A$ studied in my thesis and realized that one could in fact represent the Beurling-Ahlfors operator as one of these $T_A's$ and thus prove its $L^p$ boundedness from the martingale inequalities.  Even more, from this and the available estimates for the Burkholder-Gundy square function inequalities due to Davis \cite{Dav2}, one could already give a rather explicit bound for the $L^p$-norm of the Beurling-Ahlfors operator. After many conversations with Tom about whether Iwaniec's $(p^*-1)$ had anything to do with Burkholder's $(p^*-1)$,  my interest on bypassing the Burkholder-Gundy square function inequalities to study the $L^p$ boundedness of the operators $T_A$ intensified.  While some sharp stochastic integral inequalities followed from \cite{Bur39} by reducing to the discrete-time  martingale case, those needed for the application to the Beurling-Ahlfors transform did not follow, at least not in any direct way.  In addition, the techniques in \cite{Bur39} were so difficult and so new at that time that it was not at all clear how to adapt them to prove the needed martingale inequalities for the applications to the Beurling-Ahlfors operator. (Nearly thirty years later the situation has changed and we now have a much better understanding of Burkholder's techniques thanks to the work of many of those cited on this paper.)

In 1986, I received an NSF postdoctoral fellowship to go to the University of Illinois at Urbana with Don as my {``Sponsoring Senior Scientist,"} to use the language of the NSF.  The timing could not have been better.  In addition to the already exciting mathematical environment fostered by the many distinguished local probabilists and analysts, 1986-87 was a ``Special Year in Modern Analysis" at Urbana with many lectures and mini-courses taught by long-term and short-term outside visitors. There were also other postdoctoral visitors and many graduate students with whom I interacted. All this and the fact that I had no teaching duties created a superb mathematical environment for me.  During my first semester there I learned from Don of his discovery of an explicit expression for the function $U$ in (\ref{u}).   With this function at hand and the It\^o formula, one could then proceed to explore the sought-after stochastic integral inequalities and their applications to the Beurling-Ahlfors operator, finally  bypassing the Burkholder-Gundy square function inequalities used up to that point for these types of applications.  This is what was done in the paper \cite{BanWan3} written jointly with G. Wang, a student of Don whom I met during my 1986-87 year at Urbana. 

While not related to the topic of this paper, I take this opportunity to gratefully acknowledge the many mathematical conversations I had during my year in Urbana, and in subsequent years, with our departed colleagues and friends, Frank Knight and Walter Philipp. It was from Walter that year that I learned about the power of the invariance principle \cite{PhiSto1} when applied to weakly dependent ``non-probabilistic" objects, like lacunary trigonometric series.  This circle of ideas and basic philosophy of looking for (and finding) probabilistic behavior such as central limit theorem and  laws of the iterated logarithm  in non-probabiltic objects, were invaluable in some of the work presented in \cite{BanMoo1}. Frank and Walter were both special individuals and I greatly treasure the memories of my interactions with them and their hospitality and kindness during my year in Urbana.

\section{Acknowledgments}

Much of the work presented in this paper which has the author's name associated with it has been carried out over several years  in collaborations with several people. They include Gang Wang, Arthur Lindeman, Pedro M\'endez-Hern\'andez, Prabhu Janakiraman and Krzysztof Bogdan. I am grateful to these colleagues and friends for our collaborations. Over the years I have benefited greatly from discussions about problems related to the topics of this paper with Al Baernstein, Oliver Dragi\v cevi\'c, Tadeusz Iwaniec, Richard Gundy, Joan Verdera, Alexander Volberg, and the late Tom Wolff.  I thank them for sharing their mathematical ideas, interest and excitement for the problems discussed here with me over the years.    Special thanks are due to Prabhu Janakiraman for his careful reading of this paper and for his many comments and suggestions which considerably improved the presentation.  I thank Tuomas Hyt\"onen for suggesting some of the references to UMD Banach space valued singular integrals operators.  I am grateful to Fabrice Baudoin and  Adam Os\c{e}kowski for various general comments and corrections and for the many conversations we have had in recent months and weeks on topics related to those treated here. I thank the referee for the numerous useful comments. 

I specially want to thank Dick Gundy for his support and encouragement at the time (mid 80's) when not as many people were interested on some of the problems discussed here nor on the general conditional expectation operators $T_A$ where my interest for this area of analysis began. 

By far my biggest thanks are reserved for Don Burkholder. Without his
groundbreaking work on martingale inequalities much of what is
presented in this paper would not exist. In addition, Don's kind 
support and encouragement throughout my mathematical career have been
invaluable.  Don has not only provided mathematical inspiration but also friendship and wise counsel whenever needed. For all this, I am deeply indebted to him.


\begin{thebibliography}{99}

\bibitem{AbrSte1} M. Abramowitz and I.A. Stegun, {\em Handbook of Mathematical Functions with Formulas,} Graphs, and Mathematical Tables. New York, Dover, 1972.

 \bibitem{App} D. Applebaum, {\em L\'evy Processes and Stochastic Calculus,} Cambridge University Press, 2004.

\bibitem{App1} D. Applebaum, 
{\em L\'evy Processes-From Probability to Finance and Quantum Groups,}  Notices of the AMS, {\bf 51} (2004),  1336--1347.




\bibitem{Arc1} N. Arcozzi, {\em Riesz transforms on compact Lie groups, spheres and Gauss space,}  Ark. Mat.  {\bf 36}  
(1998),  201--231.

\bibitem{ArcLi1} N. Arcozzi and X. Li,  {\em Riesz transforms on spheres,}  Math. Res. Lett.  {\bf 4}  (1997),  
401--412.

\bibitem{Ast1} K. Astala, \emph{Area Distortion of Quasiconformal Mappings,} Acta Math. {\bf 173} (1994), 37--60.


\bibitem{AstIwaMar1} K. Astala, T. Iwaniec and G. Martin, \emph{Elliptic Partial Differential Equations and Quasiconformal Mappings in the Plane,} Princeton University Press, 2009. 

\bibitem{AstIwaPraSak} K. Astala, T. Iwaniec, I. Prause, E. Saksman, \emph{ Burkholder integrals, Morrey's problem and quasiconformal mappings}  {\bf (preprint)} arXiv--http://arxiv.org/abs/1012.0504. 



\bibitem{Bal1} J.M. Ball, {\em  Convexity conditions and existence 
theorems in nonlinear elasticity, } Arch. Rational Mech. Anal. {\bf 63} (1977), 337--403.

\bibitem{Ban1} R. Ba\~nuelos, {\em Martingale transforms and related singular integrals,} Trans. Amer. Math. Soc., {\bf 293} (1986), 547-563.

\bibitem{Ban1.1} R. Ba\~nuelos, {\em A sharp good-$\lambda$ inequality with an application to Riesz transforms,} Michigan Math. J. {\bf 35} (1988), 117--125. 

\bibitem{BanBau1} R. Ba\~nuelos and F. Baudoin, \emph{Martingale transforms on manifolds,} {(\bf preprint)}. 

\bibitem{BanBieBog1} R. Ba\~nuelos, A. Bielaszewski and K. Bogdan,
{\em Fourier multipliers for non-symmetric L\'evy processes,} ({\bf preprint})


\bibitem{BanBog1} R. Ba\~nuelos and K. Bogdan, \emph{L\'evy processes and 
Fourier multipliers,} J.~Funct.~Anal. {\bf 250} (2007), 197--213. 

\bibitem{BanDav1} R. Ba\~nuelos and B. Davis, {\em Donald Burkholder's work in martingales and analysis,} Selected Works of 
Donald L. Burkholder,  Springer 2011 (B. Davis and R. Song, Editors). 

\bibitem{BanJan2} R. Ba\~nuelos and P. Janakiraman, {\em On the weak-type constant of the
Beurling-Ahlfors Transform}, Michigan Math. J. {\bf 58} (2009), 339-257.

\bibitem{BanJan1} R. Ba\~nuelos and P. Janakiraman, \emph{$L^p$-bounds for the Beurling-Ahlfors 
transform,} Trans. Amer. Math. Soc. {\bf 360} (2008), 3603--3612.

\bibitem{BanLin1} R. Ba\~nuelos and A.J. Lindeman, \emph{A Martingale study of the Beurling-Ahlfors
transform in $\bR^n$,} Journal of Functional Analysis. {\bf 145} (1997), 224--265.


\bibitem{BanMen1} R. Ba\~nuelos and P. M\'endez-Hern\'andez,
 \emph{Space-time Brownian motion and the
Beurling-Ahlfors transform,}  Indiana University Math J. {\bf 52} (2003), 981--990.

\bibitem{BanMoo1} R. Ba\~nuelos and C. Moore, \emph{Probabilistic behavior of harmonic Functions}
Birk\"auser, July 1999.

\bibitem{BanOse1} R. Ba\~nuelos and A. Os\c{e}kowski, \emph{Burkholder inequalities for submartingales, 
Bessel processes and conformal martingales,} ({\bf preprint}). 

\bibitem{BanOse2} R. Ba\~nuelos and A. Os\c{e}kowski, \emph{Martingales and Sharp Bounds for Fourier multipliers,} ({\bf preprint}).

\bibitem{BanWan3} R. Ba\~nuelos and G. Wang,  \emph{Sharp inequalities for
 martingales with applications 
to the Beurling-Ahlfors and Riesz transforms,}  Duke Math. J. {\bf 80} (1995), 575--600.

\bibitem{BanWan2} R. Ba\~nuelos and G. Wang, \emph{Sharp 
inequalities for martingales under
orthogonality and differential subordination,}  
Illinois Journal of Mathematics {\bf 40} (1996), 687--691.


\bibitem{BanWan1} R. Ba\~nuelos and G. Wang, {\em Davis's inequality for orthogonal martingales under differential subordination,} Michigan Math. J., {\bf 47}  (2000), 109-124.





\bibitem{BaeSmt1} A. Baernstein II and S. J. Montgomery-Smith, {\em Some conjectures about integral means of 
$\partial f$ and $\overline\partial f$} in Complex Analysis and Differential Equations 
(Uppsala, Sweden, 1999), ed. Ch. Kiselman, Acta. Univ. Upsaliensis Univ. C Organ. Hist. 64, Uppsala Univ. 
Press, Uppsala, Sweden, (1999),  92--109.


\bibitem{Bou1} J. Bourgain, \emph{Some remarks on Banach spaces in which martingale difference sequences
are unconditional,} Ark. Mat., {\bf 21}  (1983), 163--168.

\bibitem{Bou2} J. Bourgain, {\em Vector-valued singular integrals and the H1-BMO duality,} In Probability theory
and harmonic analysis (Cleveland, Ohio, 1983), volume 98 of Monogr. Textbooks Pure Appl.
Math., Dekker, New York, 1986.


\bibitem{BorJanVol1} A. Borichev, P. Janakiraman
and A. Volberg, {\em Subordination by orthogonal martingales in $L^p$ and zeros of Laguerre polynomials,} {\bf (Preprint)}. 

\bibitem{BorJanVol2} A. Borichev, P. Janakiraman
and A. Volberg, {\em on Burkholder function for orthogonal
martingales and zeros of Legendre polynomials,} Amer. Jour. Math. (to appear). 

\bibitem{NicVol1}  A. Borichev 
and A. Volberg, {\em The $L^p$--operator norm of a quadratic
perturbation of the real part of the
Alhfors-Beurling operator} {\bf (preprint)}

\bibitem{Bur11} D.L. Burkholder,  {\em Martingale transforms,} Ann. Math. Statist. {\bf 37} (1966), 1494-1504.



\bibitem{Bur31} D.L. Burkholder,  {\em A sharp inequality for martingale transforms,} Ann. Prob {\bf 7} (1979), 858-863.


\bibitem{Bur35} D.L. Burkholder,	{\em A geometrical characterization of Banach spaces in which martingale difference sequences are unconditional,}  Ann. Probab. {\bf 9} (1981), 997-1011.

\bibitem{Bur37}  D.L. Burkholder, {\em A nonlinear partial differential equation and the unconditional constant of the Haar system in $L^p$}, Bull. Amer. Math. Soc. {\bf 7} (1982), 591-595.

\bibitem{Bur38}  D.L. Burkholder,  	{\em A geometric condition that implies the existence of certain singular integrals of Banach-space-valued functions,}  in ``Conference on Harmonic Analysis in Honor of Antoni Zygmund," (Chicago, 1981), edited by William Beckner, Alberto P. Calder\'on, Robert Fefferman, and Peter W. Jones. Wadsworth, Belmont, California, 1983, pp. 270-286.

\bibitem{Bur39} D.L. Burkholder, {\em Boundary value problems and sharp inequalities for martingale transforms,} Ann. Probab. {\bf 12} (1984), 647-702.

\bibitem{Bur40}  D.L. Burkholder, {\em An elementary proof of an inequality of R. E. A. C. Paley,} 
Bull. London Math. Soc. {\bf 17}  (1985), 474-478.

\bibitem{Bur42}  D.L. Burkholder, {\em Martingales and Fourier analysis in Banach spaces,} C.I.M.E. Lectures (Varenna (Como), Italy, 1985), Lecture Notes in Mathematics 1206 (1986), 61-108.

\bibitem{Bur43}  D.L. Burkholder, {\em A sharp and strict $L^p$-inequality for stochastic integrals,} Ann. Probab. {\bf 15} (1987 ), 268-273.

\bibitem{Bur44} D.L. Burkholder, {\em A proof of Pelczy'nski's conjecture for the Haar system,} Studia Math. {\bf 91}
(1988), 79-83.

\bibitem{Bur45}  D.L. Burkholder,	{\em Sharp inequalities for martingales and stochastic integrals,} Colloque Paul L\'evy (Palaiseau,1987), Ast'erisque 157-158 (1988), 75-94.

\bibitem{Bur46}  D.L. Burkholder,	{\em Differential subordination of harmonic functions and martingales, Harmonic Analysis and Partial Differential Equations,} (El Escorial, 1987), Lecture Notes in Mathematics 1384 (1989), 1-23.

\bibitem{Bur47} D.L. Burkholder, {\em On the number of escapes of a martingale and its geometrical significance,} in ``Almost Everywhere Convergence," edited by Gerald A. Edgar and Louis Sucheston. Academic Press, New York, 1989, 159-178.

\bibitem{Bur48} D.L. Burkholder, {\em Explorations in martingale theory and its applications,} 'Ecole d'Et'e de Probabilit\'es de Saint-Flour XIX-1989, Lecture Notes in Mathematics 1464 (1991), 1-66.

\bibitem{Bur50}  D.L. Burkholder,	{\em Strong differential subordination and stochastic integration,} Ann. Probab. {\bf 22} (1994), 995-1025.

\bibitem{Bur51}  D.L. Burkholder, 	{\em Sharp norm comparison of martingale maximal functions and stochastic integrals,} Proceedings of the Norbert Wiener Centenary Congress (East Lansing, MI, 1994), 343--358, Proc. Sympos. Appl. Math. 52, Amer. Math. Soc., Providence, RI (1997).

\bibitem{Bur52}  D.L. Burkholder, {\em Some extremal problems in martingale theory and harmonic analysis,} In Harmonic Analysis and Partial Differential Equations (Chicago, 1996), 99--115, Chicago Lectures in Math. Univ. Chicago Press, Chicago, 1999.


\bibitem{Bur53}  D.L. Burkholder, {\em Martingales and singular integrals in Banach spaces}, Handbook on the Geometry of Banach spaces, Volume 1, edited by William B. Johnson and Joram Lindenstrauss, Elsevier, (2001) 233-269.

\bibitem{Bur54}  D.L. Burkholder, {\em The best constant in the Davis inequality for the expectation of the martingale square function,} Trans. Amer. Math. Soc. {\bf 354} (2002), 91--105.

\bibitem{BurGun1} D.L. Burkholder and R. F. Gundy, {\em Extrapolation and interpolation of 
quasi-linear operators on martingales,} Acta Math. {\bf 124} (1970), 249-304 .








\bibitem{Cho3} C. Choi, {\em A submartingale inequality,} Proc. Amer. Math. Soc. {\bf 124} (1996). 2549--2553

\bibitem{Cho4} C. Choi, {\em A weak--type inequality for differentially subordinate
harmonic functions,} Tran. Amer. Math. Soc. 
{\bf 350} (1998), 2687-2696.

\bibitem{Cho2} K.P. Choi, {\em Some sharp inequalities for martingale transforms}, Trans. Amer. Math. Soc. {\bf 307} (1988), 279--300.


\bibitem{Cho1} K.P. Choi, {\em  A sharp inequality for martingale transforms and the unconditional basis constant of a monotone basis in $L^p(0,1)$,} Trans. Amer. Math. Soc. {\bf 330} (1992), 509--529.

\bibitem{ConTan1} R. Cont and P. Tankov, {\em Financial Modelling with Jump Processes,} Chapman \& Hall/CRC, 2004.





\bibitem{Dac1} B. Dacoronga, {\em Direct Methods in the Calculus of Variations,} Springer
1989.

\bibitem{Dav1} B. Davis, {\em On the weak type (1,1) inequality for conjugate
functions,}  Proc. Amer. Math. Soc. {\bf 44} (1974), 307-311. 

\bibitem{Dav2} B. Davis,  \emph{On the Lp norms of stochastic integrals and other martingales,}
Duke Math. J. {\bf 43} (1976), 697-704. 

\bibitem{DelMey1} C. Dellacherie and P.-A. Meyer, {\em Probabilities and Potential, B. Theory of Martingales,} North-Holland Math. Stud.,vol. 72, North-Holland, Amsterdam, 1982. Translated from the French by J.P. Wilson

\bibitem{DonSul1} S. Donaldson and D. Sullivan, \emph{Quasiconformal 4--manifolds,} 
Acta Math. {\bf 163} (1989),
181--252. 

\bibitem{Dra1} O. Dragi\v cevi\'c,  {\em Some remarks on the $L^p$ estimates for powers of the
Ahlfors-Beurling operator,} Archiv der Mathematik. {\bf 97} (2011), 463--471. 

\bibitem{DraPetVol1} O. Dragi\v cevi\'c, S. Petermichl and A. Volberg,  \emph{A rotation method which gives linear $L^p$ estimates for powers of the Ahlfors-Beurling operator,} J. Math. Pures Appl. {\bf 86} (2006), no. 6, 492--509.


\bibitem{DraVol1} O. Dragi\v cevi\'c and A. Volberg, 
\emph{Bellman functions and dimensionless estimates of Littlewood-Paley type,} J. Operator Theory {\bf 56} (2006), 167--198.

\bibitem{DraVol4} O. Dragi\v cevi\'c and A. Volberg,
\emph{Sharp estimate of the Ahlfors-Beurling 
operator via averaging martingale transforms,}  
Michigan Math. J.  {\bf 51}  (2003), 415--435. 

\bibitem{DraVol2} O. Dragi\v cevi\'c and A. Volberg, \emph{Bellman function, Littlewood-Paley estimates and asymptotics for the Ahlfors-Beurling operator in $L^p(\bC)$,} Indiana Univ. Math. J. {\bf 54} (2005), no. 4, 971--995.

\bibitem{DraVol3} O. Dragi\v cevi\'c and  A. Volberg, 
\emph{Bellman function for the estimates of Littlewood-Paley type and 
asymptotic estimates in the $p-1$ problem,} C. R. Math. Acad. Sci. Paris {\bf 340} (2005), no. 10, 731--734. 








\bibitem{DuoRub1} J.  Duoandikoetxea and J.L. Rubio de Francia, {\em Estimations ind\'ependantes de la dimension pour les transform\'ees de Riesz}, C. R. Acad. Sci. Paris S\'er. I Math. {\bf 300} (1985),193--196.


\bibitem{Dur1} R.Durrett, {\em Brownian Motion and Martingales
in Analysis},
Wadsworth,  Belmont, CA, 1984.

\bibitem{Ess1} M. Ess\'en, {\em A superharmonic proof of the M. Riesz
conjugate function theorem,} Ark.\ Math. {\bf 22} (1984), 
281--288.  

\bibitem{FabGutSct1} E. Fabes, C. Guti\'errez, R. Scotto, {\em Weak-type estimates for the Riesz transforms 
associated with the Gaussian measure,}  Rev. Mat. Iberoamericana  {\bf 10}  (1994),  229--281.




\bibitem{Fig1} T. Figiel,  {\em Singular integral operators: a martingale approach,} In Geometry of Banach spaces
(Strobl, 1989), volume 158 of London Math. Soc. Lecture Note Ser., Cambridge
Univ. Press, Cambridge (1990), 95--110. 

\bibitem{Fuj1} T. Fujita, \emph{On some properties of holomorphic diffusion  processes,} Hitotsubashi
Journal of Arts and Sciences {\bf 34} (1993), 83--90.

\bibitem{Fuk1} M. Fukushima and M. Okada, \emph{On conformal martingale diffusions and pluripolar sets,} J. Functional Analysis {\bf 55} (1984),377--388.

\bibitem{Fuk2} M. Fukushima, {\em On the continuity of plurisubharmonic functions along con-
formal diffusions,} Osaka J. Math {\bf 23} (1986), 69--75.

\bibitem{GeiSmiSak1} E. Geiss, S. Montgomery-Smith, E. Saksman,
\emph{On singular integral and martingale transforms,} Trans. Amer. Math. Soc. {\bf 362} (2010), 553--575.

\bibitem{GerRei} F. W. Gehring and E. Reich, \emph{Area distortion under
 quasiconformal mappings,}  Ann. Acad. Sci. 
 Fenn. Ser A I {\bf 388} (1966), 1--15. 
 
  
\bibitem{GetSha1} R.K.Getoor and M.J. Sharpe, {\em Conformal martingales,} Invent. Math. {\bf 16} (1972), 271--308.


\bibitem{Gra1} L. Grafakos, {\em Classical and Modern Fourier Analysis,} Pearson Education Inc., 2004. 

\bibitem{Gil1} J. T. Gill, {\em on the Beurling Ahlfors transform's
weak-type constant,} Michigan Math. J. {\bf 59} (2010), 353--363.

\bibitem{Gro1} L. Gross, {\em Abstract Wiener spaces,} Proc. Fifth Berkeley Sympos. 
Math. Statist. and Probability, Vol. II,  Contributions to Probability Theory, Univ. California Press, (1965-66), 31-42.


\bibitem{Gun1} R.F. Gundy, {\em Some Topics in Probability and Analysis. CBMS Regional Conference Series in Mathematics,} {\bf 70}  American Mathematical Society, Providence, RI, 1989.

\bibitem{Gun2} R. F. Gundy, {\em Sur les transformations de Riesz
 pour le semi-groupe d'Ornstein-Uhlenbeck, (French) 
[Riesz transformation on the Ornstein-Uhlenbeck process]}
 C. R. Acad. Sci. Paris Ser. I Math.  {\bf 303}  (1986), 
967--970.

\bibitem{GunSil1} R.F. Gundy and M.L. Silverstein, {\em On a probabilistic interpretation for the Riesz transforms,} Functional
Analysis inMarkov Processes, Lect. Notes in Math., {\bf 923}, Springer,
Berlin (1982), 199--203,


\bibitem{GunVor1} R.F. Gundy and N. Th. Varopoulos, {\em Les transformations de Riesz et les int\'egrales stochastiques,} C. R. Acad. Sci. Paris S\'er. A-B {\bf 289} (1979), A13--A16.

\bibitem{Ham2} W. Hammack, {\em Sharp inequalities for the distribution of a stochastic integral in which the integrator is a bounded submartingale,} Ann. Probab. {\bf 23} (1995), 223--235.


\bibitem{Ham1} W. Hammack, {\em Sharp maximal inequalities for stochastic integrals in which the integrator is a submartingale,} Proc. Amer. Math. Soc. {\bf 124} (1996), 931--938.


\bibitem{Gut1} C. Guti\'errez,  {\em On the Riesz transforms for
 Gaussian measures,}  J. Funct. Anal.  {\bf 120}  (1994), 107--134.
 
 \bibitem{Hyt3} T.P. Hyt\"onen: \emph{On the norm of the Beurling-Ahlfors operator in
several dimensions,}   Canad. Math. Bull. {\bf 54} (2011), 113--125.

\bibitem{Hyt00} T.P. Hyt\"onen, {\em Vector-valued extension of linear operators, and Tb theorems,} Vector measures, integration and related topics, 245--254, Oper. Theory Adv. Appl., 201, Birkh\"auser Verlag, Basel, 2010

 

\bibitem{Hyt0} T.P. Hyt\"onen, {\em  Vector-valued singular integrals, and the border between the one-parameter and the multi-parameter theories,}  CMA/AMSI Research Symposium ``Asymptotic Geometric Analysis, Harmonic Analysis, and Related Topics'', Proc. Centre Math. Appl. Austral. Nat. Univ., {\bf 42}, Austral. Nat. Univ., Canberra, (2007), 11-41.

 
 \bibitem{Hyt1} T.P. Hyt\"onen, {\em Littlewood-Paley-Stein theory for semigroups in UMD spaces,} Rev. Mat. Iberoam. {\bf 23}(2007), 973--1009.
 


\bibitem{Hyt2} T.P. Hyt\"onen, {\em Aspects of probabilistic Littlewood-Paley theory in Banach spaces,} Banach spaces and their applications in analysis, Walter de Gruyter, Berlin, (2007),  343--355,



\bibitem{Iwa1} T. Iwaniec, {\em Extremal inequalities in Sobolev spaces and quasiconformal mappings,} Z. Anal. 
Anwendungen {\bf 1} (1982), 1--16.


\bibitem{Iwa2} T. Iwaniec, \emph{$L^p$-theory of quasiregular mappings,}
in Quasiconformal Space Mappings, Ed. Matti 
Vuorinen, Lecture Notes in Math. 1508, Springer, Berlin, 1992. 

\bibitem{Iwa3} T. Iwaniec, \emph{Nonlinear Cauchy-Riemann operators in ${\Bbb R}\sp n$,} Trans. Amer. Math.
Soc. {\bf 354} (2002), 1961--1995.


\bibitem{IwaMar0} T.Iwaniec, and G.J. Martin, {\em Quasiregular mappings in even dimensions,} Acta Math. {\bf 170} (1993), 29--81. 


\bibitem{IwaMar1} T. Iwaniec and G. Martin, {\em Riesz transforms and related singular integrals,} J. Reine Angew. 
Math. {\bf 473} (1996), 25--57.

\bibitem{IwaMar2} T. Iwaniec, and G. J. Martin, {\em Geometric Function Theory and Nonlinear Analysis}, Oxford University Press, 2001.


\bibitem{Jan0} P. Janakiraman, {\em Orthogonality in complex martingale spaces
and connections with the Beurling-Ahlfors transform,} {(\bf preprint)}. 


\bibitem{Jan}  P. Janakiraman, {\em Weak-type estimates for singular integrals and the Riesz transform,} Indiana Univ. Math. J. {\bf 53} (2004), 533--555.

\bibitem{Jan1}  P. Janakiraman,  {\em Best weak-type $(p,p)$ constants, $1\leq p\leq 2$, for orthogonal harmonic functions and martingales,} Illinois J. Math. {\bf 48}  (2004), 909--921.


\bibitem{Lar1} L. Larsson-Cohn, {\em On the constants in the Meyer inequality,} Monatsh. Math. 
{\bf 137} (2002), 51--56.

%
 \bibitem{Leh1} O. Lehto,
\emph{Remarks on the integrability of the derivatives of quasiconformal mappings,}
Ann. Acad. Sci. Fenn. Series AI Math. {\bf 371} (1965), 8 pp.

\bibitem{Li1} X.-D. Li,  {\em Riesz transforms on forms and $L^p$--Hodge decomposition on complete Riemannian manifolds,}
Rev. Mat. Iberoam. {\bf 26} (2010), 481--528

\bibitem{Li2} X.-D. Li, {\em On the weak $L^p$--Hodge decomposition
and Beurling--Ahlfors transforms on complete Riemannian manifolds,} Probab. Theory Relat. Fields {\bf 150} (2011), 111--144.  

\bibitem{Li3} X.-D. Li, {\em Martingale transforms and $L^p$--norm estimates of Riesz transforms on complete Riemannian
manifolds,} Probab. Theory Relat. Fields {\bf 141} (2008), 247--281. 

\bibitem{Mar1} J. Marcinkiewicz, {\em Quelques theoremes sur les series orthogonales,} Ann. Soc. Polon. Math. {\bf 16} (1937), 84--96.



\bibitem{McC1} T. McConnell, {\em On Fourier multiplier transformations of Banach-valued functions,} Trans. Amer. Math. Soc. {\bf 285} (1984), 739--757.

\bibitem{Mel1} A. D. Melas, {\em The Bellman functions of dyadic-like maximal
operators and related inequalities,}  Advances in Mathematics,  {\bf 192} (2005), 310-340.


\bibitem{Mel2} A. D. Melas, {\em Dyadic-like maximal operators on LlogL functions,} Journal of Functional Analysis {\bf 257} (2009), 1631-1654.



\bibitem{Mey2} P.A. Meyer, {\em Le dual de $H^1$: D\'emonstration probabiliste. S\'eminaire de Probab., XI,}
Lect. Notes in Math., 581, Springer, Berlin (1977), 132--195.

\bibitem{Mey0} P.A. Meyer, {\em Retour sur la th\'eorie de Littlewood-Paley, Seminaire de probabilit\'es, XV,} 
Lect. Notes in Math., 850, Springer, Berlin (1979/20), 151-166.

\bibitem{Mey1} P.A. Meyer, {\em Transformations de Riesz pour les lois gaussiennes,} S\'eminaire de Probab., XVIII, Lect. Notes in Math., {\bf 1059}, Springer, Berlin (1984), 179--193.



\bibitem{Mor1} C.B. Morrey, \emph{Quasi-convexity and the lower semicontinuity of multiple integrals,} 
Pacific J. Math., {\bf 2}  (1952), 25-53. 

\bibitem{Mor2} C.B. Morrey, {\em Multiple integrals in the calculus of variations,} Die Grundlehren der mathematischen Wissenschaften, Band 130 Springer-Verlag New York, Inc., New York 1966.



\bibitem{NazTre1} F.L. Nazarov and S.R. Treil, {\em The hunt for a Bellman function: applications to estimates for singular integral operators and to other classical problems of harmonic analysis,} St. Petersburg Math. J. {\bf 8} (1997), 721--824. 

\bibitem{NazTreVol1} F. Nazarov, S. Treil and A. Volberg, {\em The Bellman functions and two-weight inequalities for Haar
multipliers,} J. Amer. Math. Soc. {\bf 4} (1999), 909-928.

\bibitem{NazTreVol2} F. Nazarov, S. Treil and A. Volberg, {\em Bellman function in stochastic optimal control and harmonic analysis (how our Bellman function got its name),} Oper. Theory: Adv. Appl. {\bf 129} (2001),  393-424.



\bibitem{NazVol2} F. Nazarov and  A. Volberg, {\em Heat extension of the Beurling operator and estimates for its norm,} St. Petersburg Math. J. {\bf 15}, (2004), 563-573.

\bibitem{OksSul1} B. \O ksendal and A. Sulem, {\em Applied Stochastic Control of Jump Diffusions,} Springer, 2004.


\bibitem{Osc6} A. Os\c{e}kowski, {\em Sharp inequality for bounded submartingales and their differential subordinates,} Electron. Commun. Probab. {\bf 13} (2008), 660--675.
 
\bibitem{Osc7} A. Os\c{e}kowski, {\em Sharp maximal inequality for stochastic integrals}. Proc. Amer. Math. Soc. {\bf 136} (2008), 2951--2958. 

\bibitem{Osc1} A. Os\c{e}kowski, {\em Sharp weak-type inequalities for differentially subordinated martingales,} 
Bernoulli {\bf 15} (2009), 871--897.

\bibitem{Osc2} A. Os\c{e}kowski, {\em Sharp norm inequalities for stochastic integrals in which the integrator is a nonnegative supermartingale,} Probab. Math. Statist. {\bf 29} (2009), no. 1, 29--42.

\bibitem{Osc3} A. Os\c{e}kowski, {\em On the best constant in the weak type inequality for the square function of a conditionally symmetric martingale,} Statist. Probab. Lett. {\bf 79}  (2009), 1536--1538.


\bibitem{Osc4} A. Os\c{e}kowski, {\em Weak type inequality for the square function of a nonnegative submartingale,} Bull. Pol. Acad. Sci. Math. {\bf 57} (2009), 81--89.

\bibitem{Osc5} A. Os\c{e}kowski, {\em Sharp maximal inequality for martingales and stochastic integrals. Electron,} Commun. Probab. {\bf 14} (2009), 17--30.
 
 \bibitem{Osc8} A. Os\c{e}kowski, {\em Sharp inequalities for differentially subordinate harmonic functions and martingales,} Canadian Mathematical Bulletin, ({\bf to appear}).


\bibitem{Ube1}  J. Ub\O e, {\em Conformal martingales and analytic functions,} Math. Scand. {\bf 60}
(1987), 292--309.

\bibitem{Pal} R. E. A. C. Paley, {\em A remarkable series of orthogonal}
functions I, { Proc. London Math. Soc.} {\bf 34} (1932) 241--264.

\bibitem{Pel1} A. Pe\l czy\'nski, {\em Norms of classical operators in function spaces,} Colloque Laurent Schwartz, Ast\'erisque {\bf 131} (1985), 137--162.

\bibitem{PetSlaBre1} S. Petermichl, L. Slavin, and B. Wick, {\em New estimates for the Beurling-Ahlfors operator on differential forms} {\bf (preprint)}. 

\bibitem{PetVol1} S. Petermichl and A. Volberg, {\em Heating of the Beurling operator: weakly quasiregular
maps on the plane are quasiregular,} Duke Math.  {\bf 112} (2002), 281--305.

\bibitem{Pic1} S.K. Pichorides, {\em On the best values of the constants in the theorems of M. Riesz, Zygmund and Kolmogorov,} Collection of articles honoring the completion by Antoni Zygmund of 50 years of scientific activity, II. Studia Math. {\bf 44} (1972), 165--179.

\bibitem{PhiSto1} W. Philipp and W. Stout, \emph{Almost sure invariance principles for partial sums of weakly dependent random variables,}
Mem.  Amer. Math. Soc. {\bf 161} (1975).  

\bibitem{Pis0} G. Pisier, {\em Don BurkholderÕs work on Banach spaces}, Selected Works of 
Donald L. Burkholder,  Springer 2011 (B. Davis and R. Song, Editors). 
 

\bibitem{Pis1} G. Pisier, {\em Riesz transforms: simpler analytic proof of P.-A. Meyer's inequality},  S\'eminaire de Probabilit\'es, XXII, 485--501, Lecture Notes in Math. {\bf 1321}, Springer, Berlin, 1988.


\bibitem{Pro1} P.E. Protter, {\em Stochastic Integration and Differential Equations,} second ed., Stoch. Model. Appl. Probab., vol. {\bf 21}, Springer-Verlag, Berlin, 2004.


\bibitem{Sat1} K-I. Sato, {\em  L\'evy Processes and Infinitely Divisible Distributions,}
{Cambridge University Press, Cambridge,  1999}.




\bibitem{Ste2} E. M. Stein,  {\em Singular integrals and
Differentiability Properties of Functions,} Princeton University
Press, Princeton, 1970.

\bibitem{Ste0} E. M. Stein, {\em Topics in harmonic analysis related to the Littlewood-Paley theory,} Princeton Univ. Press, Princeton, NJ, 1970. 


\bibitem{Ste4} E.M. Stein, {\em Some results in Harmonic Analysis in $\bR^n$ for
$n\rightarrow \infty$.} Bull. Amer. Math. Soc. {\bf 9} (1983), 71-73.

\bibitem{Ste5} E.M. Stein,  {\em Problems in harmonic analysis related to curvature
and oscillatory integrals, Proceedings of the International Congress of
mathematicians,} Berkeley, CA., 1986

 

\bibitem{Suh1} J. Suh, {\em A sharp weak type $(p,p)$ inequality $(p>2)$ for martingale transforms and other subordinate martingales,} Trans. Amer. Math. Soc. {\bf 357} (2005), 1545--1564. 

\bibitem{Sve1} V. \v{S}ver\'ak, {\em Examples of rank-one convex functions,} Proc. Roy. Soc. Edinburgh {\bf 114A} (1990), 237--242. 

\bibitem{Sve2} V. \v{S}ver\'ak,  {\em Rank-one convexity does not imply quasiconvexity,} Proc. Roy. Soc. Edinburgh {\bf 120A} (1992), 185--189.

\bibitem{Sve3} V. \v{S}ver\'ak,  {\em New examples of quasiconvex functions,} Arch. Rational Mech. Anal. {\bf 119} (1992), 293-300.

\bibitem{Tor} A. Torchinsky, {\em Real Variable Methods in
Harmonic  Analysis,} Academic Press, Inc. Orlando, FL,
1986.

\bibitem{Var1} N. Th. Varopoulos, {\em Aspects of probabilistic Littlewood-Paley theory,}  J. Funct. Anal.  {\bf 38}  (1980), no. 1,
25--60.

\bibitem{VasVol0} V. Vasyunin and  A. Volberg, {\em Burkholder's function via Monge-Amp\`ere equations,} {\bf (preprint)}. 

\bibitem{VasVol1} V. Vasyunin and  A. Volberg, {\em The Bellman functions for a certain two weight inequality: The case study,} Algebra i
Analiz {\bf 18} (2006).

\bibitem{VasVol2} V. Vasyunin and A. Volberg, {\em Bellman functions technique in harmonic
analysis}, (sashavolberg.wordpress.com). 

\bibitem{Vea1} I. E. Verbitsky, {\em An estimate of the norm of a function in a Hardy space in terms
of the norms of its real and imaginary parts,} Mat. Issled. {\bf 54} (1980), 16--20. 
(Russian). English transl.: Amer. Math. Soe. Transl. {\bf 124} (1984), 11-15.


\bibitem{Vol1} A. Volberg, {\em Bellman approach to some problems in harmonic analysis,}  in S\'eminaire
aux \'equations d\'erives partielles, {\bf 20},  Ecole Polyt\'echnique, Palaiseau,  (2002), 1-14.

\bibitem{Wan2} G. Wang, {\em Sharp inequalities for the conditional square function of a martingale,}  Ann. 
Probab. {\bf 19} (1991), 1679--1688. 

\bibitem{Wan3} G. Wang, {\em Sharp maximal inequalities for conditionally symmetric martingales and Brownian motion,} Proc. Amer. Math. Soc. {\bf 112} (1991), no. 2, 579--586

\bibitem{Wan4} G. Wang, {\em Sharp square-function inequalities for conditionally symmetric martingales,} Trans. Amer. Math. Soc. {\bf 328} (1991), 393--419.


\bibitem{Wan1} G. Wang, \emph{Differential subordination and strong differential subordination for
continuous-time  martingales and related sharp inequalities,} Ann. Probab. {\bf 23} (1995), 522--551.



\end{thebibliography}
\end{document}